\newcommand{\Aut}{\mathop{\mathrm{Aut}}\nolimits}

\newcommand{\End}{\mathop{\mathrm{End}}\nolimits}
\newcommand{\GL}{\mathop{\mathrm{GL}}\nolimits}
\newcommand{\Hom}{\mathop{\mathrm{Hom}}\nolimits}
\newcommand{\indlim}{\mathop{\underrightarrow{\lim}}\limits}
\newcommand{\Ka}{K_{\mathrm{a}}}
\newcommand{\Kb}{K_{\mathrm{b}}}
\newcommand{\Kc}{K_{\mathrm{c}}}
\newcommand{\Kd}{K_{\mathrm{d}}}

\newcommand{\mapsfrom}{\leftarrow\!\shortmid}
\newcommand{\Sy}{\mathop{\mathfrak{S}}\nolimits}
\newcommand{\Sm}{\mathop{\mathrm{Sm}}\nolimits}
\newcommand{\St}{\mathop{\mathrm{Stab}}\nolimits}
\newcommand{\wK}{\widetilde{K}}
\documentclass[11pt,twoside]{amsart} 
\usepackage{amssymb,hyperref}
\advance \topmargin by -\headheight
\advance \topmargin by -\headsep
\evensidemargin \oddsidemargin
\newtheorem{theorem}{Theorem}[section] 
\newtheorem{conjecture}[theorem]{Conjecture}
\newtheorem{corollary}[theorem]{Corollary}
\newtheorem{lemma}[theorem]{Lemma}
\newtheorem{proposition}[theorem]{Proposition}
\theoremstyle{remark} 
\newtheorem{example}[theorem]{Example}
\theoremstyle{definition} 
\newtheorem{definition}[theorem]{Definition}
\newtheorem{notation}[theorem]{Notation}
\theoremstyle{remark} 
\newtheorem{remark}[theorem]{Remark}
\newtheorem{remarks}[theorem]{Remarks}
\begin{document} 
\title{Partial fraction decompositions, and semilinear representations of infinite symmetric groups} 
\author{M.Rovinsky} 
\dedicatory{To Fedor Bogomolov with admiration} 
\address{National Research University Higher School of Economics, 
AG Laboratory, HSE, 6 Usacheva str., Moscow, Russia, 119048}
\email{marat@mccme.ru} 
\date{}

\begin{abstract} 
Let $F|k$ be a non-trivial regular field extension, $\Psi$ be an infinite (discrete) set, $\Sy_{\Psi}$ 
be the group of all permutations of $\Psi$ endowed with the compact-open (a.k.a. finite or Krull) 
topology, $L$ be the fraction field of the tensor product over $k$ of the copies 
of $F$ labeled by $\Psi$. The field $L$ is endowed with the natural $\Sy_{\Psi}$-action. For each 
$\Sy_{\Psi}$-invariant subfield $K$ of $L$, let $\Sm_K(\Sy_{\Psi})$ denote the category of {\sl smooth} 
(i.e. with open stabilizers) $K$-{\sl semilinear} representations of $\Sy_{\Psi}$, cf. \S\ref{Goals}. 

The categories $\Sm_K(\Sy_{\Psi})$ (especially, their simple and injective objects) 
are the principal object of the present study, though only in some particular cases. 

It is known (\cite[Theorem 6.1]{Lueroth}) that the indecomposable injective objects 
of the category $\Sm_L(\Sy_{\Psi})$ are the $L$-exterior powers $L\langle\binom{\Psi}{s}\rangle$ 
($s\ge 0$) of the $L$-vector space with the basis $\Psi$, while $L$ is the only simple object. 
It turns out that the objects $K\langle\binom{\Psi}{s}\rangle$ are injective quite generally. 

Let $K=L^H\subset L$ be the fixed field of an algebraic automorphism $k$-group $H$ of $F|k$ acting on 
$L$ diagonally. The question is: what could be a relation (a kind of the Schur--Weyl duality) between 
representations of $H$ and the indecomposable injectives or simple objects of $\Sm_K(\Sy_{\Psi})$? 

In this paper we consider several examples, where $H$ is either 
a subgroup of $\mathrm{PGL}_{2,k}=\Aut(k(t)|k)$ or a torus. 
In these examples: a) a natural bijection between the {\sl finite-dimensional} simple objects 
of $\Sm_K(\Sy_{\Psi})$ and the irreducible rational representations of $H$ is constructed; b) 
for $H\neq\mathrm{PGL}_{2,k}$, the indecomposable injectives and the simple objects 
of $\Sm_K(\Sy_{\Psi})$ are described completely. 

For $H=\mathrm{PGL}_{2,k}$, an infinite list of infinite-dimensional simple objects is produced, which 
is shown to be complete if $F\neq k$; a system of indecomposable injective {\sl generators} is described. 
\end{abstract}

\maketitle 
\tableofcontents 
\section{Introduction} 
\subsection{Motivation and goals} \label{Goals} Let $G$ be a group of permutations of 
a set $\Psi$. We consider $G$ as a totally disconnected group with a prebase of 
topology formed by the left (or right) translates of the stabilizers of the elements of $\Psi$. 

One is interested in describing continuous representations of $G$ in discrete vector spaces 
(i.e., with open stabilizers), called {\sl smooth} in what follows. This problem arises naturally in many 
situations. E.g., some birational `motivic' algebro-geometric questions over a characteristic 0 algebraically 
closed field $k$ are related to describing certain irreducible smooth representations of the automorphism group 
$G$ of an algebraically closed field extension $C|k$ of infinite transcendence degree. There are some 
reasons to expect that the numerous `interesting' representations can be embedded into the single 
$C$-vector space of differential forms $\Omega^{\bullet}_{C|k}=\bigoplus_{q\ge 0}\Omega^q_{C|k}$. 
Note that the $G$-action on $\Omega^q_{C|k}$ is $C$-{\sl semilinear}. This leads to the study of 
the category $\Sm_C(G)$ of smooth $C$-semilinear representations of $G$, that can be considered 
as `quasicoherent sheaves' in the dominant topology (in the sense of \cite{glatt}). 

Classical examples, where smooth semilinear representations show up, are (i) `Hilbert's Theorem 90' 
(\cite[Satz 1]{Speiser}) stating that the smooth semilinear {\sl Galois} representations are `trivial', 
and (ii) the theory of ($q$-)difference equations. 
The first case where $G$ is not locally (pre)compact (namely, $G$ is the symmetric group $\Sy_{\Psi}$ 
of all permutations of an infinite set $\Psi$) is considered in \cite{H90,NagpalSnowden}. 

In analogy with Hilbert's Theorem 90, one may ask whether for a given permutation group $G$ and 
a field $K$ endowed with a smooth $G$-action there exists a smooth $G$-field extension $B|K$ such 
that all smooth $B$-semilinear representations of $G$ are `as trivial as they can'.\footnote{Though 
a more direct analogue of the Tate--Fontaine's condition would be weaker: after coefficient 
extension to $B$ the smooth $K$-semilinear representations of $G$ become `as trivial as they can'.} 

The purpose of this paper is (i) to describe the category $\Sm_K(\Sy_{\Psi})$ of all smooth $K$-semilinear 
representations of the symmetric group $\Sy_{\Psi}$ over some $\Sy_{\Psi}$-fields $K$ and 
(ii) to construct, for an arbitrary field $K$ endowed with a smooth $\Sy_{\Psi}$-action, 
a smooth $\Sy_{\Psi}$-field extension $B|K$, called a `weak period' field extension, such that $B$ is 
a cogenerator of the category $\Sm_B(\Sy_{\Psi})$. 

Compared to the linear representations over fields of various characteristics, the series of examples of 
categories $\Sm_{K_?}(\Sy_{\Psi})$ considered here (where $K_?$ is an $\Sy_{\Psi}$-field, depending on a 
label $?\in\{\mathrm{a,b,c}\}$ and a field extension $F|k$, see below) have the feature that the {\sl structure} 
of $\Sm_K(\Sy_{\Psi})$ (e.g. its Gabriel spectrum) depends rather on the label $?$ of $K$ than on $F|k$ (with 
a minor modification in a few particular cases). 

Though the case of the group $\Sy_{\Psi}$ is a toy model of this kind of problems, a faithful and exact 
functor relates the above `algebro-geometric' category $\Sm_C(G)$ of `quasicoherent sheaves' to the 
category $\Sm_{k(\Psi)}(\Sy_{\Psi})$, where $k(\Psi)=k(x~|~x\in\Psi)$ is the field of rational functions 
over $k$ in variables labeled by a transcendence basis $\Psi$ of $C$ over $k$ (and endowed with the 
natural $\Sy_{\Psi}$-action). In particular, at least some of `interesting' objects of $\Sm_C(G)$ are 
transformed to injective objects of $\Sm_{k(\Psi)}(\Sy_{\Psi})$, see \cite[Propositions 4.6 and 4.7]{H90}.

\subsection{Basic notation} \label{BasicNota} 
For an abelian group $A$ and a set $S$, we denote by $A\langle S\rangle$ the abelian 
group that is the direct sum of copies of $A$ indexed by $S$, i.e., the elements of 
$A\langle S\rangle$ are the finite formal sums $\sum_{i=1}^Na_i[s_i]$ for all integer 
$N\ge 0$, $a_i\in A$, $s_i\in S$, with addition defined in the obvious way. 

For a group $G$ and a $G$-set $S$, denote by $S^G$ the subset of $S$ fixed by $G$. For 
a unital associative ring $A$ endowed with a $G$-action (by ring automorphisms), we consider 
$A\langle G\rangle$ as a unital associative ring, so $A\langle S\rangle$ becomes a left 
$A\langle G\rangle$-module with the (diagonal) left $G$-action both on $A$ and $S$: $g(a[s])=a^g[gs]$ 
for all $g\in G$, $a\in A$, $s\in S$, where we write $a^g$ for the result of applying of $g$ to $a$. 

If $S$ is a pointed $G$-orbit, i.e. $S=G/H$ for a subgroup $H\subseteq G$, we may also consider 
$A\langle S\rangle$ as an $A\langle G\rangle$-$A^H\langle N_G(H)/H\rangle$-bimodule: 
$(a[g])(b[g'\pmod H])(c[h']):=ab^gc^{gg'}[gg'h']$ for all $g,g'\in G$, $h'\in N_G(H)/H$, 
$a,b\in A$ and $c\in A^H$. 

The left $A\langle G\rangle$-modules are also called $A$-{\sl semilinear representations} of $G$. 

For each set $S$ and an integer $s\ge 0$, denote by $\binom{S}{s}$ the set of all subsets 
of $S$ of cardinality $s$. Denote by $\Sy_S$ the group of all permutations of the set $S$. 

\subsubsection{Notation and terminology for fields} \label{NotaF} 
For a field $L$ and a set $S$, $L(S)$ denotes the purely transcendental field extension 
of $L$ with $S$ as a transcendence base (the fraction field of the symmetric $L$-algebra 
of $L\langle S\rangle$). If a group acts on $L$ and on $S$ then it acts on $L(S)$ as well. 

For a subfield $K$ of $L$, $\Aut(L|K)$ and $\mathrm{tr.deg}(L|K)$ denote the group of 
field automorphisms of $L$ identical on $K$ and the transcendence degree, respectively. 

For a group $G$, a ({\sl non-trivial}) $G$-{\sl field} is a field endowed with a (non-trivial) $G$-action. 

The labels $\mathrm{a,b,c,d}$ and the $\Sy_{\Psi}$-fields $K_?$ for a label $?\in\{\mathrm{a,b,c,d}\}$ 
are defined in Notation~\ref{the_fields_K?}. 

For each field $k$ and each collection $\{A_i\}_{i\in S}$ of objects in the category of unital associative 
commutative $k$-algebras, indexed by a set $S$, denote by $\bigotimes_{k,~i\in S}A_i$ the coproduct of this 
collection. 

For each regular field extension $F|k$ and each set $S$, denote by $F_S=F_{k,S}$ the fraction 
field of the (integral) $k$-algebra $\bigotimes_{k,~i\in S}F$. In particular, $F_S=k$ if 
either $F=k$ or $S=\varnothing$; $F(x)_S$ (here $F(x)$ is the field of rational functions 
over $F$ in a variable $x$) is nothing but the field $F_S(S)$ of rational 
functions over $F_S$ in the variables enumerated by the set $S$. 

Any permutation group of $S$ acts smoothly on $F_S$ by permuting the tensor factors of $\bigotimes_{k,~i\in S}F$. 

For each field extension $L|K$ and integer $s\ge 0$, $\Omega^s_{L|K}:=\bigwedge^s_L\Omega^1_{L|K}$ denotes 
the $L$-vector space of differential $s$-forms on $L$ over $K$; $L|K$ is {\sl non-trivial} if $L\neq K$. 

\subsubsection{Notation for categories of smooth semilinear representations and their Picard groups} 
For an associative ring $A$ and a permutation group $G$ (cf. Definition~\ref{permutation_group}) acting 
on $A$, denote by $\Sm_A(G)$ the {\sl category of smooth left $A$-semilinear representations of} $G$. 
E.g., $\Sm_A(1)$ is the category of left $A$-modules. 

For an object $X$ of a Grothendieck category, $E(X)$ denotes an injective hull of $X$. 
When dealing with common objects of several embedded categories, one has to modify the notation $E(-)$. 
In particular, an injective hull in the category $\Sm_{\Kc}(\Sy_{\Psi})$ is denoted $E_{\mathrm{c}}(-)$. 

\begin{remark} If $A$ is a division ring and there exists a non-zero object of $\Sm_A(G)$ then the $G$-action 
on $A$ is smooth. [Indeed, for any $A$-semilinear representation $V$ of $G$, any non-torsion $v\in V$, any 
non-zero $a\in A$, one has $\St_{av}\cap\St_v=\St_a\cap\St_v$. If $V$ is smooth then the group on the left 
hand side is open, so the stabilizer $\St_a$ is open, i.e. the $G$-action on $A$ is smooth.] \end{remark} 

If $K$ is a smooth $G$-field then $\Sm_K(G)$ is a tensor $k$-linear Grothendieck category, where 
$k:=K^G$. The smooth finite-dimensional $K$-semilinear representations of $G$ form a rigid tannakian 
category ($K$ is the unit, $\End_{K\langle G\rangle}(K)=k$; the duals are $V^{\vee}:=\Hom_K(V,K)$). 

Denote by $\mathrm{Pic}_K(G)$ the group of isomorphism classes of invertible (i.e. one-dimensional 
over $K$) objects of $\Sm_K(G)$, called the {\sl Picard group} of $\Sm_K(G)$, and set 
$\mathrm{Pic}_K:=\mathrm{Pic}_K(\Sy_{\Psi})$. 

\subsection{Results} We are looking for an analogue of Hilbert's Theorem 90 
(cf. Proposition \ref{Satz1}) for the symmetric group $\Sy_{\Psi}$ of an infinite set $\Psi$, i.e. 
we study the category $\Sm_K(\Sy_{\Psi})$ of smooth $K$-semilinear representations of $\Sy_{\Psi}$, 
where $K$ is a non-trivial smooth $\Sy_{\Psi}$-field. 

By \cite[Theorem 3.18]{H90}, the category $\Sm_K(\Sy_{\Psi})$ is locally noetherian. 

\subsubsection{General results} \label{General_results} \begin{itemize} 
\item The finitely generated objects in $\Sm_K(\Sy_{\Psi})$ {\sl split `locally'} (i.e. under restriction 
to sufficiently small open subgroups of $\Sy_{\Psi}$) into finite direct sums of $K$-semilinear 
representations induced by trivial representations of open subgroups (Proposition~\ref{local-structure}), 
which is somewhat similar to the description of the `local' structure of smooth representations of 
$\mathrm{GL}_n$ over local division rings in \cite{Howe} and \cite{HenniartVigneras}. 

This phenomenon allows to describe the Grothendieck $K_0$-ring of the full subcategory of 
compact objects in $\Sm_K(\Sy_{\Psi})$ for some $K$ (Theorem \ref{lambda-homomorphism}) 
generalizing \cite[Theorem 4.18]{NagpalSnowden}. 

It turns out that the morphisms of the latter category {\sl are `locally' split} (Theorem~\ref{Local-injectivity}). 
This is related to a canonical `{\sl level}' filtration on the objects, introduced in \S\ref{level_filtration}. 
\item It follows from the noetherian property of $\Sm_K(\Sy_{\Psi})$ that any $\Sy_{\Psi}$-field $K$ admits 
a {\sl `weak period' extension}, i.e. of a smooth $\Sy_{\Psi}$-field extension $\wK|K$ such that $\wK$ 
is a cogenerator of $\Sm_{\wK}(\Sy_{\Psi})$. Moreover, Proposition~\ref{weak_period} provides a functorial 
construction of $\wK$ for any {\sl non-trivial} $\Sy_{\Psi}$-field $K$. However, $\wK$ constructed there is 
by no means optimal, in particular, $\wK^{\Sy_{\Psi}}\neq K^{\Sy_{\Psi}}$. 

If $K$ is trivial then any `weak period' extension of $K$ admits an $\Sy_{\Psi}$-subfield isomorphic to $K(\Psi)$ 
(Example~\ref{minimal_period}) that is a `period' extension of $K$, i.e. $K(\Psi)^{\Sy_{\Psi}}=K$. \end{itemize} 

\begin{definition} The {\sl Gabriel spectrum} of a Grothendieck category $\mathcal C$ is the topological 
space whose points are isomorphism classes of indecomposable injectives. For each object $X$ of 
$\mathcal C$ denote by $[X]$ the set of points $E$ with $\Hom(X,E)=0$. Then base of opens consists of 
sets of the form $[X]$ as $X$ ranges over the compact (finitely presentable) objects. \end{definition}

\begin{itemize} \item Let $K$ be a field endowed with a non-trivial smooth $\Sy_{\Psi}$-action, 
$\mathrm{Spec}_K$ be the Gabriel spectrum of the category $\Sm_K(\Sy_{\Psi})$. 
It is explained in Remark~\ref{points_P} that, for all integer $s\ge 0$, the injective hulls of the objects 
$K\langle\binom{\Psi}{s}\rangle$ of $\Sm_K(\Sy_{\Psi})$ represent pairwise 
distinct points $P_s=P_s^{(K)}$ of $\mathrm{Spec}_K$. Any set containing $P_s$ for infinitely many $s\ge 0$ 
is dense in $\mathrm{Spec}_K$ (Lemma~\ref{closed-pts_bd} (\ref{density_of_P_n})). 
The closure of $P_s$ in the set $\{P_t~|~t\ge 0\}$ is $\{P_t~|~0\le t\le s\}$ if $s\ge s_0(K)$, where 
$[K^{\Sy_{\Psi|J}}:K^{\Sy_{\Psi}}]=\infty$ for any $J$ of order $\ge s_0(K)$. 
\end{itemize} 

\subsubsection{The fields $F_{\Psi}$ and their $\Sy_{\Psi}$-invariant subfields; the fields $K_?$} 
\label{subfields} Let $\Psi$ be an infinite set, and $F|k$ be a regular field extension. The most 
straightforward example of a smooth $\Sy_{\Psi}$-field is the fraction field $F_{\Psi}$ 
of the `$\Psi$-th' tensor power of $F$ over $k$ (defined in \S\ref{NotaF}).

One can get more examples of smooth 
$\Sy_{\Psi}$-fields as $\Sy_{\Psi}$-invariant subfields of $F_{\Psi}$ or as its extensions. 

For any group $H$ of automorphisms of $F|k$, its diagonal action on $F_{\Psi}$ commutes with the 
$\Sy_{\Psi}$-action, so for any field extension $L|k$ in $F$, the subfield 
$L_{\Psi}^H$ of $F_{\Psi}$ (fixed by $H$ in $L_{\Psi}$) is $\Sy_{\Psi}$-invariant. 

This construction can be generalized by considering (pro-){\sl algebraic} $k$-groups $H$ and understanding 
the fixed subfields $L_{\Psi}^H$ accordingly.

In the case of characteristic zero extensions $F|k$ of transcendence degree 1, all possible 
$\Sy_{\Psi}$-invariant field extensions of $k$ in $F_{\Psi}$ are precisely of type $L_{\Psi}^H$ (this is 
\cite[Theorem 3.4, Propositions 3.6 and 3.8]{Lueroth}). 
Here is an {\sl explicit part} of their complete list: 
\begin{theorem} \label{invar-subf} Let $F|k$ be 
a transcendence degree $1$ regular field extension of characteristic $0$, 
$K\neq k$ be an $\Sy_{\Psi}$-invariant field extension of $k$ in $F_{\Psi}$. 

Then the transcendence degree $d$ of $F_{\Psi}$ over $K$ is $\le 3$. 

For each $x\in\Psi$ and $f\in F$, denote by $f(x)$ the image of $f$ under the field embedding 
$F\hookrightarrow F_{\Psi}$ identifying $F$ with the $x$-th tensor factor in $\bigotimes_{k,\Psi}F$. 

If $d=3$ then there exists a unique $R\in\mathrm{PGL}_2(k)\backslash(F\smallsetminus k)$ 
such that \[K=k\left(\frac{(R(w)-R(x))(R(y)-R(z))}{(R(w)-R(y))(R(x)-R(z))}~|~w,x,y,z\in\Psi\right)
\cong k(\Psi)^{\mathrm{PGL}_{2,k}}.\] 

If $d=2$ then there exists a unique $R\in\mathbb P_k(F/k)$ 
such that \[K=k\left(\frac{R(u)-R(w)}{R(u)-R(v)}~|~u,v,w\in\Psi\right)\cong 
k(\Psi)^{\mathbb G_{\mathrm{m},k}\ltimes\mathbb G_{\mathrm{a},k}}.\] 

If $d=1$ and $K$ is {\bf algebraically closed} in $F_{\Psi}$ then there is a system 
$(\pi_{ij}\colon W_i\to W_j)_{ij}$ of isogenies between torsors $W_i$ over geometrically irreducible 
one-dimensional algebraic $k$-groups $E_i$ endowed with a compatible system of $k$-field 
embeddings $\sigma_i:k(W_i)\hookrightarrow F\ ($i.e. $\sigma_i\pi_{ij}^*=\sigma_j$ for all $i,j)$ 
such that $K=\bigcup_i(\overline{k}(W_i)_{\overline{k},\Psi}^{E_i(\overline{k})})^{\mathrm{Gal}(\overline{k}|k)}
\subset\bigcup_ik(W_i)_{\Psi}\subseteq F_{\Psi}$, where $E_i(\overline{k})$ acts on 
$\overline{k}(W_i)_{\Psi}$ diagonally.\qed \end{theorem}

\begin{remark} \S\ref{reps-over-fixed-subfields} provides examples of proper invariant subfields $K$ of 
$F_{\Psi}$ which are not algebraically closed in $F_{\Psi}$, and a description of $\Sm_K(\Sy_{\Psi})$ 
for such $K$'s. \end{remark} 
\begin{proposition}[\cite{Lueroth}, Proposition 3.6] \label{sub-algebraic} Let $K$ be an 
$\Sy_{\Psi}$-invariant field extension of $k$ in $F_{\Psi}$ over which $F_{\Psi}$ is algebraic. Then there 
is an intermediate subfield $F'$ in $F|k$ over which $F$ is algebraic such that $F'_{\Psi}\subseteq K$. 

If the characteristic of $k$ is $0$ then $K=L_{\Psi}^{\Gamma}$, 
where $L$ is an intermediate subfield in $F|k$ and $\Gamma$ is a profinite algebraic 
$k$-group of automorphisms of $L|k$ acting diagonally on $L_{\Psi}$. More explicitly, 
$K=\left((L\otimes_k\overline{k})_{\Psi}^{\Gamma(\overline{k})}\right)^{\mathrm{Gal}(\overline{k}|k)}$, 
where $\overline{k}$ is an algebraic closure of $k$. \qed \end{proposition} 

In arbitrary characteristic, if $F|k$ is a transcendence degree 1 regular field extension and 
$K\neq k$ is an $\Sy_{\Psi}$-invariant field extension of $k$ in $F_{\Psi}$ then 
$d:=\mathrm{tr.deg}(F_{\Psi}|K)$ is finite. 
Some of such $K$ can be constructed as $L_{\Psi}^H$ for (closed) subgroups $H$ of automorphisms of 
intermediate subfields $L$ in $F|k$ (e.g. $H$ is a subgroup of $\mathrm{PGL}_{2,k}$ if $L\cong k(X)$). 
Moreover, Theorem~\ref{invar-subf} shows that (i) $d\le 3$, and 
(ii) in characteristic 0, `essentially', each $K$ algebraically closed in $F_{\Psi}$ is obtained by 
the above construction (and its isomorphism class depends only on $d$ if $d\neq 1$, while the case 
$d=1$ admits more options).

If $F=L(X)$ for a field extension $L|k$ and $H$ is a $k$-subgroup of $\mathrm{PGL}_{2,k}$, 
the $\mathrm{PGL}_{2,k}$-action on $F_{\Psi}$ is understood as an $L_{\Psi}$-algebra homomorphism 
from $F_{\Psi}$ to a localization of $F_{\Psi}\otimes_k\mathcal O(H)$: \begin{gather} 
\label{Borel-action} \mbox{in the case of }H=\mathbb G_{\mathrm{a},k}\rtimes\mathbb G_{\mathrm{m},k}:\ 
u\mapsto u\otimes A+1\otimes B\in k(\Psi)\otimes_kk[A,A^{-1},B]\mbox{ for all }u\in\Psi. \end{gather} 

One may wonder, whether there are other $\Sy_{\Psi}$-invariant subfields $K$ of $F_{\Psi}$. 

In the case of the $\Sy_{\Psi}$-field $F_{\Psi}$, for any regular field extension $F|k$, 
one may expect that 
\begin{conjecture} \label{level-l-fields} 
Any $\Sy_{\Psi}$-invariant intermediate field $K$ in $F_{\Psi}|k$ is contained, in fact, in $L_{\Psi}$ 
for a field extension $L|k$ in $F$ such that for any $L'|k$ in $L$ with $\mathrm{tr.deg}(L'|k)<\infty$ 
there is a field extension $L''|L'$ in $L$ with $\mathrm{tr.deg}(L''_{\Psi}|K\cap L''_{\Psi})<\infty;$ 
moreover, the algebraic closure of $L$ in $F$ is determined uniquely. \end{conjecture}

This conjecture is proved in \cite[Theorem 2.4]{Lueroth} when $k$ is of characteristic $0$.

In arbitrary characteristic, if $F|k$ is a transcendence degree 1 regular field extension and 
$K\neq k$ is an $\Sy_{\Psi}$-invariant field extension of $k$ in $F_{\Psi}$ then 
$d:=\mathrm{tr.deg}(F_{\Psi}|K)$ is finite. 
Some of such $K$ can be constructed as $L_{\Psi}^H$ for (closed) subgroups $H$ of automorphisms of 
intermediate subfields $L$ in $F|k$ (e.g. $H$ is a subgroup of $\mathrm{PGL}_{2,k}$ if $L\cong k(X)$). 
Moreover, Theorem~\ref{invar-subf} shows that in characteristic 0, (i) $d\le 3$, 
and (ii) `essentially', each $K$ algebraically closed in $F_{\Psi}$ is obtained by the above construction 
(and its isomorphism class depends only on $d$ if $d\neq 1$, while the case $d=1$ admits more options). 
In the case of $K$ algebraically non-closed in $F_{\Psi}$ with $d=0$, there exist smooth irreducible 
semilinear representations of $\Sy_{\Psi}$ of finite dimensions $>1$ (namely, $2,\ 3,\ 4,\ 5$, and $q$ if 
$k$ contains $\mathbb F_q$), cf. Examples~\ref{irred_2-dim_ex}, \ref{irred_A_4_ex}, \ref{irred_A_5_ex}. 
In particular, though in Theorem~\ref{spectrum-triple} all points of a given level form a 
$\mathrm{Pic}_{K_?}$-orbit, this is not the case if $K_?$ is replaced by an arbitrary $K$, even in level 0. 

It is shown in Proposition~\ref{sub-algebraic} that if $F|k$ is a regular field extension, and $K|k$ 
is an $\Sy_{\Psi}$-invariant field extension in $F_{\Psi}$ over which $F_{\Psi}$ is algebraic, then 
$F'_{\Psi}\subseteq K$ for an intermediate subfield $F'$ in $F|k$ over which $F$ is algebraic.

\begin{notation}[The fields $K_?$] \label{the_fields_K?} In what follows, we deal with fields denoted 
$K_?$ for the label $?\in\{\mathrm{a,b,c,d}\}$ depending on $\Psi$ and $F|k$ as above, while $\Ka$ depends 
moreover on a finite set $S$ and on a subgroup $\Gamma$ of the free abelian group $\Xi$ with basis $S$. 
We assume that $S$ is non-empty if $F=k$. 

There are inclusions $F_{\Psi}\subset\Kd\subset\Kc\subset\Kb\subset F(X)_{\Psi}=F_{\Psi}(\Psi)$ 
and $\Ka\subseteq F_{\Psi}(\Psi\times S)=F(S)_{\Psi}$. 

Here $F_{\Psi}(\Psi\times S)$ is the purely transcendental $\Sy_{\Psi}$-field extension of 
$F_{\Psi}$ with a transcendence basis consisting of the variables labeled by the $\Sy_{\Psi}$-set 
$\Psi\times S$, where $u_s$ denotes the variable corresponding to $(u,s)\in\Psi\times S$. 
For all $u\in\Psi$ and $\gamma=\sum_sm_s[j_s]\in\Xi$, set $u^{\gamma}:=\prod_su_{j_s}^{m_s}$. 

For each label $?\in\{\mathrm{a,b,c,d}\}$, we define an $\Sy_{\Psi}$-field extension $K_?$ of $F_{\Psi}$: 
$\Ka\subseteq F_{\Psi}(\Psi\times S)$ and $K_?\subseteq F_{\Psi}(\Psi)$ for $?\in\{\mathrm{b,c,d}\}$. 
Thus, they all depend on $\Psi$ and $F|k$, and moreover, $\Ka$ depends also on $S$ and a subgroup 
$\Gamma\subseteq\Xi:=\mathbb Z\langle S\rangle$. 
The field $\Ka$ is the fixed subfield of the algebraic group $\Hom(\Xi/\Gamma,\mathbb G_{\mathrm{m},k})$ 
acting on the field extension $F_{\Psi}(\Psi\times S)|F_{\Psi}$, while the fields $K_?$ for 
$?\in\{\mathrm{b,c,d}\}$ are the fixed subfields of, respectively, the algebraic subgroups 
$\mathbb G_{\mathrm{a},k}$, $\mathbb G_{\mathrm{a},k}\rtimes\mathbb G_{\mathrm{m},k}$, $\mathrm{PGL}_{2,k}$ 
of $\mathrm{PGL}_{2,k}$ acting on the field extension $F_{\Psi}(\Psi)|F_{\Psi}$. 

More precisely and explicitly, \[\Ka=K_{\Psi,S,\Gamma}^{(F|k)}:=
F_{\Psi}\left(u^{\gamma},\frac{u_s}{v_s}~|~\gamma\in\Gamma,~s\in S,~u,v\in\Psi\right) 
=F_{\Psi}\left(x^{\gamma},\frac{u_s}{x_s}~|~\gamma\in\Gamma,~s\in S,~u\in\Psi\right),\] 
(so $K_{\Psi,S,\Xi}^{(F|k)}=F(S)_{\Psi};$ $K_{\Psi,\{\ast\},0}^{(F|k)}=F_{\Psi}(u/v~|~u,v\in\Psi);$ 
$K_{\Psi,S\sqcup S',\Gamma\oplus\mathbb Z\langle S'\rangle}^{(F|k)}=K_{\Psi,S,\Gamma}^{(F(S')|k)});$ 

\[\Kb:=F_{\Psi}(u-v~|~u,v\in\Psi)\subset F_{\Psi}(\Psi);\quad
\Kc:=F_{\Psi}\left(\frac{u-v}{u-w}~|~u,v,w\in\Psi,~\#\{u,v,w\}=3\right)\subset\Kb;\] 

\[\Kd:=F_{\Psi}\left(\frac{(t-u)(v-w)}{(u-v)(w-t)}~|~t,u,v,w\in\Psi,~\#\{t,u,v,w\}=4\right)\subset\Kc
\quad\mbox{is the ``cross-ratio'' field}.\] \end{notation} 

\subsubsection{Smooth $\Sy_{\Psi}$-field extensions of $F_{\Psi}$} 
Given a smooth $G$-field, it is natural to study not just its $G$-invariant subfields but its smooth $G$-field 
extensions as well, and to compare semilinear representations of $G$ over such $G$-fields. 

In the setting of Notation~\ref{the_fields_K?}, any period $\Sy_{\Psi}$-extension of $\Ka$ 
contains a unique $\Sy_{\Psi}$-extension of $\Ka$ isomorphic to $F_{\Psi}(\Psi\times S)$, i.e. 
$F_{\Psi}(\Psi\times S)$ is the smallest period $\Sy_{\Psi}$-extension of $\Ka$. 
Similarly, $F_{\Psi}(\Psi)$ is the smallest period $\Sy_{\Psi}$-extension of $\Kd$ (as well as of $\Kb$ 
and of $\Kc$).

Concerning the smooth $\Sy_{\Psi}$-field extensions $L|K$ with $L^{\Sy_{\Psi}}=K^{\Sy_{\Psi}}$, 
Example~\ref{finite_smooth_G-ext} lists some of $K$ with no non-trivial {\sl finite} $L|K$, while 
\cite[Proposition~5.8]{Lueroth} shows that there are no 
non-trivial {\sl `isotrivial' finitely generated} $L|K$ with $K=F_{\Psi}$. 
In \cite[\S5.1]{Lueroth} some conditions on $L|K$ forcing the (iso)triviality are listed.

\subsubsection{$\Sm_K(\Sy_{\Psi})$ for certain subfields $K\subseteq F_{\Psi}$} \label{actions} 

Keeping notation and having a supply of $\Sy_{\Psi}$-fields $K=F_{\Psi}^H$ from \S\ref{subfields} 
(Notation~\ref{the_fields_K?}), one may ask for a description of the categories $\Sm_K(\Sy_{\Psi})$.

\begin{theorem}[Picard groups] \label{Pic_K?} 
Fix distinct elements $x,y\in\Psi$. Then 
\begin{enumerate} \item the invertible objects $x^{\lambda}\Ka\subseteq F_{\Psi}(\Psi\times S)$ 
of $\Sm_{\Ka}(\Sy_{\Psi})$ are injective for all $\lambda\in\Xi;$ 
\item $\mathrm{Pic}_{\Ka}\cong\Xi/\Gamma$ and its elements $P_{0,\overline{\lambda}}$ for all 
$\overline{\lambda}\in\Xi/\Gamma$ are presented by $x^{\lambda}\Ka$ for all $\lambda\in\Xi;$ 
\item $\mathrm{Pic}_{\Kb}=\mathrm{Pic}_{\Kd}=0;$ and 
\item $\mathrm{Pic}_{\Kc}\cong\mathbb Z$ is generated by the class of $\Omega^1_{\Kc|\Kd}\cong
(x-y)\Kc\subset\Kb$. \end{enumerate} \end{theorem}

To certain extent, this can be interpreted as follows: {\it for an algebraic subgroup $H$ of 
$\mathrm{PGL}_{2,k}$, the group $\mathrm{Pic}_{F_{\Psi}^H}$ coincides with the group of characters of} $H$.

\begin{theorem}[Spectra I: essentially, only in the $\mathrm{a,b,c}$ cases] \label{spectrum-triple-rough} 
Let $?\in\{\mathrm{a,b,c}\}$. Then 
\begin{itemize} \item the points of $\mathrm{Spec}_{K_?}$ are of finite level$;$ any point of 
$\mathrm{Spec}_{K_?}$ of level $n\ge 0$ is contained in the $\mathrm{Pic}_{K_?}$-orbit of the 
class $P_n^{(K_?)}$ of an injective hull of the object $K_?\langle\binom{\Psi}{n}\rangle$ of 
$\Sm_{K_?}(\Sy_{\Psi});$ 
\item the $\mathrm{Pic}_{K_?}$-orbit of $P_n^{(K_?)}$ consists of 
the $($single$)$ class of $K_?\langle\binom{\Psi}{n}\rangle$ if $n>2$; 
\item the objects $\Kd\langle\{\{1,2,3\}\hookrightarrow\Psi\}\rangle$ and $\Kd\langle\binom{\Psi}{s}\rangle$ 
of $\Sm_{\Kd}(\Sy_{\Psi})$ for $s\ge 4$ are injective. \end{itemize} \end{theorem} 
In particular, for each $?\in\{\mathrm{a,b,c,d}\}$, any infinite collection of objects 
of type $K_?\langle\binom{\Psi}{N}\rangle$ for $N\ge 4$ forms a system of {\bf injective} generators 
of the category $\Sm_{K_?}(\Sy_{\Psi})$. 

It turns out that, for many $F|k$ and $H$, the functor from the category of algebraic representation 
of $H$ to $\Sm_{F_{\Psi}^H}(\Sy_{\Psi})$, given by $V\mapsto\Hom_{k\langle H\rangle}(V,F_{\Psi})$ 
(again, understood properly), induces a bijection between isomorphism classes of 
simple finite-dimensional objects of the source and of the target. 
\begin{theorem}[Simple objects, possibly with omissions when $F=k$ in the $\mathrm{d}$ case] 
\label{simple_objects_are_invertible_Pic} 
\begin{enumerate} \item The\\ simple objects of $\Sm_{K_?}(\Sy_{\Psi})$ are invertible for each 
$?\in\{\mathrm{a,b,c}\}$, but not for $?=\mathrm{d}$. 
\item \label{BW} {\rm (Theorem~\ref{non-triv-Simple-Kd} (\ref{finite-Simple-Kd}))} 
There is a natural bijection between the {\sl finite-dimensional} simple objects of 
$\Sm_{\Kd}(\Sy_{\Psi})$ and the irreducible rational representations of $\mathrm{PGL}_{2,k}$. 
\item {\rm (Theorem~\ref{non-triv-Simple-Kd} (\ref{infinite-Simple-Kd}))} The rational non-negative 
characters of $\mathbb G_m$ give rise to an infinite list of {\sl infinite-dimensional} simple objects of 
$\Sm_{\Kd}(\Sy_{\Psi})$, which is complete, at least if $F\neq k$. \end{enumerate} \end{theorem}

Theorem~\ref{simple_objects_are_invertible_Pic} (\ref{BW}) is an analogue of Borel--Weil theorem for 
$\mathbb P^1$. Namely, the valuations of $\Kc$ trivial on $\Kd$ can be considered as points of the 
projective line $\mathbb Y$ over $\Kd$ endowed with a natural $\Sy_{\Psi}$-action, while any 
finite-dimensional simple object is isomorphic to the socle of the module of global sections of 
some plurianticanonical sheaf on $\mathbb Y$. These objects are self-dual. Their dimensions are 
arbitrary odd positive integers if the characteristic $p$ of $k$ is 0, and are arbitrary products of 
positive integers $\le p$ if $p>0$ (which is a version of the Steinberg's tensor product theorem for 
$\mathrm{PGL}_2$).

\begin{theorem}[Spectra II] \label{spectrum-triple} 
Let $p\ge 0$ be the characteristic of $k$. Fix an element $x\in\Psi$. 
\begin{enumerate} \item The objects of $\Sm_{\Ka}(\Sy_{\Psi})$ of finite length are semisimple. 
For any integer $N\ge 1$, the indecomposable objects of $\Sm_{\Kb}(\Sy_{\Psi})$ of 
length $N$ are isomorphic, if $p=0$. 
\item The injection $E\colon\mathrm{Pic}_K\to\mathrm{Cl}_K$, $\mathcal L\mapsto E(\mathcal L)$, to 
the set of closed points of $\mathrm{Spec}_K$, commutes with the $\mathrm{Pic}_K$-action, 
it is bijective for $K=\Kb$ and for $K=\Ka$ if $F\neq k$ or $\Gamma\neq 0;$ 
\begin{itemize} \item $\mathrm{Cl}_{\Ka}\smallsetminus\mathrm{Pic}_{\Ka}=\{P_1^{(\Ka)}\}$ if 
$\Xi\neq\Gamma=0$ and $F=k$ $($in this case any infinite subset of $\mathrm{Cl}_{\Ka}$ is dense$)$, 
\item \label{unip} $\mathrm{Cl}_{\Kb}$ consists of the class $P_0$ of the $($non-noetherian, artinian 
if and only if $p=0)$ object $E(\Kb)=\Kb[x]\subset F_{\Psi}(\Psi)$, 
\item \label{Kc_P01} $\mathrm{Cl}_{\Kc}=\{P_{i,n}~|~n\in\mathbb Z,\ i\in\{0,1\}\}$ is the union of the free 
$\mathrm{Pic}_{\Kc}$-orbits 
of $P_0$ and of $P_1$, where $P_0$ is the class of $E(\Kc)=\Kc\left[\frac{x}{x-y}\right]\subset F_{\Psi}(\Psi)$ 
and $P_1$ is the class of $\Kc\langle\Psi\rangle$. \end{itemize} 

\item For each $?\in\{\mathrm{a,b,c}\}$, the non-closed points of $\mathrm{Spec}_{K_?}$ 
are represented by noetherian objects. The non-closed points are $P_2^{(K_?)},P_3^{(K_?)},\dots$ and 
\begin{itemize} \item $P_1^{(K_?)}$ if either $?=\mathrm{c}$ or 
$?=\mathrm{b}$ and $F\neq k$ or $?=\mathrm{a}$ and $\Gamma\neq 0$, 
\item the class $P_2'$ of $(x-y)\cdot\Kc\langle\binom{\Psi}{2}\rangle\subset\Kb\langle\binom{\Psi}{2}\rangle$ 
if $?=\mathrm{c}$, $F=k$ and $p\neq 2$. \end{itemize} 

If, for an integer $s\ge 1$, $K_?\langle\binom{\Psi}{s}\rangle$ is not injective then 
$?\in\{\mathrm{b,c}\}$, $p=2$, $s=2$, $F=k$ and 
$E(K_?\langle\binom{\Psi}{2}\rangle)\cong K_?\langle\Psi^2\smallsetminus\Delta_{\Psi}\rangle$, 
where $\Delta_{\Psi}$ is the diagonal in $\Psi^2=\Psi\times\Psi$. 
\item For $?\in\{\mathrm{a,b}\}$ and any $s>0$, the closure of $P_s^{(K_?)}$ is the set 
$\mathrm{Cl}_{K_?}\cup\{P_1^{(K_?)},P_2^{(K_?)},\dots,P_s^{(K_?)}\}$. 

For $?=\mathrm{c}$, the closure of $P_s^{(\Kc)}$ $($for $s\ge 2)$ is \begin{itemize} \item 
$\mathrm{Cl}_{\Kc}\cup\{P_{1,n}~|~n\in\mathbb Z\}\cup\{P_2^{(\Kc)},\dots,P_s^{(\Kc)}\}$ if $s>2$ and $p=2$, 
\item $\mathrm{Cl}_{\Kc}\cup\{P_{1,n}~|~n\in\mathbb Z\}\cup\{P_2',P_2^{(\Kc)},\dots,P_s^{(\Kc)}\}$ 
if $s>2$ and $p\neq 2$, 
\item $\mathrm{Cl}_{\Kc}\cup\{P_{1,n}~|~(-1)^n=1\mbox{ {\rm in} }k\}\cup\{P_2^{(\Kc)}\}$ if $s=2;$ 
\end{itemize} 
the closure of $P_{1,n}$ is $\mathrm{Cl}_{\Kc}\cup\{P_{1,n}\}$. \end{enumerate} \end{theorem}

In terminology of \cite[\S4.20, p.295]{Popescu}, the categories $\Sm_{\Ka}(\Sy_{\Psi})$ for $\Gamma=\Xi$ 
and $\Sm_{\Kb}(\Sy_{\Psi})$ are {\sl local}, while $\Sm_{\Ka}(\Sy_{\Psi})$ for $\Gamma\neq\Xi$ and 
$\Sm_{K_?}(\Sy_{\Psi})$ for $?\in\{\mathrm{c},\mathrm{d}\}$ are not. 

\begin{example} \label{forms-on-F_Psi} It is shown in \cite[Theorem 6.1]{Lueroth} that, for an $F|k$ 
with $F\neq k$ as above, the points of the Gabriel spectrum of $\Sm_{F_{\Psi}}(\Sy_{\Psi})$ are 
$F_{\Psi}\langle\binom{\Psi}{s}\rangle$, while $F_{\Psi}$ is a cogenerator of $\Sm_{F_{\Psi}}(\Sy_{\Psi})$. 

For any intermediate field $L$ in $F|k$, the natural map $F_{\Psi}\langle\Psi\rangle\otimes_F\Omega^1_{F|L}
\xrightarrow{f[u]\otimes\omega\mapsto f\omega(u)}\Omega^1_{F_{\Psi}|L_{\Psi}}$, is bijective. 
In particular, $\Omega^1_{F_{\Psi}|L_{\Psi}}\cong F_{\Psi}\langle\Psi\rangle$ if $\mathrm{tr.deg}(F|L)=1$ and 
$F|L$ is separable. Therefore, $F_{\Psi}\langle\binom{\Psi}{s}\rangle\cong\Omega^s_{F_{\Psi}|L_{\Psi}}$ (the 
space of differential $s$-forms on $F_{\Psi}(\Psi)$ over $F_{\Psi}$) for all integer $s\ge 0$. \end{example}

Theorems~\ref{Pic_K?}--\ref{spectrum-triple} 
provide several examples of fields $K$ and a non-precompact group $G$ of their automorphisms such that 
the Gabriel spectrum (and in particular the simple objects, cf. Lemma~\ref{rmks_on_Gabriel_topology}) 
of $\Sm_K(G)$ admits an explicit description. In all examples $G=\Sy_{\Psi}$, though the description depends 
crucially on the ``type'' (but not that essentially on the characteristic!) of the field $K$. 

\begin{remarks} \label{various_remarks} \begin{enumerate} \item Let $k$ be a trivial $\Sy_{\Psi}$-field. 
Then it is well-known that $k$ is an injective object of $\Sm_k(\Sy_{\Psi})$ if and only if 
the characteristic of $k$ is $0$, cf. \cite[Theorem~A.17]{H90} and Proposition~\ref{no-incr-morphisms}.
\item Proposition~\ref{Satz1} below implies that the pair $(G,K)$ cannot be reconstructed from the category 
$\Sm_K(G)$. However, for a given set $\Psi$, one can reconstruct the field $\Ka$ with torsion free 
$\mathrm{Pic}_{\Ka}$ from the category $\Sm_{\Ka}(\Sy_{\Psi})$ as follows. The field $k$ is the endomorphism 
ring of any closed point of $\mathrm{Spec}_{\Ka}$. Let $S$ be a basis of $\mathrm{Pic}_K$. If there is 
a (unique) point $I_1$ of $\mathrm{Spec}_{\Ka}$ whose closure is $\mathrm{Pic}_K\sqcup\{I_1\}$ (i.e. either 
$F\neq k$ or $\Gamma\neq 0$) then $L:=\End_{\Sm_{\Ka}(\Sy_{\Psi})}(I_1)$ is the fraction field of the group 
$F$-algebra of $\Gamma$, and $\Ka\cong K_{\Psi,S,0}^{(L|k)}$. Otherwise, $\Ka\cong K_{\Psi,S,0}^{(k|k)}$. 
\item Let $G$ be a permutation group and $L|K$ be a smooth $G$-field extension. Then \begin{itemize} 
\item any injective object of the category $\Sm_L(G)$ is injective as an object of $\Sm_K(G)$; 
\item \label{irred-contained-irred-semilin} any simple object $W$ of $\Sm_K(G)$ can be embedded into 
a simple object of $\Sm_L(G)$ (namely, into any simple quotient of $L\otimes_KW$). \end{itemize} 

In particular, it follows from Theorem~\ref{simple_objects_are_invertible_Pic} that, for some fixed 
$x\neq y$ in $\Psi$, {\it any smooth irreducible representation of $\Sy_{\Psi}$ over a field $k$ 
can be embedded into the representation $(x-y)^{-n}k\left(\frac{u-v}{u-w}~|~u,v,w\in\Psi\right)
\subset x^{-n}k\left(\frac{u}{v}~|~u,v\in\Psi\right)\subset k(\Psi)$ for some $n\in\mathbb Z$}. 
\item Lemma~\ref{iso-restr} identifies the smooth $\Sy_{\Psi}$-sets with sheaves on a site 
$\mathrm{FI}^{\mathrm{op}}$. Thus when dealing with $\Sm_K(G)$, we will switch sometimes 
the terminology from representations to sheaves. \end{enumerate} \end{remarks}

\section{Permutation groups and smooth representations as sheaves} 
\begin{definition} \label{permutation_group} A {\bf permutation group} is a Hausdorff topological group $G$ 
admitting a base of open subsets consisting of the left and right shifts of subgroups. \end{definition} 
If $B$ is a collection of open subgroups such that the finite intersections of conjugates of elements of $B$ form 
a base of open neighbourhoods of $1$ in $G$ then $G$ acts faithfully on the set $\Psi:=\coprod_{U\in B}G/U$, so 
(i) $G$ becomes a {\sl permutation group of the set} $\Psi$, (ii) the shifts of the pointwise stabilizers $G_T$ 
of the finite subsets $T\subset\Psi$ form a base of the topology of $G$. Clearly, $G$ is totally disconnected. 

As we mainly deal not with a permutation group $G$ itself, but rather with smooth $G$-sets (or representations, 
etc.), we may replace $G$ by the {\sl prodiscrete left cancellative semigroup} $\varprojlim_{U\in B}G/U$. 

It is easy to see that, for any collection of smooth $G$-sets and any collection of finitary relations 
on them, their common stabilizer in $G$ is closed. In particular, if $\Psi$ is such a structure as 
a group, ring, field, module (over a `constant' ring), affine or projective space, etc. then 
its automorphism group is a closed subgroup of the symmetric group $\Sy_{\Psi}$. 

Recall (see, e.g., \cite[Expos\'e IV, \S2.4--2.5]{SGA4 I} or \cite[Section 8.1, Example 8.15 (iii)]{Johnstone}) 
that, for any permutation group $G$, the smooth $G$-sets and their $G$-equivariant maps form a topos. Namely, 
let $\mathfrak{T}$ be the category whose objects are the elements of some base $B$ of open subgroups of $G$ 
and \[\Hom_{\mathfrak{T}}(U,V):=\Hom_{G\text{-sets}}(G/V,G/U)=(G/U)^V=\{h\in G~|~hUh^{-1}\supseteq V\}/U.\] 

Unlike the category of smooth $G$-sets, the category $\mathfrak{T}$ does not admit finite products if $G\neq 1$. 
Indeed, fix any $G\neq U\in B$. The product of two copies of $G/U$ should coincide with each 
$G$-orbit in the cartesian square of $G/U$, which is impossible as there are at least two distinct $G$-orbits. 

As any morphism in $\mathfrak{T}$ is onto, it is natural to consider $\mathfrak{T}$ as a site 
with the maximal topology, i.e., to declare {\sl covering} any non-empty sieve. Then the sheaves 
of sets, groups, etc. on $\mathfrak{T}$ are identified with the smooth $G$-sets, groups, etc.: 
$\mathcal F\mapsto\indlim_{U\in B}\mathcal F(U)$ (this is a smooth $G$-set, since any element 
in it belongs to the image of some $\mathcal F(U)$, 
while the $U$-action on it is trivial) and $W\mapsto(U\mapsto W^U)$. 

A presheaf $\mathcal F$ on $\mathfrak{T}$ is a sheaf if and only if, for any diagram 
$U_1\subseteq U_0\supseteq U_2$ in $B$, the diagram $\mathcal F(U_0)\to\mathcal F(U_1)\times\mathcal F(U_2)
\rightrightarrows\prod_{[g]\in U_1\backslash U_0/U_2}\mathcal F(U_1\cap gU_2g^{-1})$ is an equalizer.

\subsection{Substructures} \label{substructures} 
Let $G\subseteq\Sy_{\Psi}$ be a permutation group of a set $\Psi$. 

For a subset $S\subset\Psi$, (i) we denote by $G_S$ the pointwise stabilizer of 
the set $S$; (ii) we call the fixed set $\Psi^{G_S}$ the $G$-{\sl closure} of $S$. 
We say that a subset $S\subset\Psi$ is $G$-{\sl closed} if $S=\Psi^{G_S}$. 

Any intersection $\bigcap_iS_i$ of $G$-closed sets $S_i$ is $G$-closed: as $G_{S_i}\subseteq G_{\bigcap_jS_j}$, 
one has $G_{S_i}s=s$ for any $s\in\Psi^{G_{\bigcap_jS_j}}$, so $s\in\Psi^{G_{S_i}}=S_i$ for any $i$, and thus, 
$s\in\bigcap_iS_i$. This implies that the subgroup generated by $G_{S_i}$'s is dense in $G_{\bigcap_iS_i}$ 
(and coincides with $G_{\bigcap_iS_i}$ if at least one of $G_{S_i}$'s is open). 

The $G$-closed subsets of $\Psi$ form a small category with the morphisms being all those embeddings that are 
induced by elements of $G$: $\Hom(X,Y):=\{g\in G~|~g(X)\subseteq Y\}/G_X$. 

For a $G$-closed subset $T\subset\Psi$, \label{Aut-notation} 
(hiding $G$ and $\Psi$ from notation) set $\Aut(T):=N_G(G_T)/G_T$. 

\begin{lemma} \label{strong-generation} Let $\Psi$ be either {\rm (i)} an infinite set, or {\rm (ii)} an 
infinite-dimensional vector space over a field, or {\rm (iii)} an algebraically closed field extension of 
a field $k$ of infinite transcendence degree. Let $G$ be the automorphism group of $\Psi$, i.e. either 
$\Sy_{\Psi}$, or $\mathrm{GL}_{}(\Psi)$, or the field automorphism group of $\Psi$ over $k$. For an integer 
$n$, let $T\subset\Psi$ be a subset of order $n$, resp. a subspace of dimension $n$, resp. a subfield of 
transcendence degree $n$ over $k$. Then there exist elements $g_0,\dots,g_n\in G$ such that 
$G=\bigcup_{i=0}^nG_Tg_iG_Tg_i^{-1}$. 

Let, in the cases {\rm (i)} and {\rm (ii)}, $T'\subset\Psi$ be a finite $G$-closed subset. 
In the case {\rm (iii)}, let $T,T'$ be algebraically closed subfields in $\Psi|k$ of finite 
transcendence degree. Then the subgroups $G_T$ and $G_{T'}$ generate the subgroup $G_{T\cap T'}$. 
If $T\subsetneqq T'$ and $V^{G_T}=V^{G_{T'}}$ for a $G$-set $V$ then $V^{G_T}=V^G$. \end{lemma} 
\begin{proof} Choose $n+1$ subsets (resp. subspaces, resp. subfields) in general position $T_0,\dots,T_n$ 
isomorphic to $T$, i.e. $\#(T\cup\bigcup_{i=0}^nT_i)=(n+2)n$ (resp. $\dim(T+\sum_{i=0}^nT_i)=(n+2)n$, 
resp. $\mathrm{tr.deg}(TT_0\cdots T_n|k)=(n+2)n$). Looking at the $G_T$-orbit of the identical embedding 
$T_i\hookrightarrow\Psi$, we see that $G_TG_{T_i}=\{g\in G~|~g(T)\cap T_i=\varnothing\}$ (resp. 
$G_TG_{T_i}=\{g\in G~|~\dim(g(T)+T_i)=2n\}$, resp. 
$G_TG_{T_i}=\{g\in G~|~\mathrm{tr.deg}(g(T)T_i|k)=2n\}$), so $G=\bigcup_{i=0}^nG_Tg_iG_Tg_i^{-1}$ 
for elements $g_0,\dots,g_n\in G$ such that $g_i(T)=T_i$ for all $i$. 

More generally, let us show that $G_TG_{T'}=\{g\in G_{T\cap T'}~|~g(T')\cap T=T\cap T'\}=:\Xi$. 
The inclusion $\subseteq$ is trivial. On the other hand, the set $\Xi/G_{T'}=G_T/(G_T\cap G_{T'})$
consists of all embeddings of $T\cup T'$ (or rather of its $G$-closure) into $\Psi$ identical on $T$ 
that are induced by elements of $G$, while such embeddings form a $G_T$-orbit. 

Let $T_0\subseteq T$ be a minimal $G$-closed subset such that $V^{G_T}=V^{G_{T_0}}$. Assuming 
$T_0$ is not initial (i.e. $\neq\varnothing$ or 0), let $g\in N_G(G_{T'})$ be an element such that 
$g(T_0)\neq T_0$. Then \[V^{G_{T'}}=V^{G_{T'}}\cap g(V^{G_{T'}})=V^{G_{T_0}}\cap g(V^{G_{T_0}})=
V^{G_{T_0}}\cap V^{G_{g(T_0)}}=V^{\langle G_{T_0},G_{g(T_0)}\rangle}=V^{G_{T_0\cap g(T_0)}},\] 
contradicting the minimality of $T_0$. \end{proof} 

\begin{lemma} \label{open-subgrps-descr} Let $G$ be either $\Sy_{\Psi}$ for an infinite set $\Psi$, or 
$\mathrm{GL}_{\mathbb F_q}(\Psi)$ for an infinite-dimensional vector space $\Psi$ over a finite field 
$\mathbb F_q$. For any open subgroup $U$ of $G$ there exists a unique $G$-closed subset $T\subset\Psi$ 
such that $G_T\subseteq U$ and the following equivalent conditions hold: {\rm (a)} 
$T$ is minimal; {\rm (b)} $G_T$ is normal in $U$; {\rm (c)} $G_T$ is of finite index in $U$. In particular, 
{\rm (i)} such $T$ is finite, {\rm (ii)} the open subgroups of $G$ correspond bijectively to the pairs 
$(T,H)$ consisting of a finite $G$-closed subset $T\subset\Psi$ and a subgroup $H\subseteq\Aut(T)$ under  
$(T,H)\mapsto\{g\in N_G(G_T)~|~\text{{\rm restriction of $g$ to $T$ belongs to $H$}}\}$. \end{lemma} 
\begin{proof} Any open subgroup $U$ in $G$ contains the subgroup $G_T$ for a finite $G$-closed 
subset $T\subset\Psi$. Assume that $T$ is chosen to be minimal. If $\sigma\in U$ then 
$U\supseteq\sigma G_T\sigma^{-1}=G_{\sigma(T)}$, and therefore, (i) $\sigma(T)$ is also minimal, (ii) 
$U$ contains the subgroup generated by $G_{\sigma(T)}$ and $G_T$. By Lemma \ref{strong-generation}, the 
subgroup generated by $G_{\sigma(T)}$ and $G_T$ is $G_{T\cap\sigma(T)}$ and thus, $U$ contains the subgroup 
$G_{T\cap\sigma(T)}$. The minimality of $T$ means that $T=\sigma(T)$, i.e., $U\subseteq N_G(G_T)$. If 
$T'\subset\Psi$ is another minimal subset such that $G_{T'}\subseteq U$ then, by Lemma~\ref{strong-generation}, 
$G_{T\cap T'}\subseteq U$, so $T=T'$. This proves (b) and (the uniqueness in the case) (a). It follows 
from (b) that $G_T\subseteq U\subseteq N_G(G_T)$, so $G_T$ is of finite index in $U$. As the subgroups 
$G_T$ and $G_{T'}$ are not commensurable for $T'\neq T$, we get the uniqueness in the case (c). \end{proof}

\begin{notation} For a set $\Psi$ and a subset $T\subseteq\Psi$, denote (i) by $\Sy_{\Psi|T}$ the pointwise 
stabilizer of $T$ in the group $\Sy_{\Psi}$; (ii) by $\Sy_{\Psi,T}:=\Sy_{\Psi\smallsetminus T}\times\Sy_T$ 
the group of all permutations of $\Psi$ preserving $T$ (in other words, the setwise stabilizer of $T$ in 
$\Sy_{\Psi}$, or equivalently, the normalizer of $\Sy_{\Psi|T}$ in $\Sy_{\Psi}$). \end{notation} 

\begin{example} If $G=\Sy_{\Psi}$ and $B$ consists of $\Sy_{\Psi|J}$ for all finite $J\subset\Psi$ 
then $\Hom_{\mathfrak{T}}(\Sy_{\Psi|I},\Sy_{\Psi|J})$ is naturally identified with the set of embeddings 
$J\hookrightarrow I$. This means that $\mathfrak{T}$ is anti-equivalent to the category $\mathrm{FI}$. 
Here $\mathrm{FI}$ is the category of all finite sets, where the morphisms are the injections.

A presheaf $\mathcal F$ on $\mathrm{FI}^{\mathrm{op}}$ is a sheaf if and only if, for any 
diagram $J_1\hookleftarrow J_0\hookrightarrow J_2$ in $\mathrm{FI}$, the diagram 
$\mathcal F(J_0)\to\mathcal F(J_1)\times\mathcal F(J_2)\rightrightarrows\mathcal F(J_1\sqcup_{J_0}J_2)$ 
is an equalizer, where $\sqcup_{J_0}$ is the colimit of $J_1\hookleftarrow J_0\hookrightarrow J_2$ 
in the category of sets. 
\end{example}

\begin{notation} 
For each `reasonable'\ category $\mathcal C$, each $J\in\mathrm{FI}$ and each $\mathcal C$-valued 
sheaf $\mathcal F$ on $\mathrm{FI}^{\mathrm{op}}$, denote by $\mathcal F_{+J}$ the sheaf 
$I\mapsto\mathcal F(I\sqcup J)$. Thus, \begin{itemize} \item $(-)_{+J}\colon\mathcal F\mapsto\mathcal F_{+J}$ 
is an endofunctor on the category of $\mathcal C$-valued sheaves on $\mathrm{FI}^{\mathrm{op}}$; 
\item $(-)_{+\varnothing}$ is the identity functor; 
\item any embedding of finite sets $J_1\hookrightarrow J_2$ induces a morphism of functors 
$(-)_{+J_1}\to(-)_{+J_2}$; \item for each pair of finite $J,J'$, there is a natural 
isomorphism $(-)_{+J}\circ(-)_{+J'}\cong(-)_{+(J\sqcup J')}$. \end{itemize} \end{notation} 

For a sheaf of rings $\mathcal A$, we say 
`{\it an $\mathcal A$-module}'\ instead of `{\it a sheaf of $\mathcal A$-modules}'. 

If $\mathcal A$ is a sheaf of rings and $\mathcal F$ is an $\mathcal A$-module then 
$\mathcal F_{+J}$ is an $\mathcal A_{+J}$-module, so there is a natural morphism of functors 
$\mathcal A_{+J}\otimes_{{\mathcal A}}(-)\to(-)_{+J}$ on the category of $\mathcal A$-modules.

\begin{lemma} \label{iso-restr} Let $\Psi$ be an infinite set. The functor 
\[\nu_{\Psi}\colon\{\text{{\rm sheaves of sets on $\mathrm{FI}^{\mathrm{op}}$}}\}\xrightarrow{\sim}
\{\text{{\rm smooth $\Sy_{\Psi}$-sets}}\},\quad\mathcal F\mapsto\mathcal F(\Psi):=
\indlim_{J\subset\Psi}\mathcal F(J),\] where $J$ runs over the finite subsets of $\Psi$, 
is an equivalence of categories. 

If $\mathcal O$ is a sheaf of fields then $\nu_{\Psi}$ induces an equivalence 
$\{\text{{\rm $\mathcal O$-modules}}\}\xrightarrow{\sim}\Sm_{\mathcal O(\Psi)}(\Sy_{\Psi})$. 

The functor $\nu_{\Psi}$ admits a quasi-inverse $\nu_{\Psi}^{-1}$ such that for any infinite 
subset $\Psi'\subseteq\Psi$ the equivalences $\nu_{\Psi'}\circ\nu_{\Psi}^{-1}\colon
\{\text{{\rm smooth $\Sy_{\Psi}$-sets}}\}\xrightarrow{\sim}\{\text{{\rm smooth $\Sy_{\Psi'}$-sets}}\}$ 
and $\Sm_{\mathcal O(\Psi)}(\Sy_{\Psi})\xrightarrow{\sim}\Sm_{\mathcal O(\Psi')}(\Sy_{\Psi'})$ 
are given by $M\mapsto\indlim_{J\subset\Psi'}M^{\Sy_{\Psi|J}}\subseteq M^{\Sy_{\Psi|\Psi'}}$, 
where $J$ runs over the finite subsets of $\Psi'$. 

For infinite sets $\Psi''\subseteq\Psi'\subseteq\Psi$ and a sheaf $\mathcal F$ on $\mathrm{FI}^{\mathrm{op}}$, 
one has $\nu_{\Psi''}\circ\nu_{\Psi'}^{-1}(\mathcal F(\Psi))=\mathcal F(\Psi''\cup(\Psi\smallsetminus\Psi'))$. 

Let $J$ be a finite subset of $\Psi$. Then the functor 
$\nu_{\Psi\smallsetminus J}\circ(-)_{+J}\circ\nu_{\Psi}^{-1}$ is isomorphic to the functor 
\[\{\text{{\rm smooth $\Sy_{\Psi}$-sets}}\}\xrightarrow{\mathrm{Res}_{\Sy_{\Psi|J}}}
\{\text{{\rm smooth $\Sy_{\Psi|J}$-sets}}\}=\{\text{{\rm smooth $\Sy_{\Psi\smallsetminus J}$-sets}}\}.\] 
\end{lemma} 
\begin{proof} This is essentially \cite[Lemma 3.4]{H90}. \end{proof}

\begin{lemma} \label{inject} Let $G$ be a group, $A$ be a division 
ring endowed with a $G$-action $G\to\Aut_{\mathrm{ring}}(A)$, and $V$ be an $A\langle G\rangle$-module. 
Then $V^G$ is an $A^G$-module and the natural map $A\otimes_{A^G}V^G\to V$ is injective. \end{lemma} 
\begin{proof} This is well-known (see, e.g. \cite[Lemma 3.1]{H90}). \end{proof}

\begin{remark} Let $\Psi$ be an infinite set, $A$ be an integral smooth $\Sy_{\Psi}$-ring, $L|K$ be 
a smooth $\Sy_{\Psi}$-field extension, and $W\in\Sm_K(\Sy_{\Psi})$. Then, for any subset $\Psi'\subset\Psi$, 
\begin{enumerate} \item \label{int-cl} the subring $A^{(\Psi')}_{(\Psi)}$ is integrally closed in $A$, 
\item \label{inj-te} the multiplication map $K\otimes_{K^{(\Psi')}_{(\Psi)}}W^{(\Psi')}_{(\Psi)}\to W$ 
is injective $($in particular, with $W=L)$. \end{enumerate} 
\begin{proof} (\ref{int-cl}) If an element $u\in A$ is integral over $A^{(\Psi')}_{(\Psi)}$, i.e. 
$u^n+a_{n-1}u^{n-1}+\cdots+a_0=0$ for a finite subset $I\subset\Psi'$ and some $a_i\in A^{\Sy_{\Psi|I}}$, 
then the $\Sy_{\Psi|I}$-orbit of $u$ is finite, and therefore, $u\in A^{\Sy_{\Psi|I}}$. 

(\ref{inj-te}) follows from the injectivity of $K\otimes_{K^{\Sy_{\Psi|I}}}W^{\Sy_{\Psi|I}}\to W$ 
(Lemma~\ref{inject}). \end{proof} \end{remark} 

\begin{example} Given a sheaf of abelian groups $\mathcal F$, an integer $s\ge 0$ 
and $J\in\mathrm{FI}$, define the presheaves of abelian groups 
$h^J_{\mathcal F}\colon I\mapsto\mathcal F(I)\langle\{J\hookrightarrow I\}\rangle$ and 
$\bar{h}^s_{\mathcal F}\colon I\mapsto\mathcal F(I)\langle\binom{I}{s}\rangle$. 
E.g., $h^{\varnothing}_{\mathcal F}=\bar{h}^0_{\mathcal F}=\mathcal F$ and 
$h^{\{\ast\}}_{\mathcal F}=\bar{h}^1_{\mathcal F}\colon I\mapsto\mathcal F(I)\langle I\rangle$. 
Then $h^J_{\mathcal F}$ and $\bar{h}^s_{\mathcal F}$ are sheaves of abelian groups, 
faithfully flat for $\mathcal F=\underline{\mathbb Z}$. \end{example} 

\section{Hilbert's theorem 90 and period fields} 
\label{Satz90_period_fields} The following is a version of Speiser's generalization of Hilbert's theorem 90, 
cf. \cite[Satz 1]{Speiser}, or \cite[Prop. 3, p.159]{CL}, or \cite[Proposition 5]{dSG}. 
\begin{proposition} \label{Satz1} Let $K$ be a field of characteristic $p\ge 0$ endowed with a smooth 
action of a permutation group $G$, and $H\subseteq G$ be a closed normal subgroup. Then \begin{itemize} 
\item the functor $\Gamma_H\colon\Sm_K(G)\to\Sm_{K^H}(G/H)$, $V\mapsto V^H$, is an equivalence if and only 
if $H$ acts on $K$ faithfully and $H$ is precompact, i.e., any open subgroup of $H$ is of finite index$;$ 
\item the category $\Sm_K(G)$ is semisimple if and only if $G$ is precompact and $p$ divides index of 
no open subgroup of the action kernel $N:=\ker[G\to\Aut(K)]$. \end{itemize} \end{proposition} 
\begin{proof} As it is mentioned after Definition~\ref{permutation_group}, $G$ is a 
permutation group of a set $\Psi$, so $K\langle\Psi\rangle$ is a smooth faithful representation of $G$. 
If $\Gamma_H$ is an equivalence then its left adjoint $K\otimes_{K^H}(-)$ should be quasi-inverse. 
Obviously, the adjunction map $K\otimes_{K^H}K\langle\Psi\rangle^H\to K\langle\Psi\rangle$ is 
surjective only if the $H$-action on $K$ is faithful. On the other hand, $K\langle G/U\rangle^H\neq 0$ 
for each open subgroup $U\subseteq G$, so $U\cap H$ is of finite index in $H$, i.e. $H$ is precompact. 

If $G$ is finite then \cite[Satz 1]{Speiser}, appropriately reformulated, implies that any $K$-semilinear 
representation of $G$ is a sum of copies of $K$. Namely, with $k:=K^G$, the field extension $K|k$ is 
finite, so the multiplication and the natural $G$-action on $K$ give rise to a $k$-algebra homomorphism 
from $K\langle G\rangle$ to the algebra $\End_k(K)$ of endomorphisms of $K$ considered as $k$-vector space, 
which is (a) surjective by Jacobson's density theorem and (b) injective by the independence of characters. 
Then any $K\langle G\rangle$-module is isomorphic to a direct sum of copies of $K$. 

If $H$ is precompact and faithful on $K$, we have to check that the adjunction map 
$K\otimes_{K^H}V^H\xrightarrow{\xi}V$ is surjective for any object $V$ of $\Sm_K(G)$. 
For each $v\in V$, consider the intersection $S$ of all conjugates of the stabilizer of $v$ in $H$. 
Thus, $v$ is contained in the $K^S$-semilinear representation $V^S$ of the group $H/S$. 
As $H/S$ is finite and the action of $H/S$ on $K^S$ is faithful (cf. \cite[p.151]{Jacobson}), 
$V^S=K^S\otimes_{(K^S)^{H/S}}(V^S)^{H/S}=K^S\otimes_{K^H}V^H$, i.e., $v$ is contained in the image of $\xi$. 

If $\Sm_K(G)$ is semisimple then the projection $\pi_U\colon K\langle G/U\rangle\xrightarrow{[g]\mapsto 1}K$, 
splits for any open subgroup $U\subseteq G$, so $K\langle G/U\rangle^G$ contains some $\alpha\neq 0$, 
and thus, $[G:U]<\infty$, i.e. $G$ is precompact. Then $\alpha=\sum_ia_i\sum_{[g]\in O_i}[g]$, where $O_i$ 
are $N$-orbits on $G/U$, so $O_i\cong N/(U\cap N)$, and therefore, $\pi_U(\alpha)=\#(N/(U\cap N))\sum_ia_i$. 
Then $\pi_U(\alpha)\neq 0$ only if $p\nmid[N:U\cap N]$. 

Conversely, under our assumption on $N$, the categories $\Sm_K(N)$ and $\Sm_K(G/N)$ are semisimple, 
so for any $V\in\Sm_K(G)$ the functors $\Sm_K(G)\xrightarrow{\Hom_{\Sm_K(N)}(V,-)}\Sm_K(G/N)$ 
and $\Sm_K(G/N)\xrightarrow{\Gamma_{G/N}}\Sm_{K^G}(1)$ are exact. Then their composition 
$\Hom_{\Sm_K(G)}(V,-)\colon\Sm_K(G)\to\Sm_{K^G}(1)$ is exact as well, so 
$\mathrm{Ext}^{>0}_{\Sm_K(G)}(-,-)=0$, i.e., $\Sm_K(G)$ is semisimple. \end{proof} 

\begin{example} Let $\Psi$ be an infinite-dimensional vector space over a finite field, 
$U\subset\mathrm{GL}(\Psi)$ be an open subgroup. There is a unique maximal compact normal subgroup $H$ of $U$. 
Namely, let $J$ be the minimal subspace of $\Psi$ such that $U$ contains be the pointwise stabilizer of $J$ in 
$\mathrm{GL}(\Psi)$ (equivalently, $J$ is the maximal finite-dimensional subspace of $\Psi$ stabilized by $U$). 
Let $K$ be a field endowed with a smooth $U$-action faithful on $H$. Then $U/H=\mathrm{PGL}(\Psi/J)$ and 
$\Gamma_H\colon\Sm_K(U)\to\Sm_{K^H}(\mathrm{PGL}(\Psi/J))$, $V\mapsto V^H$, is an equivalence of categories. 

Similarly (and simpler), if $\Psi$ is an infinite set and $U\subseteq\Sy_{\Psi}$ be an open subgroup, 
$K$ be a field endowed with a smooth and faithful $U$-action 
then $\Gamma_H\colon\Sm_K(U)\to\Sm_{K^H}(\Sy_{\Psi\smallsetminus J})$, $V\mapsto V^H$, 
is an equivalence of categories, where $J$ is the maximal finite subset of $\Psi$ 
stabilized by $U$, and $H$ is the maximal finite normal subgroup of $U$. \end{example} 

Recall that a family $\mathcal{S}$ of objects in a category is {\sl generating} (resp., {\sl cogenerating}) 
if any pair of distinct morphisms $g_1,g_2\colon X\rightrightarrows Y$ admits a morphism $\theta\colon S\to X$ 
(resp., $\theta\colon Y\to S$) with $S\in\mathcal{S}$ such that $g_1\circ\theta\neq g_2\circ\theta$ (resp., 
$\theta\circ g_1\neq\theta\circ g_2$). 
For a Grothendieck category, a family $\mathcal{S}$ is generating (resp., cogenerating) 
if and only if any object is {\it a quotient of a direct sum} (resp., 
{\it a subobject of a direct product}) of objects in $\mathcal{S}$. An object $U$ is called 
a {\sl generator} (resp., {\sl cogenerator}) if the family $\{U\}$ is generating (resp., cogenerating). 

\begin{definition} \label{def_period_field} For a permutation group $G$ and a smooth $G$-field $F$, 
a smooth $G$-field extension $K|F$ is called a {\sl weak $G$-period extension} of $F$ 
if $K$ is a cogenerator of $\Sm_K(G)$. 

If, moreover, $K^G=F^G$ then $K$ is called a (strong) {\sl $G$-period extension} of $F$. \end{definition} 

The first natural questions are: Given a permutation group $G$ and a smooth $G$-field $F$, does 
there exist a $G$-period extension of $F$? Can it be chosen to be `minimal' in some sense? 

\begin{example} \label{minimal_period} Let $k$ be a field and $K|k$ be a weak $\Sy_{\Psi}$-period 
extension. Then $K$ contains a $\Sy_{\Psi}$-subfield isomorphic to $k(\Psi)$, so $k(\Psi)$ 
can be considered as a minimal $\Sy_{\Psi}$-period extension of $k$. \end{example} 
\begin{proof} Fix some $x\in\Psi$. There should exist an embedding of $K\langle\Psi\rangle$ into 
a product of copies of $K$, so $K^{\Sy_{\Psi}}\neq K^{\Sy_{\Psi|x}}$. 
As $\Sy_{\Psi|x}$ is a maximal proper subgroup of $\Sy_{\Psi}$, fixing an element of 
$K^{\Sy_{\Psi|x}}\smallsetminus K^{\Sy_{\Psi}}$, we can identify the $\Sy_{\Psi}$-set $\Psi$ with an 
$\Sy_{\Psi}$-orbit in $K$. Suppose that some pairwise distinct elements $x_1,\dots,x_n\in\Psi\subset K$ 
are algebraically dependent over $k$, i.e. $P(x_1,\dots,x_n)=0$ for a non-zero polynomial $P$ over $k$, 
and $n\ge 2$ is minimal. Then $P(x_1^g,\dots,x_n^g)=0$ for any $g\in\Sy_{\Psi}$. The 
$\Sy_{\Psi|\{x_2,\dots,x_n\}}$-orbit of $x_1$, i.e. $\Psi\smallsetminus\{x_2,\dots,x_n\}$, 
gives infinitely many solutions of the equation $P(X,x_2,\dots,x_n)=0$, 
while it has at most $\deg_XP$ solutions. This is contradiction. \end{proof} 

\section{Open subgroups, permutation modules, generators and projectives} 
A permutation group $G$ is called {\sl unibased} if, for any open proper subgroup, 
the finite intersections of its conjugates form a base of open subgroups of $G$. 

A reason to introduce this class is a `simplicity' property (Proposition~\ref{simplicity_of_Sm}) of $\Sm_K(G)$, 
if $K$ is a non-trivial $G$-field: any non-zero subcategory of $\Sm_K(G)$ closed under direct products and 
subquotients in $\Sm_K(G)$ is equivalent to $\Sm_K(G)$ itself. (There exist, however, non-zero Grothendieck 
subcategories of $\Sm_K(G)$ not equivalent to $\Sm_K(G)$, e.g. the full subcategory of all semisimple objects.) 

\begin{example} The following groups are unibased. 

(i) Simple finite groups. 

(ii) The automorphism groups of each of the following structures: an infinite set, an 
infinite-dimensional projective space, an algebraically closed field extension $F|k$ 
of infinite transcendence degree of an algebraically closed field $k$. \end{example} 

\begin{lemma} Any unibased group $G$ is topologically simple;\footnote{equivalently, any smooth 
non-trivial representation of $G$ is faithful} its topology is induced by any non-trivial continuous 
homomorphism $\varphi$ from $G$ to a permutation group. In particular, a permutation group is unibased if 
and only if it is topologically simple and {\bf minimal} in the sense of \cite{Stephenson}. \end{lemma} 
\begin{proof} If $S$ is a non-trivial closed normal proper subgroup of $G$, fix an element 
$g\in G\smallsetminus S$ and a neighbourhood of $g$ in $G\smallsetminus S$. We may choose this neighbourhood 
of the form $gU$ for an open subgroup $U$ of $G$. As $g\notin SU$, $SU$ is a proper open subgroup of 
$G$, while all conjugates of $SU$ contain $S$, and thus, cannot generate a Hausdorff topology on $G$. 

The target of $\varphi$ admits an open subgroup $U$ that does not contain the image 
of $\varphi$, so $\varphi^{-1}(U)$ is an open proper subgroup of $G$. \end{proof}

\begin{remark} A unibased group need not be simple as an abstract group: the finitary 
permutations of an infinite set $\Psi$ form a proper normal subgroup of $\Sy_{\Psi}$, while 
the finite-rank perturbations of the identity operator on an infinite-dimensional vector space 
$W$ form a proper normal subgroup of $\mathrm{PGL}(W)$. On the other hand, the unibased group 
$\mathrm{Aut}(\mathbb C|\overline{\mathbb Q})$ is simple, \cite{D.Lascar}. \end{remark} 

\begin{proposition} \label{simplicity_of_Sm} Let $G$ be a unibased group, and $K$ be a non-trivial 
smooth $G$-field. 

Then, for any non-zero object $V$ of $\Sm_K(G)$, $\Sm_K(G)$ admits a system of generators embeddable 
into direct product of copies of $V$. In particular, $\Sm_K(G)$ is equivalent to its arbitrary 
non-zero subcategory closed under direct products and subquotients in $\Sm_K(G)$. \end{proposition} 
\begin{proof} 
Consider $G$ as a subgroup of $\Aut(K)$. The objects $K\langle G/G_T\rangle$ for all finite subsets 
$T\subseteq K$ such that $V^{G_T}\neq 0$ form a system of generators of the category $\Sm_K(G)$, so 
it suffices to show that there is an embedding of $K\langle G/G_T\rangle$ into a direct product of 
copies of $V$, or equivalently, that there is a family of morphisms $K\langle G/G_T\rangle\to V$ 
with vanishing common kernel. 

Fix a non-zero $\alpha\in V^{G_T}$ and consider, for all $t\in K^{G_T}$, the morphisms 
$t\alpha:K\langle G/G_T\rangle\to V$, $a[\sigma]\mapsto a(t\alpha)^{\sigma}$, for all $a\in K$ and 
$\sigma\in G$. Let $\xi=\sum_{i=1}^Na_i\sigma_i\in K\langle G/G_T\rangle$ be an element in the common 
kernel of these morphisms. The elements of the set $G/G_T$ can be considered as (pairwise distinct) 
$K^{\times}$-valued characters of the group $(K^{G_T})^{\times}$, since $G_{K^{G_T}}=G_T$, so 
the element $\xi$ can be considered as a $K$-linear relation between characters. Due to the 
linear independence of such characters, one has $a_1=\dots=a_N=0$, i.e., 
$\bigcap_{t\in(K^{G_T})^{\times}}\ker(K\langle G/G_T\rangle\xrightarrow{t\alpha}V)=0$. \end{proof}

\begin{lemma} \label{no-proj} Let $G$ be a permutation group, and $K$ be a $G$-field. Assume that for any 
open subgroup $U\subset G$ there is an open subgroup $U'\subset U$ such that each conjugate of $U'$ meets $U$ 
in a subgroup of infinite index in $U$. Then there are no non-zero projectives in $\Sm_K(G)$. \end{lemma} 
\begin{proof} Let $P\neq 0$ be an object of $\Sm_K(G)$. Fix a non-zero $e\in P$. For 
each $u\in P$, fix an open subgroup $U_u\subset\St_e\cap\St_u$ such that 
each conjugate of $U_u$ meets $\St_e$ in a subgroup  of infinite index. Consider the 
surjection $\bigoplus_{u\in P}K\langle G/U_u\rangle\xrightarrow{\pi}P$, $[g]_u\mapsto gu$. If 
the support of an element in $K\langle G/U_u\rangle^{\St_e}$ contains $[g]\in G/U_u$ 
then it contains $\St_egU_u/U_u$. This set is of the same cardinality as 
$g^{-1}\St_egU_u/U_u=g^{-1}\St_eg/(g^{-1}\St_eg\cap U_u)\cong\St_e/(\St_e\cap gU_ug^{-1})$. 
As the latter set is infinite, 
$K\langle G/U_u\rangle^{\St_e}=0$, so $\pi$ does not split, and thus, $P$ is not projective. \end{proof} 

\begin{example} The automorphism groups of infinite sets, infinite-dimensional vector spaces, 
algebraically closed field extensions $F|k$ of infinite transcendence degree 
satisfy assumptions of Lemma~\ref{no-proj}. \end{example}

\begin{lemma} \label{no-finite-dim-submod} Let $G$ be a group acting on a field $K$. Let $U\subset G$ 
be a subgroup of infinite index, and $V$ be a $K\langle U\rangle$-module. Let $K'\subseteq K$ be 
a $G$-invariant subfield. Then there are no non-zero $K'\langle G\rangle$-submodules in 
$K\langle G\rangle\otimes_{K\langle U\rangle}V$ finite-dimensional over $K'$. \end{lemma} 
Example. If $V=K$ then $K\langle G\rangle\otimes_{K\langle U\rangle}V\xrightarrow{\sim}K\langle G/U\rangle$, 
$a[g]\otimes f\mapsto af^g[g]$. 
\begin{proof} One has a canonical decomposition 
$K\langle G\rangle\otimes_{K\langle U\rangle}V=\bigoplus_{[g]\in G/U}V_{[g]}$ into a direct sum of $K$-vector 
spaces, where $V_{[g]}=K\langle gU\rangle\otimes_{K\langle U\rangle}V$. By the {\sl support} of an element 
$\alpha=(\alpha_{[g]})_{[g]\in G/U}\in K\langle G\rangle\otimes_{K\langle U\rangle}V$ we mean the subset of 
$G/U$ consisting of those $[g]\in G/U$ for which $\alpha_{[g]}\neq 0$. Let 
$V'\subseteq K\langle G\rangle\otimes_{K\langle U\rangle}V$ be a $K'\langle G\rangle$-submodule finite-dimensional 
over $K'$, and $S\subset G/U$ be the union of the supports in $G/U$ of all elements of $V'$. Then $S$ is finite 
and $G$-stable. However, the only finite and $G$-stable subset of $G/U$ is empty, so $V'=0$. \end{proof}

\begin{lemma} \label{embedding_into_prod} Let $G$ be a group acting on a field $K$. Then the following conditions 
on a subgroup $U\subseteq G$ are equivalent: \begin{itemize} \item an element $g\in G$ acts identically on $K^U$ 
if and only if $g\in U$; \item $K\langle G/U\rangle$ embeds into a product of copies of $K$. \end{itemize} 
Assume that $U$ satisfies these conditions, and $V\xrightarrow{\lambda}K$ 
is a morphism of $K\langle G\rangle$-modules inducing a non-zero map $V^U\to K^U$. 
Then $K\langle G/U\rangle$ embeds into a product of copies of $V$. \end{lemma}

\begin{proof} For any $K\langle G\rangle$-module $V$, define a morphism of $K\langle G\rangle$-modules 
$K\langle G/U\rangle\xrightarrow{\varphi_V}\Hom_{K^G}(V^U,V)$ by $\sum_gb_g[g]\mapsto[v\mapsto\sum_gb_gv^g]$. 
The first condition means that the elements $[g]\in G/U$ can be considered as certain pairwise distinct 
one-dimensional characters $\chi_{[g]}\colon(K^U)^{\times}\to K^{\times}$. By Artin's independence of 
characters theorem, the characters $\chi_{[g]}$ are linearly independent in the $K$-vector space of all 
functions $(K^U)^{\times}\to K$, so $\varphi_K$ is injective. Conversely, 
$\Hom_{K\langle G\rangle}(K\langle G/U\rangle,K)=K^U$, while if $\varphi_K$ is injective then 
$\chi_{[g]}\neq\chi_{[1]}$ for any $[g]\in G/U\smallsetminus\{[1]\}$, i.e. $g$ acts non-trivially on $K^U$. 

Fix some $v\in V^U$ with $\lambda(v)=1$. Then the composition 
\[K\langle G/U\rangle\xrightarrow{\varphi_V}\Hom_{K^G}(V^U,V)\xrightarrow{\lambda}
\Hom_{K^G}(V^U,K)\xrightarrow{|_{K^U\cdot v}}\Hom_{K^G}(K^U\cdot v,K)\cong\Hom_{K^G}(K^U,K)\] 
coincides with $\varphi_K$, which is injective. \end{proof}

\begin{lemma} \label{simple-in-K} Let $G$ be a group, $K$ be a $G$-field, $k\subseteq K^G$ be a subfield, $F|k$ 
be a field extension. Assume that $K\otimes_kF$ is integral. Let $L$ be the fraction field of $K\otimes_kF$. 

Then any simple $K\langle G\rangle$-submodule $M$ of $L$ coincides with 
$K\cdot a$ for some non-zero $a\in L^G$. \end{lemma} 
\begin{proof} Let $Q\in L^{\times}$ be a non-zero element of $M$, so $Q=\alpha/\beta$ is a ratio of 
a pair of elements $\alpha,\beta\in K\otimes_kA$ for a finitely generated $k$-subalgebra $A$ of $F$. 
There is a finite field extension $k'|k$ and a $k$-algebra homomorphism $\varphi\colon A\to k'$ with 
an invertible image of $\alpha\beta$ under the $G$-equivariant $K$-algebra homomorphism 
$\Phi\colon K\otimes_kA\to K\otimes_kk'$, $\sum_ib_i\otimes a_i\mapsto\sum_ib_i\varphi(a_i)$. 
The homomorphism $\Phi$ extends to $(K\otimes_kA)[(\beta^g)^{-1}~|~g\in G]$ with a non-zero restriction to 
a morphism of left $K\langle G\rangle$-modules $M\to K\otimes_kk'\cong\bigoplus_{\text{a $k$-basis of $k'$}}K$, 
so $M\cong K$, and thus, $M=K\cdot a$ for some non-zero $a\in M^G\subseteq L^G$. \end{proof} 

\begin{remark} \label{fixed_field_alg_closed} If $G$ has no open proper subgroups of finite index (e.g. if 
$G=\Sy_{\Psi}$ for an infinite set $\Psi$) and $K$ is smooth then $K^G$ is algebraically closed in $K$: 
any finite extension of $K^G$ in $K$ is fixed by an open subgroup of $G$ of finite index, i.e. by $G$. 

If $K^G$ is algebraically closed in $K$ then all $G$-fields $L$ in Lemma~\ref{simple-in-K} come from 
the particular case $k=K^G$, since $L$ is the fraction field of $K\otimes_{K^G}\tilde{F}$, 
where $\tilde{F}$ is the fraction field of $K^G\otimes_kF$. \end{remark} 

\begin{corollary} \label{simple-in-K_S} Let $\Psi$ be an infinite set, $J\subset\Psi$ be a subset 
with infinite complement. Let $F|k$ be a non-trivial regular field extension. 
Then any simple $F_{\Psi\smallsetminus J}\langle\Sy_{\Psi|J}\rangle$-submodule 
$M$ of $F_{\Psi}$ coincides with $aF_{\Psi\smallsetminus J}$ for some $a\in F_J^{\times}$. 
In particular, $M$ is isomorphic to $F_{\Psi\smallsetminus J}$. \end{corollary} 
\begin{proof} This is Lemma~\ref{simple-in-K} for $G=\Sy_{\Psi|J}$, 
$K=F_{\Psi\smallsetminus J}$, $k=K^G$, with $F$ replaced by $F_J$ (so $L=F_{\Psi}$). \end{proof}

\begin{lemma} \label{no-simple-submod} Let $G$ be a group acting on 
a field $K$. Let $U\subset G$ be a subgroup of infinite index such that an element 
$g\in G$ acts identically on $K^U$ if and only if $g\in U$. Set $k:=K^G$. Let $F|k$ be a field extension 
such that $k$ is algebraically closed either in $F$ or in $K$. Let $L$ be the fraction field of $K\otimes_kF$. 
Then \begin{itemize} \item $L\langle G/U\rangle$ contains no simple $K\langle G\rangle$-submodules; 
\item for each integer $s\ge 0$, there is a natural bijection 
between $s$-dimensional $k$-vector subspaces $\Xi$ in $K^U$ and $K\langle G\rangle$-submodules $V_{\Xi}$ 
in $K\langle G/U\rangle$ such that $K\langle G/U\rangle/V_{\Xi}\cong K^s$ in $\Sm_K(G)$. 
\end{itemize} \end{lemma} 
\begin{proof} By Lemma~\ref{embedding_into_prod}, $L\langle G/U\rangle$ embeds into a product of 
copies of $L$, so any simple $K\langle G\rangle$-submodule $V$ of $L\langle G/U\rangle$ embeds into $L$. 
By Lemma~\ref{simple-in-K}, $V$ is isomorphic to $K$, contradicting Lemma~\ref{no-finite-dim-submod}. 

The image of $K\langle G/U\rangle\xrightarrow{\varphi_K}\Hom_k(K^U,K)$ is dense, which means that 
for any finite-dimensional $k$-vector subspace $\Xi\subset K^U$ the composition 
$K\langle G/U\rangle\xrightarrow{\varphi_K}\Hom_k(K^U,K)\xrightarrow{\text{restriction to $\Xi$}}\Hom_k(\Xi,K)$ 
is surjective: otherwise there was a non-zero element 
$\sum_ia_i\otimes f_i\in K\otimes_k\Xi$ vanishing on the image, i.e., $\sum_ia_i\otimes f_i^g=0$ for all 
$g\in G/U$, contradicting linear independence of characters $(K^U)^{\times}\to K^{\times}$. 

To each $k$-vector subspace $\Xi\subset K^U$ we associate the $K\langle G\rangle$-submodule 
$V_{\Xi}:=\bigcap_{\xi\in\Xi}\ker\xi$, where all $\xi$'s are considered as morphisms $K\langle G/U\rangle\to K$. 
Conversely, to each $K\langle G\rangle$-submodule $V\subset K\langle G/U\rangle$ we associate 
the $k$-vector subspace $\Xi_V:=\bigcap_{v\in V}\ker v\subset K^U$, where all $v$'s are considered as 
$k$-linear maps $K^U\to K$. \end{proof} 

{\sc Example.} Let $G=\Sy_{\Psi}$ and $U\subset\Sy_{\Psi}$ be a maximal proper subgroup, i.e., $U=\Sy_{\Psi,I}$ 
for a finite non-empty subset $I\subset\Psi$ (so $\Sy_{\Psi}/U$ can be identified with the set $\binom{\Psi}{\#I}$). 
Suppose that $K^{\Sy_{\Psi,I}}\neq k$. Then we are under assumptions of Lemma~\ref{no-simple-submod}, so there 
are no irreducible $K$-semilinear subrepresentations in $K\langle\binom{\Psi}{\#I}\rangle$. 

\begin{remarks} \label{more_remarks} \begin{enumerate} \item \label{condition_on_symm-group-field} Let 
$G$ be as in Lemma~\ref{open-subgrps-descr}. It follow from the explicit description of open subgroups 
in Lemma~\ref{open-subgrps-descr} that index of any proper open subgroup $U$ in $G$ is infinite. 

One cannot claim that, for an arbitrary smooth $K$, there are no 
elements in $\Sy_{\Psi}\smallsetminus U$ acting identically on $K^U$, even if $K^U\neq K^{\Sy_{\Psi}}$. 
For instance, if $K=F_{\binom{\Psi}{s}}$ for $s>1$, $U=\Sy_{\Psi|S}$ for a finite 
$S$, $\#S=s$, then $K^U\cong F$, but $H=\Sy_{\Psi,S}$ acts trivially on $K^U$. 

However, this is the case if $K^{\Sy_{\Psi|x}}\neq K^{\Sy_{\Psi}}$, or equivalently, 
$K^{\Sy_{\Psi}}(\Psi)$ embeds into $K$. 
\item One cannot claim that $K\langle G/U\rangle$ is trivial, 
when $U$ is of finite index in $G$, even if $G$ acts faithfully on $K$. E.g., let $U=\Sy_{\Psi|\{x,y\}}$, 
$G=\Sy_{\Psi,\{x,y\}}=\Sy_2\times U$, $K=K'=k(\frac{z-y}{x-y}~|~z\in\Psi\smallsetminus\{x,y\})$. Then $K^U=k$ 
and $K\langle G/U\rangle^G=(K\langle G/U\rangle^U)^G=(k\langle G/U\rangle)^G=k([1]+[(12)])$. 

\end{enumerate} \end{remarks} 

\begin{lemma} \label{finite-index-split} Let $K$ be a field, $G$ be a group of automorphisms of 
the field $K$. Let $U\subseteq H\subseteq G$ be open subgroups of $G$, where the index of $U$ 
in $H$ is finite. Suppose that $[K^U:K^H]=\#(H/U)$. Then there is a natural isomorphism 
$\rho_G\colon K\langle G/H\rangle\otimes_{K^H}K^U\xrightarrow{\sim}K\langle G/U\rangle$, given by 
$\sum_ia_i[g_i]\otimes f_i\mapsto\sum_i\sum_{[\xi]\in G/U,~[\xi]\bmod H=[g_i]}a_if_i^{\xi}[\xi]$. \end{lemma} 
\begin{proof} Any element $\alpha$ of $K\langle G/H\rangle\otimes_{K^H}K^U$ can be presented as 
$\sum_ia_i[g_i]\otimes f_i$, where $[g_i]$ are determined uniquely if they are pairwise distinct. Then 
$\rho_G(\alpha)=0$ only if $\sum_{i:~[\xi]\bmod H=[g_i]}a_if_i^{\xi}=0$ for all $\xi$, i.e. only if $\alpha=0$. 
As $[K^U:K^H]=\#(H/U)$, $K$-dimensions of the source and the target of the embedding 
$\rho_H\colon K\langle H/H\rangle\otimes_{K^H}K^U=K\otimes_{K^H}K^U\to K\langle H/U\rangle$ coincide, 
so $\rho_H$ is bijective. Then the induction functor $K\langle G\rangle\otimes_{K\langle H\rangle}(-)$ 
transforms the bijection $\rho_H$ to the isomorphism $\rho_G$. \end{proof} 

\begin{example} \label{fin-ind-ex} Let $G=\Sy_{\Psi}$ for an infinite set $\Psi$, and $U\subseteq H\subseteq G$ 
be open subgroups of $G$, where the index of $U$ in $H$ is finite. By Lemma~\ref{open-subgrps-descr}, there 
exists a unique finite subset $J\subseteq\Psi$ such that $\Sy_{\Psi|J}$ is a subgroup of finite index in $U$. 
Then the assumption of Lemma~\ref{finite-index-split} is equivalent to the faithfulness of the $\Sy_J$-action 
on $K^{\Sy_{\Psi|J}}$. This is the case if, for a field $k$, $K$ contains 
one of the following fields: \begin{itemize} \item $k(x/y~|~x,y\in\Psi)$; \item $k(x-y~|~x,y\in\Psi)$ and 
either $k$ is not of characteristic 2, or $U$ coincides with $\Sy_{\Psi|\{x,y\}}$ for no $x\neq y$ in $\Psi$; 
\item $k\left(\frac{x-y}{x-z}~|~x,y,z\in\Psi\right)$, where $U$ coincides 
with $\Sy_{\Psi|\{x,y\}}$ for no $x\neq y$ in $\Psi$, 
\item $k\left(\frac{(w-x)(y-z)}{(x-y)(z-w)}~|~w,x,y,z\in\Psi\right)$, where $U$ coincides with 
$\Sy_{\Psi|\{x,y\}}$ for no $x\neq y$ in $\Psi$ and $U$ contains $\Sy_{\Psi|\{x,y,z\}}$ as 
a subgroup of index $\le 3$ for no pairwise distinct $x,y,z\in\Psi$. \end{itemize} 
Then Lemma~\ref{finite-index-split} asserts that, for any such open subgroup $U\subseteq G$, one has 
$K\langle G/U\rangle\cong K\langle\binom{\Psi}{s}\rangle^r$ for some integer $s\ge 0$ and $r\ge 1$. \end{example}

\begin{definition} \label{Gal-pair} A pair $(K,G)$ consisting of a permutation group $G$ and a $G$-field $K$ is 
called a {\sl Galois pair} if there is an open subgroup $U_0$ such that the following equivalent conditions hold: 
\begin{enumerate} \item \label{separ} for any open subgroup $U\subseteq U_0$ and any subgroup 
$\widetilde{U}\subseteq G$ containing $U$ as a proper subgroup of finite index one has $K^U\neq K^{\widetilde{U}}$; 
\item \label{faithf} for any open subgroup $U\subseteq U_0$ and any subgroup $\widetilde{U}\subseteq G$ 
containing $U$ as a proper normal subgroup of 
finite index the $\widetilde{U}/U$-action on $K^U$ is faithful. \end{enumerate} \end{definition} 

\begin{proof} (\ref{separ})$\Rightarrow$(\ref{faithf}). If $U\subseteq U_0$ is normal in $\widetilde{U}$ and 
$h\in\widetilde{U}/U$ acts trivially on $K^U$ then $K^U=K^{\langle g,U\rangle}$, where $g$ is a preimage of $h$ in 
$\widetilde{U}$, so $\langle g,U\rangle=U$, i.e. $g\in U$, so the $\widetilde{U}/U$-action on $K^U$ is faithful. 

(\ref{faithf})$\Rightarrow$(\ref{separ}). If $U\subseteq U_0$ and $\widetilde{U}\subseteq G$ contains $U$ as a 
subgroup of finite index, set $U':=\bigcap_{g\in\widetilde{U}}gUg^{-1}\subseteq U$. Then $\widetilde{U}/U'$ acts 
faithfully on $K^{U'}$, i.e. $\widetilde{U}/U'$ is the Galois group of $K^{U'}|K^{\widetilde{U}}$, and therefore, 
$K^U:=(K^{U'})^{U/U'}=(K^{U'})^{\widetilde{U}/U'}=:K^{\widetilde{U}}$ if and only if $U=\widetilde{U}$. \end{proof}

\begin{lemma} \label{Gal-pair-examples} Let $K$ be a field endowed with a faithful smooth $G$-action, 
where $G$ is either $\Sy_{\Psi}$ for an infinite set $\Psi$, or a quotient by a central subgroup of 
the automorphism group of an infinite-dimensional vector space $\Psi$ 
over a field. Then $(K,G)$ is a Galois pair. \end{lemma} 
\begin{proof} Let $J\subset\Psi$ be a maximal $G$-closed subset such that $K^{G_J}=K^G$, and 
$x_1,x_2,\dots$ be a sequence of elements of $\Psi$ in general position with respect to $J$, i.e. each 
$x_{i+1}$ is not in the $G$-closure $J^{(i)}$ of $J\cup\{x_1,\dots,x_i\}$. As $G_{J\cup\{x_i\}}$ and 
$G_{J\cup\{x_j\}}$ generate $G_J$ for all $i\neq j$, $K^{G_{J\cup\{x_i\}}}\cap K^{G_{J\cup\{x_j\}}}=K^{G_J}$, 
so the group $\Sy_{\{x_1,\dots,x_N\}}\subset\Aut(J^{(N)}|J)$ acts faithfully on 
$\sum_{i=1}^NK^{G_{J\cup\{x_i\}}}$. Then, for any $N\ge 5$, the group $\Aut(J^{(N)})$ acts faithfully on 
$K^{G_{J^{(N)}}}\supseteq\sum_{i=1}^NK^{G_{J\cup\{x_i\}}}$, 
since all non-central normal subgroups of $\Aut(J^{(N)})$ (i.e. $\Sy_{J^{(N)}}$ and $\mathfrak{A}_{J^{(N)}}$ 
if $G=\Sy_{\Psi}$; $\mathrm{SL}(J^{(N)})\cdot\langle\mbox{central subgroup}\rangle$ if $G=\mathrm{GL}(\Psi)$) 
meet the subgroup $\Sy_{\{x_1,\dots,x_N\}}\subset\Aut(J^{(N)}|J)$ non-trivially. \end{proof}

\begin{lemma} \label{semilin-gener} Let $K$ be a field and $G$ be a group of automorphisms 
of the field $K$. Let $B$ be such a system of open subgroups of $G$ that any open subgroup contains a 
subgroup conjugated, for some $H\in B$, to an open subgroup of finite index in $H$. Suppose that $(K,G)$ 
is a Galois pair. Then the objects $K\langle G/H\rangle$ for all $H\in B$ form a system of generators 
of the category $\Sm_K(G)$ of smooth $K$-semilinear representations of $G$. \end{lemma} 
\begin{proof} Let $V$ be a smooth semilinear representation of $G$. Then the stabilizer of any vector 
$v\in V$ is open, i.e., the stabilizer of some vector $v'$ in the $G$-orbit of $v$ admits an open subgroup 
$U\subseteq G$ that is of finite index in some $H\in B$. The $K$-linear envelope $W$ of the (finite) 
$H$-orbit of $v'$ is a smooth $K$-semilinear representation of $H$. We may assume that $U$ is normal 
in $H$. Then the elements of the $H$-orbit of $v'$ are fixed by $U$, so $W=K\otimes_{K^U}W^U$. 

If $(K,G)$ is a Galois pair then we may assume that $U\subseteq U_0$ as in Definition~\ref{Gal-pair}, 
so the $H/U$-action on $K^U$ is faithful. By Proposition~\ref{Satz1}, this implies that 
$W^U=K^U\otimes_{(K^U)^{H/U}}(W^U)^{H/U}=K^U\otimes_{K^H}W^H$. Then $W=K\otimes_{K^H}W^H$, 
i.e., $v'$ belongs to the $K$-linear envelope of the $K^H$-vector subspace fixed by $H$. 
As a consequence, there is a morphism from a finite cartesian power of $K\langle G/H\rangle$ 
to $V$, containing $v'$ (and therefore, containing $v$ as well) in the image. \end{proof}

\begin{example} \label{Omega-binom} Let $K$ be a field endowed with a smooth faithful $\Sy_{\Psi}$-action. 
Then (i) the assumptions of Lemma \ref{semilin-gener} hold if $B$ is the set of subgroups $\Sy_{\Psi,T}$ for 
a collection of finite subsets $T\subset\Psi$ of unbounded cardinality, (ii) $K\langle\binom{\Psi}{N}\rangle$ 
is isomorphic to $K\langle\Sy_{\Psi}/\Sy_{\Psi,T}\rangle$ for any $T$ of order $N$. 

Thus, any collection of the objects $K\langle\binom{\Psi}{N}\rangle$ for infinitely many $N$ form a system 
of generators of the category $\Sm_K(\Sy_{\Psi})$ of smooth $K$-semilinear representations of $\Sy_{\Psi}$. 

Let $K$ admit an $\Sy_{\Psi}$-equivariant field embedding 
$\iota\colon k\left(\frac{x-y}{x-z}~|~x,y,z\in\Psi\right)\to K$ for a prime field $k$. For 
each $n\in\mathbb Z$ and some $x\neq y$ in $\Psi$, denote by $(x-y)^nK$ the one-dimensional 
$K$-vector space spanned by $(x-y)^n$ endowed with the semilinear $\Sy_{\Psi}$-action 
$\sigma\colon(x-y)^n\mapsto(x-y)^n\cdot\left(\frac{x^{\sigma}-y^{\sigma}}{x-y}\right)^n$. 
Obviously, the isomorphism class of $(x-y)^nK$ is independent of $x,y\in\Psi$. Then 
$n\mapsto(x-y)^nK$ gives rise to homomorphism from $\mathbb Z$ to the Picard group of the 
isomorphism classes of invertible objects of $\Sm_K(\Sy_{\Psi})$, which is trivial if and only 
if $\iota$ extends to an $\Sy_{\Psi}$-equivariant field embedding of $k(u-v~|~u,v\in\Psi)$. 

For any object $V$ of $\Sm_K(\Sy_{\Psi})$, set $(x-y)^nV:=(x-y)^nK\otimes_KV$. Then 
\[\bigwedge\nolimits_K^NK\langle\Psi\rangle\xrightarrow[\sim]{[s_1]\wedge\dots\wedge[s_N]\mapsto 
\prod_{1\le i<j\le N}(s_i-s_j)\{s_1,\dots,s_N\}}(x-y)^{\binom{N}{2}}K\langle\binom{\Psi}{N}\rangle\] 
is an isomorphism. If $K\cong k(\Psi)$ for a $\Sy_{\Psi}$-subfield $k$ then 
$\bigwedge_K^NK\langle\Psi\rangle\cong\Omega^N_{K|k}$ (differential $N$-forms on $K$ over $k$), 
$[s_1]\wedge\dots\wedge[s_N]\leftrightarrow ds_1\wedge\dots\wedge ds_N$. \qed \end{example} 

\begin{lemma} \label{fixed_basis} Let $G$ be a group, $U\subseteq G$ be a subgroup, $A$ be 
a $G$-ring, $V$ be an $A\langle G\rangle$-module admitting an $A$-basis $B$ fixed by $U$, i.e. 
the natural map $A\otimes_{A^U}V^U\xrightarrow{\times}V$ is an isomorphism. 

Then there is a natural isomorphism of $A\langle G\rangle$-modules 
$\eta\colon A\langle G/U\rangle\otimes_{A^U}V^U\xrightarrow{\sim}V\langle G/U\rangle$. \end{lemma} 
\begin{proof} Define $\eta$ by $\sum_ia_i[g_i]\otimes v_i\mapsto\sum_ia_iv_i^{g_i}\otimes[g_i]$. 
Then $\eta$ is $G$-equivariant and transforms the $A$-basis $B\times(G/U)$ of the source to the $A$-basis 
$\bigsqcup_{[g]\in G/U}g(B)\otimes[g]$ of the target, so it extends by $A$-linearity to an isomorphism. \end{proof} 

\begin{lemma} \label{indecomp} Let $G$ be a group, $U\subseteq G$ be a subgroup such that 
$[U:U\cap(gUg^{-1})]=\infty$ unless $g\in U$, $K$ be a $G$-field, and $V$ be a $K\langle G\rangle$-module. 
Then {\rm(i)} $V\langle G/U\rangle\cong\bigoplus_BK\langle G/U\rangle$ if and only if $V^U$ contains 
a $K$-basis of $V;$ {\rm(ii)} $V\langle G/U\rangle$ is indecomposable if $V$ is invertible $($i.e. 
$\dim_KV=1)$. \end{lemma} 
\begin{proof} As $UgU=g(g^{-1}Ug)U$ consists of $[g^{-1}Ug:(g^{-1}Ug)\cap U]=[U:U\cap(gUg^{-1})]$ classes 
in $G/U$, the only finite $U$-orbit on $G/U$ is $\{[U]\}$, so $V\langle G/U\rangle^U\subseteq V\cdot[U]$, 
and (i) follows from Lemma~\ref{fixed_basis}. 

If $V$ is invertible then 
$\End_{K\langle G\rangle}(V\langle G/U\rangle)=\End_{K\langle G\rangle}(K\langle G/U\rangle)\cong 
K\langle G/U\rangle^U=K^U\cdot[U]$ is a field, so $V\langle G/U\rangle$ is indecomposable. \end{proof} 

For a wide class of categories, one can define the {\sl length} of an object. In particular, the length 
of a $G$-closed subset $X$ is the minimal cardinality of the subsets $S\subset\Psi$ such that $X$ is the 
$G$-closure of $S$. 

{\sc Examples.} 1. Let $\Psi$ be an infinite set, possibly endowed with a structure of 
a projective space. Let $G$ be the group of automorphisms of $\Psi$, respecting the structure, 
if any. Let $J$ be the $G$-closure of a finite subset in $\Psi$, i.e., a finite subset or 
a finite-dimensional subspace. Let $U$ be the stabilizer of $J$ in $G$. Then $G/U$ is identified 
with the set of all $G$-closed subsets in $\Psi$ of the same length as $J$. 

2. By Lemma \ref{indecomp}, the object $K\langle G/U\rangle$ of the category 
$\Sm_K(G)$ is indecomposable in the following examples: \begin{enumerate} \item 
$G$ is the group of projective automorphisms of an infinite projective space $\Psi$ (i.e., 
either $\Psi$ is infinite-dimensional, or $\Psi$ is defined over an infinite field), $U$ is 
the setwise stabilizer in $G$ of a finite-dimensional subspace $J\subseteq\Psi$. Then $G/U$ 
is identified with the Grassmannian of all subspaces in $\Psi$ of the same dimension as $J$. 
\item $G$ is the group of permutations of an infinite set $\Psi$, $U$ is 
the stabilizer in $G$ of a finite subset $J\subset\Psi$. Then $G/U$ is identified 
with the set $\binom{\Psi}{\#J}$ of all subsets in $\Psi$ of order $\#J$. 
\item $G$ is the automorphism group of an algebraically closed extension $F$ of 
a field $k$, $U$ is the stabilizer in $G$ of an algebraically closed subextension 
$L|k$ of finite transcendence degree. Then $G/U$ is identified with the set 
of all subextensions in $F|k$ isomorphic to $L|k$. \end{enumerate} 

\section{Structure of smooth semilinear representations of \texorpdfstring{$\Sy_{\Psi}$}{}} 

\subsection{The level filtration} \label{level_filtration} 
Let $G$ be a permutation group of a set $\Psi$, and $\mathcal A$ be a concrete category such 
that each object $H$ admits a unique minimal subobject containing a given subset of $H$.  
Let $\mathcal A(G)$ be the category whose objects are the objects of $\mathcal A$ with their underlying 
sets endowed with a smooth $G$-action; the morphisms are $G$-equivariant morphisms in $\mathcal A$. 

\begin{definition} Let $H$ be an object of $\mathcal A(G)$. For each integer $s\ge 0$, define 
\begin{itemize} \item ${}''N_s^GH={}''N_s^{\Psi,G}H$ as the minimal subobject of $H$ containing $H^{G_J}$ 
for all $J\in\binom{\Psi}{s};$ 
\item ${}'N_sH={}'N_s^GH={}'N_s^{\Psi,G}H$ as the minimal subobject of $H$ containing 
all subobjects $H'\subseteq H$ of $\mathcal A(G)$ such that there exists 
an open subgroup $U\subseteq G$ with the property that ${}''N_s^UH'=H';$

\item $N_sH=N_s^GH=N_s^{\Psi,G}H\in\mathcal A(G)$ as the minimal subobject in $H$ containing all 
subobjects $H'\in\mathcal A(G)$ in $H$ such that there exists an open subgroup $U\subseteq G$ (depending 
on $H'$) such that there are no proper subobjects of $\mathcal A$ in $H'$ containing ${H'}^{U'}$ 
for all $J\in\binom{\Psi}{s}$ and all open subgroups $U'\subseteq U_J$ of finite index. \end{itemize}

We say that $H$ is {\sl of level} $\le s$ {\sl with respect to} $\Psi$ if $H=N_s^{\Psi,G}H$. \end{definition}

\begin{example} The only examples of (full subcategories of) the category $\mathcal A(G)$ considered 
here are the category of smooth $G$-fields and $\Sm_A(G)$ for an associative ring $A$ endowed with 
a smooth $G$-action. 

If, for an integer $s\ge 0$, an object $V$ of 
$\Sm_A(G)$ is a quotient of a direct sum of objects $A\langle G/G_J\rangle$ for a collection of 
$J\in\binom{\Psi}{s}$ then $N_s^GV=V$. 

For any pair of subgroups $U,H\subseteq G$, restriction to $U$ splits $A\langle G/H\rangle$ as 
$\bigoplus_{\alpha\in U\backslash G/H}A\langle U\widetilde{\alpha}\rangle$, where 
$\widetilde{\alpha}\in G/H$ is any representative of $\alpha$, so $G/H\supseteq U\widetilde{\alpha}
\xrightarrow[\sim]{u\widetilde{\alpha}\mapsto[u]}U/(U\cap\widetilde{\alpha}H\widetilde{\alpha}^{-1})$. 
In particular, when $H=G_J$ we see that restriction to open subgroups $U$ preserves the level.

If $G$ is either $\Sy_{\Psi}$ or $\mathrm{GL}(\Psi)$, then 
$A\langle\binom{\Psi}{t}_q\rangle=\bigoplus_{\Lambda\subseteq I}M_{\Lambda}$, where $M_{\Lambda}$ is 
the free $A$-module on the set of all subobjects of $\Psi$ of length $t$ and meeting $I$ along $\Lambda$. 

\end{example}

\begin{lemma} \label{restriction_to_open} Let $G$ be a permutation group of a set $\Psi$, $A$ 
be an associative ring endowed with a smooth $G$-action, $V$ be an $A\langle G\rangle$-module. Then 
\begin{itemize} \item ${}''N_0^G$, ${}'N_0^G$ and $N_0^G$ are independent of $\Psi;$ 
\item ${}''N_s^G\subseteq{}'N_s^G\subseteq N_s^G$ are functors to $\Sm_A(G)$, 
${}''N_s^G\ {}''N_t^G={}''N_t^G\ {}''N_s^G={}''N_s^G$ and $N_s^GN_t^G=N_t^GN_s^G=N_s^G$ for $s\le t;$ 
\item ${}''N_0^GV\subseteq{}''N_1^GV\subseteq{}''N_2^GV\subseteq\dots$ 
and $N_0^GV\subseteq N_1^GV\subseteq N_2^GV\subseteq\dots$ are functorial filtrations, exhausting 
if $V$ is smooth. Both filtrations stabilize for finitely generated smooth $V$. \end{itemize} 

Let $U\subseteq G$ be an open subgroup. Then any smooth $A\langle G\rangle$-module $V$ is also smooth when 
considered as an $A\langle U\rangle$-module, and $N_s^GV\subseteq N_s^UV$. If $V\neq N_0^GV$ then 
$N_0^GV\neq N_0^UV$. 

Suppose that the set $U\backslash G/U'$ is finite for any open subgroup $U'\subseteq G$. 
Then any smooth finitely generated $A\langle G\rangle$-module 
is also finitely generated as an $A\langle U\rangle$-module. \end{lemma} 
\begin{proof} The equality ${}''N_s^G\ {}''N_t^G={}''N_t^G\ {}''N_s^G={}''N_s^G$ for $s\le t$ is evident. 
Then $V'={}''N_s^G$ satisfies ${}''N_s^GV'=V'$, so $V'={}''N_s^GV\subseteq N_s^GV$. 

Let us show that $N_s^GV\subseteq N_s^UV$, or equivalently, that $V'\subseteq N_s^UV$ for any subobject 
$V'\subseteq V$ in $\Sm_A(G)$ such that there exists an open subgroup $U'\subseteq G$ with the property 
that ${}''N_s^{U'}V'=V'$. By definition, the natural morphism $\bigoplus_{J\in\binom{\Psi}{s}}
A\langle U'/U'_J\rangle\otimes_{A^{U'_J}}(V')^{U'_J}\to V'$ in $\Sm_A(U')$ is surjective. Set $U'':=U\cap U'$. 
The left $U''$-action on $U'/U'_J$ gives a splitting 
$A\langle U'/U'_J\rangle=\bigoplus_{[g]\in U''\backslash U'/U'_J}A\langle U''gU'_J/U'_J\rangle.$ 
Using the equality $U''\cap gU'_Jg^{-1}=U''_{g(J)}$ and the $U''$-equivariant bijection 
$U''gU'_J/U'_J\xrightarrow{\sim}U''/(U''\cap gU'_Jg^{-1})$ ($[g]\mapsto[1]$), we see that 
the image of $A\langle U''gU'_J/U'_J\rangle\otimes_{A^{U'_J}}(V')^{U'_J}$ in $V$ 
($[ug]\otimes v\mapsto ugv$) coincides with the image of 
$A\langle U''/U''_{g(J)}\rangle\otimes_{A^{U'_{g(J)}}}(V')^{U'_{g(J)}}$ ($[u]\otimes gv\mapsto ugv$), 
i.e. the natural morphism $\bigoplus_{J\in\binom{\Psi}{s}}
A\langle U''/U''_J\rangle\otimes_{A^{U'_J}}(V')^{U'_J}\to V'$ is surjective, and therefore, 
the natural morphism $\bigoplus_{J\in\binom{\Psi}{s}}
A\langle U''/U''_J\rangle\otimes_{A^{U'_J}}(V')^{U''_J}\to V'$ is surjective as well, 
i.e. ${}''N_s^{U''}V'=V'$.

The $A\langle G\rangle$-modules $A\langle G/U'\rangle$ for all open subgroups $U'$ of $G$ form a generating 
family of the category of smooth $A\langle G\rangle$-modules. It suffices, thus, to check that 
$A\langle G/U'\rangle$ is a finitely generated $A\langle U\rangle$-module for all open subgroups $U'$ of $G$. 
Choose representatives $\alpha_i\in G/U'$ of the elements of $U\backslash G/U'$. Then $G/U'=\coprod_iU\alpha_i$, 
so $A\langle G/U'\rangle\cong\bigoplus_iA\langle U/(U\cap\alpha_iU'\alpha_i^{-1})\rangle$ 
is a finitely generated $A\langle U\rangle$-module. \end{proof}

\begin{example} 1. The finiteness assumption of Lemma~\ref{restriction_to_open} is valid for any 
{\sl Roelcke precompact} group $G$, i.e., such that the set $U'\backslash G/U'$ is finite for any open subgroup 
$U'\subseteq G$. Examples of such groups include the profinite groups, the symmetric groups $\Sy_{\Psi}$, 
the automorphism group of a vector space over a finite field [this is shown in Lemma~\ref{strong-generation}]; 
the open subgroups of any Roelcke precompact group, the products of Roelcke precompact groups. 

2. If an object $V$ of $\Sm_A(G)$ is a sum of $G$-invariant finitely generated $A$-submodules then 
$V=N_0^{\Psi}V$. [Indeed, any finitely generated $A$-module is fixed by an open subgroup of $G$.] 

3. If $G$ is locally precompact and an open subgroup of $G$ acts faithfully on a field $K$ then any object 
of $\Sm_K(G)$ is of level 0. [Indeed, let $U$ be an open precompact subgroup of $G$ acting faithfully on 
$K$. Then $V=K\otimes_{K^U}V^U$, so $V={}''N_0^UV$, and therefore, $V=N_0^GV$.] \end{example} 

For each $j\ge 0$, let $\Sm_K^{\le j}(G)$ be the full subcategory of $\Sm_K(G)$ of the objects of 
level $\le j$. Then the inclusion functor $\Sm_K^{\le j}(G)\hookrightarrow\Sm_K(G)$ admits a right adjoint 
$N_j\colon\Sm_K(G)\to\Sm_K^{\le j}(G)$, $V\mapsto N_jV$, the maximal subobject of $V$ of level $\le j$.

\begin{definition} Let $G$ be a group and $U\subseteq G$ be a subgroup. We say that $U$ is an 
{\sl f-subgroup} if there exist integers $N,m\ge 1$ and elements $g_{ij}\in G$ for $1\le i\le m$ 
and $1\le j\le N$ such that $\bigcup_{j=1}^Ng_{1j}Ug_{2j}U\cdots Ug_{mj}=G$. \end{definition} 

\begin{example} \label{examples_f-subgrp} The following subgroups $U$ of a group $G$ are f-subgroups: 
\begin{itemize} \item a subgroup of $G$ of finite index; \item an f-subgroup of an f-subgroup of $G$; 
\item $U$ contains an f-subgroup of $G$; \item an open subgroup of a Roelcke precompact group $G$; 
\item a parabolic subgroup of a linear algebraic group $G$ over an algebraically closed field; 
\item an open subgroup of $G$ that is either $\mathrm{GL}(V)$ 
for an infinite-dimensional vector space $V$ or $\Aut(F|k)$ for an algebraically closed field 
extension $F|k$ of infinite transcendence degree. \end{itemize} \end{example} 
\begin{proof} The first 4 cases are evident. The last one follows from Lemma~\ref{strong-generation}. 
The 5-th case follows from the Bruhat decomposition. \end{proof} 

\begin{lemma} \label{triv-on-open-subgrp} Let $G$ be a group acting on a left noetherian ring $A$, 
$U\subseteq G$ be an f-subgroup, $V$ be an $A\langle G\rangle$-module generated by $V^U$ as an 
$A$-module. Then $V$ is a sum of $A\langle G\rangle$-submodules noetherian as $A$-modules. \end{lemma} 
\begin{proof} For each finitely generated $A$-submodule $B\subseteq V$, denote by $B'$ an 
$A$-submodule of $V$ containing $B$ and generated by finitely many elements of $V^U$. Then $B'$ is 
an $A\langle U\rangle$-module. For any $v\in V$ and any finite collection $g_1,g_2,\dots,g_m\in G$, 
the $A$-module $A\langle g_1Ug_2U\cdots Ug_m\rangle v$ is contained in the finitely generated 
$A$-module $(g_1((g_2(\cdots(g_{m-2}((g_{m-1}((Ag_mv)'))'))'\cdots))'))'$. As $U$ is an f-subgroup 
of $G$, $A\langle G\rangle v=\sum_{j=1}^NA\langle g_{1j}Ug_{2j}U\cdots Ug_{mj}\rangle v$, 
so this shows that any cyclic $A\langle G\rangle$-submodule of $V$ is contained in a finitely 
generated $A$-module. As $A$ is noetherian, any cyclic $A\langle G\rangle$-submodule of $V$ is 
a noetherian $A$-module. \end{proof}

\begin{corollary} Let $G$ be a permutation group such that any open subgroup is an f-subgroup, $A$ be a 
left noetherian associative ring endowed with a smooth $G$-action, $V$ be an object of $\Sm_A(G)$. 
Then $N_0^GV$ is the sum of all subobjects of $V$ noetherian as $A$-modules. \end{corollary} 
\begin{proof} If an object $V$ of $\Sm_A(G)$ is finitely generated as $A$-module then its generators are fixed 
by an open subgroup of $G$, and thus, $V$ is of level 0. Conversely, a level 0 object $V$ of $\Sm_A(G)$ is 
a sum of subobjects $V_i$ generated as $A$-modules by some elements fixed by an open subgroup $U_i$ of $G$. 
By Lemma~\ref{triv-on-open-subgrp}, $V_i$ is a sum of $A\langle G\rangle$-submodules noetherian as $A$-modules. 
\end{proof}

\subsection{Local structure of finitely generated objects in \texorpdfstring{$\Sm_K(\Sy_{\Psi})$}{}} 

\begin{lemma} \label{growth} Let $n\ge 0$ be an integer, and $\binom{\Psi}{n}$ be the $\Sy_{\Psi}$-orbit 
consisting of all subsets in $\Psi$ of order $n$. Let $A$ be a division ring endowed with a $\Sy_{\Psi}$-action. 
If $M$ is a non-zero $A\langle\Sy_{\Psi}\rangle$-submodule of $A\langle\binom{\Psi}{n}\rangle$ then 
the function $d_M\colon\mathbb Z_{\ge 0}\xrightarrow{\#J\mapsto\dim_{A^{\Sy_{\Psi|J}}}(M^{\Sy_{\Psi|J}})}
\mathbb Z_{\ge 0}\sqcup\{\infty\}\ ($for any finite subset $J\subset\Psi)$ grows 
as a polynomial of degree $n:$ $\binom{N-m}{n}\le d_M(N)\le\binom{N}{n}$ for some $m\ge 0$. \end{lemma} 
\begin{proof} This is a particular case of \cite[Lemma 3.2]{H90}. \end{proof}

For a sheaf of rings $\mathcal A$, we say `{\it an $\mathcal A$-sheaf of finite type}'\ instead of 
`{\it a finitely generated object of the category of $\mathcal A$-modules}'.
\begin{proposition} \label{local-structure} Let $\mathcal O$ be a non-constant 
sheaf of fields on $\mathrm{FI}^{\mathrm{op}}$, and $V$ be an $\mathcal O$-sheaf of finite 
type. Then there is an integer $n\ge 0$ such that for any finite set $J$ of order $\ge n$ 
\begin{itemize} \item the $\mathcal O_{+J}$-module $V_{+J}$ is isomorphic to 
$\bigoplus_{s=0}^\ell(\bar{h}^s_{\mathcal O_{+J}})^{\kappa_s}$ for some uniquely determined integers 
$\ell,\kappa_0,\dots,\kappa_{\ell}\ge 0;$ the integers $\ell$ and 
$\kappa_{\ell}=\dim_{\End(\bar{h}^{\ell}_{\mathcal O_{+J}})}\Hom(V,\bar{h}^{\ell}_{\mathcal O_{+J}})$ 
are independent of $J;$ 
\item $V$ can be embedded into a product of copies of $\mathcal O_{+J}$. \end{itemize} \end{proposition} 
\begin{proof} By Lemma~\ref{restriction_to_open}, for each finite set $J'$, $V_{+J'}$ is an 
$\mathcal O_{+J'}$-sheaf of finite type, so by Lemma~\ref{semilin-gener}, there is a surjection of 
$\mathcal O_{+J'}$-modules $\bigoplus_{s=0}^{\ell(J')}\bar{h}^s_{\mathcal O_{+J'}}{}^{m_{s,J'}}\to V_{+J'}$ 
for some $\ell(J')\ge 0$ and $m_{s,J'}\ge 0$. Choose such a surjection so that $\ell(J')$ is minimal and 
$m_{\ell(J'),J'}\ge 1$ is minimal as well. It is clear, that for each finite set $J''$ with $\#J''\ge\#J'$, 
one has $\ell(J'')\le\ell(J')$ and $m_{\ell(J''),J''}\le m_{\ell(J'),J'}$ if $\ell(J'')=\ell(J')$. 
Fix some $J'$ with minimal possible $\ell(J')$ and $m_{\ell,J'}\ge 1$. The pair $(\ell,m)$ depends 
only on $V$, so it makes sense to denote it $(\ell_V,m_V)$. We replace $\mathcal O$ by $\mathcal O_{+J'}$ 
and $V$ by $V_{+J'}$. Obviously, $(\ell_V,m_V)=(\ell_{V_{+J}},m_{V_{+J}})$ for any finite $J$. 

By induction on $\ell$, we prove the statement for all $V$ with $\ell_V=\ell$, the case $\ell=0$ being trivial. 

Let $\alpha\colon(\bar{h}^\ell_{\mathcal O})^m\to V$ and 
$\beta:\bigoplus_{s=0}^{\ell-1}(\bar{h}^s_{\mathcal O})^{m_s}\to V$ be two morphisms 
such that the morphism $\alpha+\beta\colon(\bar{h}^\ell_{\mathcal O})^m\oplus
\bigoplus_{s=0}^{\ell-1}(\bar{h}^s_{\mathcal O})^{m_s}\to V$ is surjective. 

Suppose that $\alpha$ is not injective, and $0\neq(\xi_1,\dots,\xi_m)\in(\ker\alpha)(I)$ for a finite $I$. 
Without loss of generality, we may assume that $\xi_m=\sum_{i=1}^ba_iI_i\neq 0$ for some $I_i\subseteq I$ 
of order $\ell$ and non-zero $a_i$. Set $J:=I\smallsetminus I_1$. Then $\alpha$ factors 
through a quotient $(\bar{h}^\ell_{\mathcal O})^m/\langle(\xi_1,\dots,\xi_m)\rangle$. On the other hand, 
the inclusion $(\bar{h}^\ell_{\mathcal O})^{m-1}\stackrel{(-,0)}{\hookrightarrow}(\bar{h}^\ell_{\mathcal O})^m$ 
induces a surjection of $\mathcal O_{+J}$-modules \[(\bar{h}^\ell_{\mathcal O})_{+J}^{m-1}
\oplus\bigoplus_{\varnothing\neq\Lambda\subseteq J}\bar{h}^{\ell-\#\Lambda}_{\mathcal O_{+J}}
\xrightarrow{(-,0)+\sum_{\Lambda}(\underbrace{0,\dots,0}_{m-1},-\cup\Lambda)}
((\bar{h}^\ell_{\mathcal O})^m/\langle(\xi_1,\dots,\xi_m)\rangle)_{+J}\] 
giving rise to a surjection of $\mathcal O_{+J}$-modules $(\bar{h}^\ell_{\mathcal O})_{+J}^{m-1}
\oplus\bigoplus_{s=0}^{\ell-1}\bar{h}^s_{\mathcal O_{+J}}{}^{\binom{\#J}{\ell-s}+m_s}\to V_{+J}$. 
This contradicts the minimality assumption on $m$, thus showing that $\alpha$ is injective. 

By Lemma \ref{growth} and the induction hypothesis, for any subsheaf $\mathcal F\subseteq\mathrm{Im}(\beta)$ 
there is some $C>0$ such that $\dim_{{\mathcal O}(J)}\mathcal F(J)\le C\cdot\#J^{\ell-1}$ for all 
sufficiently big $J$. For any non-zero subsheaf $\mathcal F\subseteq(\bar{h}^\ell_{\mathcal O})^m$, 
Lemma~\ref{growth} gives an estimate $\dim_{{\mathcal O}(J)}\mathcal F(J)\ge C'\cdot\#J^\ell$ for some 
$C'>0$ and all sufficiently big $J$. This implies that $\mathrm{Im}(\alpha)\cap\mathrm{Im}(\beta)=0$. 
Therefore, $V\cong(\bar{h}^\ell_{\mathcal O})^m\oplus\mathrm{Im}(\beta)$. By induction hypothesis, 
$\mathrm{Im}(\beta)_{+J}$ is of required type for some $J$, thus completing the induction step. 

By Krull--Remak--Schmidt--Azumaya Theorem, the integers $\ell,\kappa_0,\dots,\kappa_{\ell}\ge 0$ 
in the statement of the Proposition are uniquely determined. 

Fix a finite set $I$ such that $\mathcal O(I)\neq\mathcal O(\varnothing)$. By Remark~\ref{fixed_field_alg_closed}, 
$\mathcal O(\varnothing)$ is algebraically closed in $\mathcal O(J)$ for any $J$, so $\mathcal O(I)$ is 
transcendental over $\mathcal O(\varnothing)$, and therefore, $\mathcal O(I)^{\Sy_I}$ is transcendental 
over $\mathcal O(\varnothing)$ as well. By Lemma~\ref{strong-generation}, $\mathcal O(J)\neq\mathcal O(I)$ 
for any $I\subsetneqq J$. Then Lemma~\ref{embedding_into_prod} implies that, for each integer $s>0$, the 
sheaf $\bar{h}^s_{\mathcal O_{+I}}$ can be embedded into a product of copies of $\mathcal O_{+I}$. Now any 
$\mathcal O$-module $V$ embeds into $V_{+J}$ for any finite $J$, so if $V$ is an $\mathcal O$-sheaf of 
finite type we can use the above description of $V_{+J}$ for sufficiently big $J$ to conclude that, $V$ can 
be embedded into a product of copies of $\mathcal O_{+I}$ for any sufficiently large finite $I$. \end{proof} 

\begin{definition} \label{level-of-a-rep} For a finitely generated object in $\Sm_K(\Sy_{\Psi})$, 
we call the corresponding integer $\ell$ in Proposition~\ref{local-structure} 
the {\sl level} of $V$, and $\kappa_{\ell}$ the {\sl rank} of $V$. \end{definition} 

\begin{remark} \label{points_P} It is easy to show that any 
non-zero submodule of $K\langle\binom{\Psi}{\ell}\rangle$ is of level $\ell$, while the key point of the proof 
of the Proposition is that any quotient of $K\langle\binom{\Psi}{\ell}\rangle$ by a non-zero submodule is of 
level $<\ell$. (More generally, this is true for any uniform module of level $\ell$.) 

In particular, the objects $K\langle\binom{\Psi}{\ell}\rangle$ of $\Sm_K(\Sy_{\Psi})$ are 
monoform,\footnote{A non-zero object $M$ of abelian an category is called {\sl monoform} if for 
any non-zero subobject $N$ of $M$, there exists no common non-zero subobject of $M$ and $M/N$, see 
\cite[Definition 2.1]{Kanda}.} so their injective hulls are indecomposable and pairwise non-isomorphic. 
\end{remark}

\begin{lemma} \label{common_kernel} Let $G$ be a group, $K$ be a $G$-field, $A:=K\langle G\rangle$, $k:=K^G$, 
$V$ and $W$ be $A$-modules. Assume that {\rm (i)} $W$ and the quotient 
of $V$ by its arbitrary non-zero submodule are finite-dimensional over $K;$ {\rm (ii)} $\Hom_A(V,W)\neq 0;$ 
{\rm (iii)} $\End_A(V)$ is an infinite-dimensional division $k$-algebra. 

Then $V$ embeds into a product of copies of $W$. \end{lemma} 
\begin{proof} Let $M$ be the common kernel of the morphisms $V\to W$, so 
$\Hom_A(V,W)=\Hom_A(V/M,W)$. If $M\neq 0$ then $V/M$ is finite-dimensional over $K$, 
therefore, $\Hom_A(V/M,W)=((V/M)^{\vee}\otimes_KW)^G$. As 
$K\otimes_k((V/M)^{\vee}\otimes_KW)^G\xrightarrow{\times}(V/M)^{\vee}\otimes_KW$ is injective, 
the $k$-vector space $\Hom_A(V,W)$ 
is finite-dimensional. But $\Hom_A(V,W)$ is a non-zero right $\End_A(V)$-vector space, so it is 
infinite-dimensional as a $k$-vector space. This contradiction implies $M=0$. \end{proof}

\begin{lemma} \label{level-1_quotients} Let $\Psi$ be an infinite set, $x\in\Psi$, $K$ be 
an $\Sy_{\Psi}$-field, and $\mathcal L\in\mathrm{Pic}_K(\Sy_{\Psi|\{x\}})$. Then any non-zero 
submodule of $V:=K\langle\Sy_{\Psi}\rangle\otimes_{K\langle\Sy_{\Psi|\{x\}}\rangle}\mathcal L\ ($e.g. 
of $K\langle\Psi\rangle)$ is of finite codimension over 
$K;\ \End_{K\langle\Sy_{\Psi}\rangle}(V)\cong K^{\Sy_{\Psi|\{x\}}}=:F;$ for any $s\ge 1$, 
the cokernel of any non-zero morphism of $K\langle\Sy_{\Psi}\rangle$-modules 
$\varphi\colon K\langle\binom{\Psi}{s}\rangle\to V$ is at most $(s-1)$-dimensional. 

Suppose that $K$ is a non-trivial smooth $\Sy_{\Psi}$-field, $K$ is a cogenerator of $\Sm_K(\Sy_{\Psi})$ 
and $K\langle\Psi\rangle$ is injective. Then there are natural bijections
\begin{enumerate} \item between the $K\langle\Sy_{\Psi}\rangle$-submodules of
$K\langle\Psi\rangle$ of codimension $s$ and the $s$-dimensional $k$-vector subspaces in $F;$
\item \label{isomorphism_classes-level-1} between the isomorphism classes of
$K\langle\Sy_{\Psi}\rangle$-submodules of $K\langle\Psi\rangle$ of codimension $s$ and the
$F^{\times}$-orbits of the $s$-dimensional $k$-vector subspaces in the space $F$. \end{enumerate} \end{lemma}
\begin{proof} Any non-zero element of $V$ can be presented as 
$\alpha=\sum_{i=1}^sa_i[g_i]\otimes e$ for some $a_i\in K^{\times}$, a non-zero 
$e\in\mathcal L$ and some $g_i\in\Sy_{\Psi}$ pairwise distinct modulo $\Sy_{\Psi|\{x\}}$. 
For any $g\in\Sy_{\Psi}$ with $gx\notin\{g_2x,\dots,g_sx\}$, there exists 
$h\in\Sy_{\Psi|\{g_2x,\dots,g_sx\}}$ such that $hg_1x=gx$. Then, for any $a\in K$, one has 
$a[g]\otimes e\in(b\cdot h)\alpha+\sum_{i=2}^sK[g_i]\otimes e$, where 
$b=\frac{a}{a_1^h}\left(\frac{e}{e^{g^{-1}hg_1}}\right)^g\in K$. This means that $\alpha$ 
generates a $K\langle\Sy_{\Psi}\rangle$-submodule of codimension $\le s-1$. 

The image under $\varphi$ of any element of $\binom{\Psi}{s}$ is a non-zero linear combination 
of some elements $[g_1]\otimes e,\dots,[g_s]\otimes e$, 
so the cokernel of $\varphi$ is at most $(s-1)$-dimensional. 

One has $\End_{K\langle\Sy_{\Psi}\rangle}(V)=\Hom_{K\langle\Sy_{\Psi|\{x\}}\rangle}(\mathcal L,V)=
\left(\mathcal L^{\vee}\otimes_KK\langle\Sy_{\Psi}\rangle\otimes_{K\langle\Sy_{\Psi|\{x\}}\rangle}
\mathcal L\right)^{\Sy_{\Psi|\{x\}}}$, and therefore, $\End_{K\langle\Sy_{\Psi}\rangle}(V)=
\left(\mathcal L^{\vee}\otimes_KK[1]\otimes_K\mathcal L\right)^{\Sy_{\Psi|\{x\}}}=K^{\Sy_{\Psi|\{x\}}}$.

We already know that the $K$-vector space $K\langle\Psi\rangle/M$ is finite-dimensional for any 
non-zero $K\langle\Sy_{\Psi}\rangle$-submodule $M$ of $K\langle\Psi\rangle$, so assuming that $K$ is a 
period field, the $K\langle\Sy_{\Psi}\rangle$-module $K\langle\Psi\rangle/M$ is isomorphic to a sum 
of copies of $K$, and therefore, $M$ is the common kernel of the elements of a finite-dimensional 
$k$-vector subspace of $\Hom_{K\langle\Sy_{\Psi}\rangle}(K\langle\Psi\rangle,K)=F$. 
As $\End_{K\langle\Sy_{\Psi}\rangle}(K\langle\Psi\rangle)=F$ is a field, the 
$\End_{K\langle\Sy_{\Psi}\rangle}(K\langle\Psi\rangle)^{\times}$-action preserves 
the isomorphism classes of the common kernels. On the other hand, 
$\Hom_{K\langle\Sy_{\Psi}\rangle}(K\langle\Psi\rangle/M,K\langle\Psi\rangle)=0$ 
and, if $K\langle\Psi\rangle$ is injective, 
the restriction morphism $\End_{K\langle\Sy_{\Psi}\rangle}(K\langle\Psi\rangle)\to
\Hom_{K\langle\Sy_{\Psi}\rangle}(M,K\langle\Psi\rangle)$ is an isomorphism, i.e., 
any morphism of $K\langle\Sy_{\Psi}\rangle$-modules $M\to K\langle\Psi\rangle$ 
is induced by an endomorphism of $K\langle\Psi\rangle$. \end{proof} 

\begin{corollary} \label{admiss} Let $\mathcal O$ be a non-constant sheaf of fields on 
$\mathrm{FI}^{\mathrm{op}}$. Then, for any $\mathcal O$-sheaf of finite type $V$, there is 
a unique polynomial $P_V(X)$ over $\mathbb Q$ such that $P_V(\#J)=\dim_{\mathcal O(J)}V(J)$ 
for any sufficiently large finite set $J$. \end{corollary} 
\begin{proof} By Proposition~\ref{local-structure}, there is an isomorphism of $\mathcal O_{+J}$-modules 
$\bigoplus_{s=0}^{\ell}\bar{h}_{\mathcal O_{+J}}^s{}^{\kappa_s}\xrightarrow{\sim}V_{+J}$ 
for some finite set $J$ and integers $\ell,\kappa_0,\dots,\kappa_{\ell}\ge 0$. 
Then $\dim_{\mathcal O(T)}V(T)=\sum_{s=0}^\ell\kappa_s\binom{\#(T\smallsetminus J)}{s}$ for any 
finite set $T$ containing $J$, i.e. $P_V(X)=\sum_{s=0}^\ell\kappa_s\binom{X-\#J}{s}$ is the unique 
polynomial whose value at $\#T$ is $\dim_{\mathcal O(T)}V(T)$ for such $T$. \end{proof}

\begin{lemma} \label{stand-injection} Let $K$ be a field endowed with an $\Sy_{\Psi}$-action; 
$k:=K^{\Sy_{\Psi}}$; $Q_s$ and $\Lambda_s$ be $k$-vector spaces. 
Then any embedding of $K\langle\Sy_{\Psi}\rangle$-modules 
$\iota\colon\bigoplus_{s=0}^{\ell}K\langle\binom{\Psi}{s}\rangle\otimes_kQ_s\hookrightarrow 
M:=\bigoplus_{s\ge 0}K\langle\binom{\Psi}{s}\rangle\otimes_k\Lambda_s$ splits, and 
$\mathrm{Coker}(\iota)\cong\bigoplus_{s\ge 0}K\langle\binom{\Psi}{s}\rangle\otimes_k\Lambda_s'$ 
for subspaces $\Lambda_s'\subseteq\Lambda_s$ of codimension $\dim Q_s$. \end{lemma} 
\begin{proof} We proceed by induction on $\ell\ge 0$. As 
$\Hom_{K\langle\Sy_{\Psi}\rangle}(K\langle\binom{\Psi}{s}\rangle,K\langle\binom{\Psi}{t}\rangle)=0$ for $s<t$, 
$\iota$ factors through $\bigoplus_{s=0}^{\ell}K\langle\binom{\Psi}{s}\rangle\otimes_k\Lambda_s\subseteq M$. 
Fix a subset $S\subset\Psi$ of order $\ell$. Let $\iota_\ell$ be the restriction of $\iota$ 
to the summand $K\langle\binom{\Psi}{\ell}\rangle\otimes_kQ_\ell$. As $\End_{K\langle\Sy_{\Psi}\rangle}
(K\langle\binom{\Psi}{\ell}\rangle)\cong K^{\Sy_{\Psi,S}}$, it follows from Lemma~\ref{growth} that the 
composition of $\iota_\ell$ with the projection $M\to K\langle\binom{\Psi}{\ell}\rangle\otimes_k\Lambda_\ell$ 
is given on $[S]\otimes Q_\ell$ by an injective element $\alpha$ of 
$\Hom_k(Q_\ell,\Lambda_\ell)\otimes_kK^{\Sy_{\Psi,S}}$. 
Choose a $k$-vector subspace $\Lambda_\ell'\subseteq\Lambda_\ell$ such that 
$\Lambda_\ell\otimes_kK^{\Sy_{\Psi,S}}=\mathrm{Im}(\alpha)\dotplus\Lambda_\ell'\otimes_kK^{\Sy_{\Psi,S}}$. 
Then $M=\bigoplus_{s=0}^{\ell-1}K\langle\binom{\Psi}{s}\rangle\otimes_k\Lambda_s\oplus\mathrm{Im}(\iota_\ell)
\oplus K\langle\binom{\Psi}{\ell}\rangle\otimes_k\Lambda_\ell'\oplus
\bigoplus_{s>\ell}K\langle\binom{\Psi}{s}\rangle\otimes_k\Lambda_s$, which reduces the level. \end{proof} 

\begin{theorem}[Local semisimplicity] \label{Local-injectivity} Let $K$ be a field endowed with a faithful 
smooth $\Sy_{\Psi}$-action. Then for any finitely generated object $M$ of $\Sm_K(\Sy_{\Psi})$ and a subobject 
$M'\subseteq M$ there exist an open subgroup $U\subseteq\Sy_{\Psi}$ and a $K\langle U\rangle$-submodule 
$M_0\subseteq M$ such that $M=M'\oplus M_0$. \end{theorem} 
\begin{proof} By \cite[Theorem 3.18]{H90}, $M'$ is finitely generated. By Proposition~\ref{local-structure}, 
there exists an open subgroup $U\subseteq\Sy_{\Psi}$ such that the restriction to $U$ of the inclusion 
$\iota\colon M'\hookrightarrow M$ is as in Lemma~\ref{stand-injection}. Then by Lemma~\ref{stand-injection}, 
$\iota$ is split in $\Sm_K(U)$, i.e. $M=M'\oplus M_0$ for a $K\langle U\rangle$-submodule $M_0\subseteq M$. 
\end{proof} 

\begin{remark} Obviously, a surjection $\bigoplus_{s=0}^{\ell}K\langle\binom{\Psi}{s}\rangle^{\lambda_s}
\to\bigoplus_{s=0}^{\ell}K\langle\binom{\Psi}{s}\rangle^{\kappa_s}$ of finitely generated objects of 
$\Sm_K(\Sy_{\Psi})$ need not be split: $K\langle\Psi\rangle\to K$, $\sum_ia_i[x_i]\mapsto\sum_ia_i$. \end{remark} 

\begin{corollary} \label{additivity} In notation of Corollary~\ref{admiss}, for any exact sequence 
$0\to V_1\to V_2\to\cdots\to V_n\to 0$ of finitely generated objects of $\Sm_K(\Sy_{\Psi})$, 
one has $\sum_{i=1}^n(-1)^iP_{V_i}=0$. \end{corollary} 
\begin{proof} By Theorem~\ref{Local-injectivity}, this sequence splits 
after restriction to an open subgroup of $\Sy_{\Psi}$, thus giving the claim. \end{proof} 

\begin{corollary} In notation of Corollary~\ref{admiss}, for any object $V$ of $\Sm_K(\Sy_{\Psi})$, 
$N_sV$ consists of the elements $v\in V$ such that $\deg P_{K\langle\Sy_{\Psi}\rangle v}\le s$. 
In particular, the level filtration 
\begin{itemize} \item on $\Sm_K(\Sy_{\Psi})$ is strictly compatible with injections, i.e. 
$N_sV=V\cap N_sW$ for any inclusion $V\subset W$ in $\Sm_K(\Sy_{\Psi})\ ($so that $N_sE(V)=0$ whenever $N_sV=0);$ 
\item is strictly compatible 
with $\Sy_{\Psi}$-field extensions $\wK|K$: $N_s(\wK\otimes_KV)=\wK\otimes_KN_sV$. 
\end{itemize} 

$\Sm_K^{\le s}(\Sy_{\Psi})$ is a cocomplete Serre subcategory of $\Sm_K(\Sy_{\Psi})$, 
non-complete if $K$ is non-trivial. \end{corollary} 

\begin{proof} We may assume that $V$ is finitely generated, so, for a finite $J\subset\Psi$, its 
restriction to $\Sy_{\Psi|J}$ splits as $\bigoplus_jK\langle\binom{\Psi\smallsetminus J}{j}\rangle^{\kappa_j}$. 
By Lemma~\ref{growth}, any element outside of 
$\bigoplus_{j\le s}K\langle\binom{\Psi\smallsetminus J}{j}\rangle^{\kappa_j}$ generates a 
$K\langle\Sy_{\Psi\smallsetminus J}\rangle$-submodule that grows as a polynomial of degree $>s$. 

By Corollary~\ref{additivity}, the polynomial $P$ is additive on exact sequences, so (i) 
the level of a subquotient of $V$ does not exceed the level of $V$; (ii) if $V$ is an extension of $V_1$ 
by $V_2$ then the level of $V$ coincides with the maximum of the levels of $V_1$ and $V_2$. 
\end{proof}

\subsection{\texorpdfstring{$K_0$}{K0} of the subcategory 
of compact objects in \texorpdfstring{$\Sm_K(\Sy_{\Psi})$}{}} 
For a Roelcke precompact $G$, denote by $\Sm_K^{\mathrm{fp}}(G)$ the full subcategory of compact 
objects in $\Sm_K(G)$. 

For each extension $L|K$ of smooth $G$-fields, the faithful functor 
$L\otimes_K(-)\colon\Sm_K^{\mathrm{fp}}(G)\to\Sm_L^{\mathrm{fp}}(G)$ induces a $\lambda$-ring 
homomorphism $K_0(\Sm_K^{\mathrm{fp}}(G))\to K_0(\Sm_L^{\mathrm{fp}}(G))$. 
This gives rise to a functor $\beta_G\colon K\mapsto K_0(\Sm_K^{\mathrm{fp}}(G))$ from the category 
$\mathrm{SmFields}(G)$ of smooth $G$-fields to the category of $\lambda$-rings.

Denote by $\Lambda_{{\mathbb Z}}$ the category of $\lambda$-rings endowed with a surjective $\lambda$-ring 
homomorphism onto the binomial ring $\mathrm{Int}(\mathbb Z)\subset\mathbb Q[X]$ of integer-valued 
polynomials in one variable. Its final object is the identity automorphism of $\mathrm{Int}(\mathbb Z)$. 
Let $\mathrm{SmFields}^+(G)$ be the subcategory of non-trivial $G$-fields. 

The following result generalizes \cite[Theorem 4.18]{NagpalSnowden}. 
\begin{theorem} \label{lambda-homomorphism} The functor $\beta_{\Sy_{\Psi}}$ factors through a functor 
$\mathrm{SmFields}^+(\Sy_{\Psi})\xrightarrow{\xi}\Lambda_{{\mathbb Z}}$, sending to the final object the fields 
$F_{k,\Psi}$ for all non-trivial regular field extensions $F|k$. \end{theorem} 
\begin{proof} By Corollary~\ref{admiss}, for any smooth non-trivial $\Sy_{\Psi}$-field $K$ and each object 
$V$ of $\Sm_K^{\mathrm{fp}}(\Sy_{\Psi})$, there is a unique polynomial $P_V(X)\in\mathbb Q[X]$ such that 
$P_V(\#T)=\dim_{K^{\Sy_{\Psi|T}}}V^{\Sy_{\Psi|T}}$ for any sufficiently large finite subset $T\subset\Psi$. 

By Corollary~\ref{additivity}, $P_{V_2}=P_{V_1}+P_{V_3}$ for any short exact sequence 
$0\to V_3\to V_2\to V_1\to 0$ in $\Sm_K^{\mathrm{fp}}(\Sy_{\Psi})$, and thus, $V\mapsto P_V$ 
induces an additive homomorphism $\varphi_K$ from $K_0(\Sm_K^{\mathrm{fp}}(\Sy_{\Psi}))$. 

As the polynomials $P_{K\langle\binom{\Psi}{s}\rangle}(X)=\binom{X}{s}$ are linearly independent for all 
$s\ge 0$, the classes of $K\langle\binom{\Psi}{s}\rangle$ are linearly independent as well. As the polynomials 
$\binom{X}{s}$ additively generate the target, $\varphi_K$ is surjective. 

To check the multiplicativity and the compatibity with $\lambda$-structures, it suffices to notice that 
$(V_1\otimes_KV_2)^{\Sy_{\Psi|J}}=V_1^{\Sy_{\Psi|J}}\otimes_{K^{\Sy_{\Psi|J}}}V_2^{\Sy_{\Psi|J}}$ 
and $(\bigwedge^r_KV)^{\Sy_{\Psi|J}}=\bigwedge^r_{K^{\Sy_{\Psi|J}}}V^{\Sy_{\Psi|J}}$ 
for all sufficiently big $J$. Namely, by Proposition~\ref{local-structure}, this is 
reduced to the case of $V_i=K\langle\binom{\Psi}{s_i}\rangle$, which is clear. 

If $K=F_{k,\Psi}$ then, as in \cite{NagpalSnowden}, $V$ admits a finite 
resolution by finite direct sums of $K\langle\binom{\Psi}{s}\rangle$'s, so the classes of 
$K\langle\binom{\Psi}{s}\rangle$'s form an additive basis of $K_0(\Sm_K^{\mathrm{fp}}(\Sy_{\Psi}))$, 
and thus, $\varphi_K$ is bijective. \end{proof}

\begin{example} According to Example~\ref{fin-ind-ex}, the classes of $\Ka\langle\binom{\Psi}{s}\rangle$ for all 
$s\ge 0$ generate a subring of $K_0(\Sm_{\Ka}^{\mathrm{fp}}(\Sy_{\Psi}))$ isomorphic to $\mathrm{Int}(\mathbb Z)$; 
the classes of $x^{\lambda}\Ka$ for all $\lambda\in\Xi$ generate a subring isomorphic to the group ring 
$\mathbb Z[\mathrm{Pic}_{\Ka}]$. Let $\varepsilon\colon\mathbb Z[\mathrm{Pic}_{\Ka}]\to\mathbb Z$ be 
given by $\sum_im_i[\mathcal L_i]\mapsto\sum_im_i$. By Lemma~\ref{fixed_basis}, 
$([\Ka]-[x^{\lambda}\Ka])[\Ka\langle\binom{\Psi}{s}\rangle]=0$ for all $s>0$, 
so $K_0(\Sm_{\Ka}^{\mathrm{fp}}(\Sy_{\Psi}))$ is isomorphic to the ring 
$\mathrm{Int}(\mathbb Z)\oplus\ker\varepsilon$ with the multiplication given by 
$(P,\alpha)(Q,\beta)=(PQ,Q(0)\alpha+P(0)\beta+\alpha\beta)$. \end{example}

\subsection{Injective objects and resolutions} \label{Injective_objects} 
\begin{proposition} \label{standard_sheaf_resolution} Let $\mathcal O$ be a non-constant sheaf of fields on 
$\mathrm{FI}^{\mathrm{op}}$. Then, for any $\mathcal O$-sheaf of finite type $W$, 
there is a finite set $J$ and a resolution of the form 
\begin{equation} \label{alm_stand_sheaf_resol} 0\to\mathcal O_{+J}\otimes_{\mathcal O}W\to
\bigoplus_{s=0}^N\bar{h}^s_{\mathcal O_{+J}}{}^{\kappa_{0,s}}\to
\bigoplus_{s=0}^{N-1}\bar{h}^s_{\mathcal O_{+J}}{}^{\kappa_{1,s}}
\to\cdots\to\bigoplus_{s=0}^1\bar{h}^s_{\mathcal O_{+J}}{}^{\kappa_{N-1,s}}
\to\mathcal O_{+J}^{\kappa_N}\to 0\end{equation} 
for some integer $\kappa_{ij}\ge 0$, where $N$ is the level of $W$. \end{proposition} 
\begin{proof} We construct recursively an ascending sequence of finite sets 
$\varnothing=:J_{-1}\subseteq J_0\subseteq J_1\subseteq\cdots\subseteq J_N=:J$ together 
with a sequence of $\mathcal O_{+J_{i-1}}$-modules $W_i$ of level $\le N-i$, where $W_0:=W$. 

Choose $J_i$ so that $(W_i)_{+(J_i\smallsetminus J_{i-1})}$ is as in 
Proposition~\ref{local-structure}. Define $W_{i+1}$ as the cokernel of the morphism 
$\mathcal O_{+J_i}\otimes_{\mathcal O_{+J_{i-1}}}W_i\hookrightarrow(W_i)_{+(J_i\smallsetminus J_{i-1})}$. 
Then $W_{i+1}$ is of level $<N-i$ and the sequence 
\[0\to\mathcal O_{+J}\otimes_{\mathcal O_{+J_{i-1}}}W_i\hookrightarrow\mathcal O_{+J}\otimes_{\mathcal O_{+J_i}}
(W_i)_{+(J_i\smallsetminus J_{i-1})}\to\mathcal O_{+J}\otimes_{\mathcal O_{+J_i}}W_{i+1}\to 0\] 
is exact. Combining all these short exact sequences we get a resolution 
\begin{multline*} 0\to\mathcal O_{+J}\otimes_{\mathcal O_{+J_{-1}}}W_0\to\mathcal O_{+J}\otimes_{\mathcal O_{+J_0}}
(W_0)_{+(J_0\smallsetminus J_{-1})}\to\mathcal O_{+J}\otimes_{\mathcal O_{+J_2}}
(W_1)_{+(J_1\smallsetminus J_0)}\to\cdots\\ 
\cdots\to\mathcal O_{+J}\otimes_{\mathcal O_{+J_{N-1}}}
(W_{N-1})_{+(J_{N-1}\smallsetminus J_{N-2})}\to\mathcal O_{+J}\otimes_{\mathcal O_{+J_N}}
(W_N)_{+(J_N\smallsetminus J_{N-1})}\to 0\end{multline*} 
of the required type (\ref{alm_stand_sheaf_resol}). \end{proof} 

\begin{lemma} \label{simples_closed-pts_cogenerators} Let $C$ be a collection of objects in 
an abelian category $\mathcal A$ such that {\rm (i)} $\End(P)$ is a division ring for all $P\in C$, 
{\rm (ii)} $\Hom(P,P')=0$ for all $P\neq P'$ in $C$. \begin{enumerate} 
\item \label{simple_cogenerators} Suppose in addition that $C$ is cogenerating. 
Then all objects in $C$ are injective. 
\item \label{cogenerators_for_fin_gen} Suppose in addition that $\mathcal A$ is a locally finitely generated 
Grothendieck category, while any finitely generated object of $\mathcal A$ admits an embedding into 
a product of objects in $C$. Then each finitely generated object in $C$ is injective. If all objects in $C$ 
are finitely generated then $C$ cogenerates $\mathcal A$. \end{enumerate} \end{lemma} 
\begin{proof} Assuming that $C$ is cogenerating, if $P_0\in C$ and $i\colon P_0\to V$ is a monomorphism 
then there exists a morphism $f\colon V\to P$ for some $P\in C$ such that $f\circ i\colon P_0\to P$ is non-zero, 
so $P=P_0$ and $(f\circ i)^{-1}\circ f\colon V\to P_0$ is a splitting of $i$, i.e., $P_0$ is injective. 

Assuming that $\mathcal A$ is locally finitely generated, if $P_0\in C$ is finitely generated and admits a 
non-trivial essential extension\footnote{Recall, that an injection $M\hookrightarrow N$ in an abelian category 
is called an {\sl essential extension} if any non-zero subobject of $N$ has a non-zero intersection with the 
image of $M$, cf. \cite[Ch. 6, \S2]{BucurDeleanu}.} $E$ then $E$ can be chosen to be finitely generated as well. 
Any morphism $E\to P$ for $P\neq P_0$ in $C$ vanishes on $P_0$, so there exists a morphism $f\colon E\to P_0$ 
non-vanishing on $P_0$. But $\ker f\cap P_0=0$, so $\ker f=0$, i.e. $E=P_0$, and thus, $P_0$ is injective. Now, 
if all objects in $C$ are finitely generated then they are injective, so for any object $W$ and each 
finitely generated $W_0\subseteq W$ there is an embedding $W_0\hookrightarrow\prod_{i\in I(W_0)}Q_{W_0}(i)$ 
into a product of objects $Q_{W_0}(i)$ in $C$ indexed by a set $I(W_0)$, while this embedding extends to 
a morphism $f_{W_0}\colon W\to\prod_{i\in I(W_0)}Q_{W_0}(i)$, thus giving an embedding 
$(f_{W_0})\colon W\to\prod_{W_0}\prod_{i\in I(W_0)}Q_{W_0}(i)$. \end{proof} 

\begin{lemma} \label{injectivity-K-condition} Let $\mathcal O$ be a non-constant sheaf of fields, and $\mathcal V$ 
be a level 0 simple $\mathcal O$-module. Then $\mathcal V$ is injective if and only if, for any finite set $J$, 
any $\mathcal O$-module embedding $\iota\colon\mathcal V\hookrightarrow\mathcal O_{+J}$ is split. \end{lemma} 
\begin{proof} The `only if' direction is trivial. In the opposite direction, let an $\mathcal O$-module 
$\mathcal E$ be an essential extension of $\mathcal V$. As we are going to figure out whether 
$\mathcal E$ can be non-trivial, we may assume that $\mathcal E$ is finitely generated. By Proposition 
\ref{local-structure}, there is a finite set $J$ and an isomorphism of $\mathcal O_{+J}$-modules 
$\mathcal E_{+J}\xrightarrow{\sim}\bigoplus_{s=0}^N(\bar{h}^s_{\mathcal O_{+J}})^{\kappa_s}$ 
for some integer $N,\kappa_0,\dots,\kappa_N\ge 0$. Then $\mathcal E$ is a finitely generated 
$\mathcal O$-submodule of $\bigoplus_{s=0}^N(\bar{h}^s_{\mathcal O_{+J}})^{\kappa_s}$. Fix 
a non-zero $\mathcal O_{+J}$-morphism $\lambda\colon\mathcal E_{+J}\to\mathcal O_{+J}$ (e.g. $\lambda$ 
can be chosen to be the composition of the natural projection $\mathcal E_{+J}\xrightarrow{\sim}
\bigoplus_{s=0}^N(\bar{h}^s_{\mathcal O_{+J}})^{\kappa_s}\to\mathcal O_{+J}^{\kappa_0}$ 
with any $\mathcal O_{+J}$-morphism $\mathcal O_{+J}^{\kappa_0}\to\mathcal O_{+J}$ that is 
non-zero on the image of $\mathcal V\subset\mathcal E$; by Lemma~\ref{no-finite-dim-submod}, 
$\bar{h}^s_{\mathcal O_{+J}}$ does not contain level 0 $\mathcal O$-submodules if $s>0$). 
Then $\lambda$ embeds $\mathcal E$ into $\mathcal O_{+J}$. By our assumption, we can choose 
a morphism of $\mathcal O$-modules $\xi\colon\mathcal O_{+J}\to\mathcal V$ such that 
the composition $\xi\circ\lambda\circ\iota$ splits the embedding $\iota$. \end{proof} 

\begin{lemma} \label{some-injective} Let $G$ be a permutation group, $\Psi$ be a smooth $G$-set, 
$K$ be a smooth $G$-field. Let $V$ be an object of $\Sm_K(G)$ such that $V\otimes_KK(\Psi)$ 
is injective and at least one of $K$ and $V$ is injective as well. Then the object 
$V\langle\binom{\Psi}{s}\rangle:=V\otimes_KK\langle\binom{\Psi}{s}\rangle$ 
is injective for any integer $s\ge 0$. \end{lemma}  
\begin{proof} Let $\wK\subset K(\Psi)$ be the subfield generated over $K$ by squares of the elements of 
$\Psi$. There is an isomorphism $\bigoplus_{s\ge 0}\wK\langle\binom{\Psi}{s}\rangle\xrightarrow{\sim}K(\Psi)$, 
$[S]\mapsto\wK\cdot\prod_{t\in S}t$, so each $\wK\langle\binom{\Psi}{s}\rangle$ is isomorphic 
to a direct summand of the object $K(\Psi)$ of $\Sm_{\wK}(G)$. 

As $\wK$ and $K(\Psi)$ are isomorphic $G$-field extensions of $K$ (under $t^2\mapsto t$ for all $t\in\Psi$), 
they are isomorphic as objects of $\Sm_K(G)$, as well as $V\otimes_K\wK$ and $V\otimes_KK(\Psi)$ are. 

We assume that $V\otimes_KK(\Psi)$ is injective in $\Sm_K(G)$, so each 
$V\otimes_K\wK\langle\binom{\Psi}{s}\rangle$ is an injective object of $\Sm_K(G)$. As we also assume 
that either $K$ or $V$ is injective, the inclusion $V\stackrel{\otimes 1}{\hookrightarrow}V\otimes_K\wK$ 
admits a splitting $V\otimes_K\wK\xrightarrow{\pi}V$. Then the inclusion 
$V\otimes_KK\langle\binom{\Psi}{s}\rangle\hookrightarrow V\otimes_K\wK\langle\binom{\Psi}{s}\rangle$ 
splits as well: $\sum_iv_i\otimes a_i[S_i]\mapsto\sum_i\pi(v_i\otimes a_i)[S_i]$, 
and thus, $V\langle\binom{\Psi}{s}\rangle$ is an injective object of $\Sm_K(G)$. \end{proof}

\subsection{Existence of weak period extensions} 
It is not clear so far, whether there exist $G$-period extensions (Definition~\ref{def_period_field}). 
However, weak $\Sy_{\Psi}$-period extensions (Definition~\ref{def_period_field}) do exist. 

\begin{proposition} \label{weak_period} Let $G$ be a permutation group, $K$ be a smooth $G$-field. 
Suppose that, for any smooth $G$-field extension $L|K$, the category $\Sm_L(G)$ is locally noetherian. 
Then there exists a weak period field over $K$. 

If $G=\Sy_{\Psi}$ then there are functorial smooth $G$-field extensions of the non-trivial smooth 
$G$-fields. Namely, for a fixed infinite set, say $\mathbb N$, and each non-constant sheaf of fields $\mathcal O$ 
on $\mathrm{FI}^{\mathrm{op}}$, the colimit $\widetilde{\mathcal O}:=\indlim_{J\subset\mathbb N}\mathcal O_{+J}$ 
over the finite subsets $J$ is a weak period sheaf over $\mathcal O$. \end{proposition} 
\begin{proof} Let $L_0:=K\subset L_1\subset L_2\subset\dots$ be a tower of smooth $G$-field extensions, and 
$\widetilde{K}:=\bigcup_iL_i$. As $\Sm_{\widetilde{K}}(G)$ is locally noetherian, any finitely generated 
object $W$ of $\Sm_{\widetilde{K}}(G)$ is the cokernel of a morphism 
$\widetilde{K}\langle G/U\rangle^m\to\widetilde{K}\langle G/U\rangle^n$ for some $m,n$ and an open 
subgroup $U\subset G$, so $W=\widetilde{K}\otimes_{L_i}W_0$ for some $i$ and $W_0\in\Sm_{L_i}(G)$. By 
Lemma~\ref{simples_closed-pts_cogenerators}~(\ref{cogenerators_for_fin_gen}), to show that $\widetilde{K}$ 
is a cogenerator of $\Sm_{\widetilde{K}}(G)$ it suffices to check that any such $W$ can be embedded into 
a product of copies of $\widetilde{K}$. 

For each $i\ge 0$, choose recursively $L_{i+1}$ so that any finitely generated object of $\Sm_{L_i}(G)$ can 
be embedded into $L_{i+1}$. (The isomorphism classes of the finitely generated objects of $\Sm_{L_i}(G)$ 
form a set, so we can define $L_{i+1}$ as the fraction field of the symmetric $L_i$-algebra of the direct 
sum of representatives of all such isomorphism classes.) Then, for each $j\ge i$, there is an embedding 
$\alpha_j\colon L_j\otimes_{L_i}V_0\hookrightarrow L_{j+1}$. Any element $\xi$ in the kernel of 
$id_{\widetilde{K}}\cdot\alpha_j\colon W\to\widetilde{K}$ belongs, in fact, to $L_s\otimes_{L_i}W_0$ 
for some $s\ge j$, so $(id_{\widetilde{K}}\cdot\alpha_s)(\xi)\neq 0$. Thus, the map 
$\alpha\colon W\xrightarrow{(id_{\widetilde{K}}\cdot\alpha_j)_{j\ge i}}\prod_{j\ge i}\widetilde{K}$ is injective. 

In the case $G=\Sy_{\Psi}$, let the field $L_j$ correspond to the $\mathcal O$-module 
$\mathcal O_{+\{1,\dots,j\}}$ under the equivalence of Lemma~\ref{iso-restr}. By \cite[Theorem 3.18]{H90}, 
any finitely generated $\widetilde{\mathcal O}$-module $W$ is noetherian, so as shown above, 
$W=\widetilde{\mathcal O}\otimes_{\mathcal O_{+J}}W_0$ for some finite set $J$ and 
an $\mathcal O_{+J}$-module $W_0$. It suffices to check that $W$ can be embedded into a product 
of copies of $\widetilde{\mathcal O}$. 

By Proposition~\ref{local-structure}, there is a finite set $J'\supseteq J$ such that 
the $\mathcal O_{+J'}$-module $W_1:=(W_0)_{+(J'\smallsetminus J)}$ is a sum of sheaves 
$\bar{h}^s_{\mathcal O_{+J'}}$. By Lemmas~\ref{embedding_into_prod} and \ref{open-subgrps-descr}, 
$\bar{h}^s_{\mathcal O_{+J'}}$ embeds into a product of copies of $\mathcal O_{+J'}$ for sufficiently 
big $J'$. As $\mathcal O_{+J'}\otimes_{\mathcal O_{+J}}W_0$ is embedded into $W_1$, the module 
$W=\widetilde{\mathcal O}\otimes_{\mathcal O_{+J'}}(\mathcal O_{+J'}\otimes_{\mathcal O_{+J}}W_0)$ 
is embedded into $\widetilde{\mathcal O}\otimes_{\mathcal O_{+J'}}W_1$, 
which is a sum of sheaves $\bar{h}^s_{\widetilde{\mathcal O}}$. \end{proof}

\subsubsection{An example of non-injective tensor units of type \texorpdfstring{$F_S$}{} in positive characteristic} 
\begin{example} Let $F|k$ be a non-trivial regular field extension, 
$s\ge 1$ be an integer, $k':=F_{\mathbb N}$ (as in \S\ref{NotaF}), 
and $\widetilde{F}:=F_{\binom{\Psi}{s}\sqcup\bigsqcup_{j=1}^{s-1}\binom{\Psi}{j}\times\mathbb N}$. 
The functor of Proposition~\ref{weak_period} produces the fraction field $\widetilde{F}'$ of 
$\widetilde{F}\otimes_kk'$ out of $F_{\binom{\Psi}{s}}$, so $\widetilde{F}'$ is a cogenerator of 
the category $\Sm_{\widetilde{F}'}(\Sy_{\Psi})$. \end{example} 

\begin{proposition} \label{no-incr-morphisms} Let $F'|k$ and $F|k$ be regular field extensions of characteristic 
$p>0$, $F'\neq k$. Let $s\ge p>t\ge 1$ be integers. Then {\rm (i)} any morphism of $\Sy_{\Psi}$-modules 
$\varphi\colon F'_{\binom{\Psi}{t}}\to F_{\binom{\Psi}{s}}$ vanishes at $1;$ {\rm (ii)} the object 
$F_{\binom{\Psi}{s}}$ is not injective in $\Sm_{F_{\binom{\Psi}{s}}}(\Sy_{\Psi})$. \end{proposition} 
\begin{proof} Fix some $x\in F'\smallsetminus k$ and $J\in\binom{\Psi}{u}$ for some 
integer $t<u\le s$. For each $I\in\binom{\Psi}{t}$, denote by $x_I$ the decomposable tensor in 
$\bigotimes_{k,\binom{\Psi}{t}}F'\subset F'_{\binom{\Psi}{t}}$ whose only non-identical 
tensor multiple distinct from 1 is $x$ on the $I$-th place. For such $I$ with $I\subset J$, let 
$a_I:=\varphi\left(\frac{x_I}{\sum_{I'\subset J}x_{I'}}\right)\in F_{\binom{\Psi}{s}}$. Then 
$a_I\in F_{\binom{\Psi}{s}}^{\Sy_{\Psi|J}}=F_{\binom{J}{s}}$, so either $a_I\in k$ or $u=s$ and 
$a_I$ is a decomposable tensor in $\bigotimes_{k,\binom{\Psi}{s}}F\subset F_{\binom{\Psi}{s}}$ 
whose only non-identical tensor multiple is on the $J$-th place. 

As $F_{\binom{\Psi}{s}}^{\Sy_{\Psi|J}}=F_{\binom{\Psi}{s}}^{\Sy_{\Psi,J}}$, $a_I=a$ is independent of $I$, and 
thus $\sum_{I\subset J}\frac{x_I}{\sum_{I'\subset J}x_{I'}}\in F'_{\binom{\Psi}{t}}$ is sent to $\binom{u}{t}a$. 
Taking $u=p$, we get that 1 is sent to 0.

As any morphism $F'_{\Psi}\to F_{\binom{\Psi}{s}}$ in $\Sm_k(\Sy_{\Psi})$ vanishes at 1, the restriction 
of any morphism $\xi\in\Hom_{F_{\binom{\Psi}{s}}}(F_{\binom{\Psi}{s}}\otimes_kF'_{\Psi},F_{\binom{\Psi}{s}})$ 
to $k\otimes_kF'_{\Psi}$ vanishes at 1, so the restriction of $\xi$ to 
$F_{\binom{\Psi}{s}}\otimes_kk=F_{\binom{\Psi}{s}}$ cannot be identical, and thus, the natural embedding 
$F_{\binom{\Psi}{s}}\xrightarrow{(-)\otimes 1}F_{\binom{\Psi}{s}}\otimes_kF'_{\Psi}$ in 
$\Sm_{F_{\binom{\Psi}{s}}}(\Sy_{\Psi})$ admits no splitting, which means that the object 
$F_{\binom{\Psi}{s}}$ of $\Sm_{F_{\binom{\Psi}{s}}}(\Sy_{\Psi})$ is not injective. \end{proof} 

\begin{remarks} 1. Now \cite[Lemma~A.10]{H90} implies that the object $k$ of 
the category $\Sm_k(\Sy_{\Psi})$ is injective if and only if the characteristic of $k$ is $0$. 

2. For any tower of fields $k\subset F\subseteq\widetilde{F}$ with $k$ algebraically closed 
in $\widetilde{F}$ and any pair of integers $s>t\ge 1$, there are no $\Sy_{\Psi}$-field embeddings of 
$F_{\binom{\Psi}{s}}$ into $\widetilde{F}_{\binom{\Psi}{t}}$ over $k$. [If there is one, 
we may restrict it to $F'_{\binom{\Psi}{s}}$ for a subfield $k\neq F'=k(t)$. Then, we may choose 
a subfield $\widetilde{F}'\subseteq\widetilde{F}$ finitely generated over $k$ such that, 
for some $I\in\binom{\Psi}{s}$, ${F'_{\binom{\Psi}{s}}}^{\Sy_{\Psi|I}}=F'$ is embedded into 
$\widetilde{F}'_{\binom{\Psi}{t}}{}^{\Sy_{\Psi|I}}=\mathrm{Frac}(\widetilde{F}'{}^{\otimes_k^{\binom{I}{t}}})$. 
But $\mathrm{tr.deg}(F'_{\binom{\Psi}{s}}{}^{\Sy_{\Psi|J}}|k)=\binom{\#J}{s}$ for 
any finite $J\subset\Psi$, while $\mathrm{tr.deg}(\widetilde{F}'_{\binom{\Psi}{t}}{}^{\Sy_{\Psi|J}}|k)
=\binom{\#J}{t}\cdot\mathrm{tr.deg}(\widetilde{F}'|k)$ is less than 
$\binom{\#J}{s}$ for sufficiently large $J$.] \end{remarks}

\section{The spectrum of \texorpdfstring{$\Sm_K(\Sy_{\Psi})$}{} for 
\texorpdfstring{$K=\Ka,\Kb,\Kc$}{K=Ka,Kb,Kc}} \setcounter{subsection}{-1}
\subsection{Topology of the Gabriel spectra} 

Given a Grothendieck category $\mathcal A$, any proper closed subset of the Gabriel spectrum 
$\mathrm{Spec}(\mathcal A)$ of $\mathcal A$ is an intersection of sets $[X]^c:=\{P~|~\Hom(X,P)\neq 0\}$ 
for a collection of compact objects $X$. As $[X]^c\cup[Y]^c=[X\oplus Y]^c$, these sets are closed under 
finite union, so an arbitrary closed set will have the form $\bigcap_i[X_i]^c$ with some compact $X_i$'s. 
\begin{lemma} \label{rmks_on_Gabriel_topology} \begin{enumerate} \item Let $X,Y$ be compact objects of $\mathcal A$ 
such that there is either an epimorphism $X\to Y$ or a monomorphism $Y\to X$. Then $[Y]^c\subseteq[X]^c$. 
\item \label{closure_of_a_pt} If $\mathcal A$ is locally noetherian then the closure $\overline{\{P\}}$ 
of a point $P$ consists of the points $P'$ such that $P$ admits a monomorphism into a product of 
injectives presenting $P'$. 
\item \label{density_generators_no_Hom} Let $T$ be a collection of compact generators 
of $\mathcal A$, and $S$ be a subset in $\mathrm{Spec}(\mathcal A)$ such that for any 
$W\in T$ there exists $P\in S$ with $\Hom(W,P)=0$. Then $S$ is dense. 

\item For any simple object of $\mathcal A$, the isomorphism class of its injective hull is a closed point. 

\end{enumerate} \end{lemma} 
\begin{proof} Let $S$ be a non-dense subset. As the closure $\overline{S}$ of $S$ is the intersection 
of all proper closed subsets containing $S$, one has \[\overline{S}=\bigcap_{X:\text{ compact, 
$\Hom(X,P)\neq 0$ for all $P\in S$}}[X]^c
=\bigcap_{X\subset\prod_{P\in S}P:\text{ compact, $\mathrm{pr}_P(X)\neq 0$ for all $P\in S$}}[X]^c.\] 

If $S=\{P\}$ consists of a single point $P$ then we get 
$\overline{\{P\}}=\bigcap_{0\neq X\subset P:\text{ compact}}[X]^c$. 
If $P'\in\overline{\{P\}}$ then for any non-zero subobject $X\subset P$ there is a non-zero morphism 
$X\to P'$, so extending it to $P$ we see that $P$ admits a monomorphism into a product of copies of $P'$. 
Conversely, if $P$ admits an embedding into a product of copies of $P'$ then any non-zero subobject 
$X\subset P$ admits a non-zero morphism to $P'$. 

If $P$ is simple then $[P]^c$ is the class of $E(P)$, i.e. the class of $E(P)$ is a closed point. \end{proof}

\begin{lemma} \label{closed-pts_bd} Let $K$ be a smooth $\Sy_{\Psi}$-field, 
and $\mathrm{Spec}_K=\mathrm{Spec}(\Sm_K(\Sy_{\Psi}))$. 
\begin{enumerate} \item \label{density_of_P_n} 
Any proper closed subset in $\mathrm{Spec}_K$ consists of injective hulls of objects of bounded level.

In particular, any infinite subset of $\{P_1^{(K)},P_2^{(K)},\dots\}$ is dense in $\mathrm{Spec}_K$, 
where $P_i^{(K)}$ is the isomorphism class of $E(K\langle\binom{\Psi}{i}\rangle)$. 
\item Let $Q=E(Y)$ be a point of $\mathrm{Spec}_K$ for an object $Y$ of level $\le n$. Then the closure of $Q$ 
does not contain points $P=E(X)$ such that $N_nX=0$. In particular, $\overline{\{P_i\}}\not\ni P_j$ for $j>i$. 
\item In notation of Theorem~\ref{spectrum-triple-rough}, 
$\overline{\{P_n^{(K_?)}\}}\supseteq\{P_0^{(K_?)},P_1^{(K_?)},\dots,P_n^{(K_?)}\}$ for any $n\ge 0$ 
with the following exceptions, where $F=k:$ $?=\mathrm{a,b,c}$, $n=1$, and $\Gamma=0$ if $?=\mathrm{a};$ 
$?=\mathrm{d}$ and $n=2$. 
\item The points $P_1^{(\Ka)}$ and $P_{1,m}^{(\Kc)}$ for all $m\in\mathbb Z$ are closed. 
\item For $?\in\{\mathrm{a},\mathrm{c}\}$, the closure of $E(\mathrm{Pic}_{K_?})$ 
contains neither $P_s$ for $s>1$, nor $P'_2$ if $?=\mathrm{c}$ and $p\neq 2$, nor $P_{1,n}$ for 
$n\in\mathbb Z$ if $?=\mathrm{c}$. \end{enumerate} \end{lemma} 
\begin{proof} For any compact object $X$ of $\Sm_K(\Sy_{\Psi})$, there is an integer 
$n\ge 0$ such that $X=N_nX$. Then $\Hom(X,E(Y))=0$ for all $Y$ such that $N_nY=0$, in particular, 
$\Hom(X,P_j)=0$ for all $j>n$, i.e. $[X]^c$ contains only finitely many $P_s$'s. This implies that any closed 
set containing infinitely many $P_s$'s is not proper, i.e. coincides with $\mathrm{Spec}_K$. 

To describe closures of points, we apply Lemmas~\ref{rmks_on_Gabriel_topology} (\ref{closure_of_a_pt}) 
and \ref{embedding_into_prod} to $U=\Sy_{\Psi,I}$, $V=K\langle\binom{\Psi}{s}\rangle$ for $s<\#I$, and 
$V\xrightarrow{\lambda}K$, $\sum_{J\in\binom{\Psi}{s}}a_J[J]\mapsto\sum_{J\in\binom{\Psi}{s}}a_J$. 
As the subgroups $\Sy_{\Psi,I}$ are maximal, we have only to check that $K^{\Sy_{\Psi,I}}\neq K^{\Sy_{\Psi}}$ 
and that there is an element $A\in K_I\langle\binom{I}{s}\rangle^{\Sy_I}$ with $\lambda(A)=1$. The 
$\Sy_I$-action on $K_?^{\Sy_{\Psi|I}}$ is non-trivial, so $K^{\Sy_{\Psi,I}}\neq K^{\Sy_{\Psi}}$, 
while in the exceptional cases $K_?^{\Sy_{\Psi,I}}=k=K_?^{\Sy_{\Psi}}$. 
Then the composition of $K^{\Sy_{\Psi,I,J}}=K_I^{\Sy_{I,J}}\xrightarrow{\sim}V^U$, 
$a\mapsto A_a:=\sum_{g\in\Sy_I/\Sy_{I,J}}a^g[gJ]\in V^U$, with $\lambda$ is the trace
$\mathrm{tr}_{K_I^{\Sy_{I,J}}|K_I^{\Sy_I}}:K_I^{\Sy_{I,J}}\to K_I^{\Sy_I}$. As the field 
extension $K_I^{\Sy_{I,J}}|K_I^{\Sy_I}$ is separable, there is an element $a$ with 
$\mathrm{tr}_{K_I^{\Sy_{I,J}}|K_I^{\Sy_I}}(a)=1$, and then $\lambda(A_a)=1$.

As the $k$-vector space $\Hom((x-y)^mK_?\langle\Psi\rangle,(x-y)^nE(K_?))=
\Hom(K_?\langle\Psi\rangle,(x-y)^{n-m}E(K_?))$ is at most one-dimensional and $E(K_?)$ is injective, 
$\Hom(\langle(x-y)^{2m-n}(y^{n-m}[x]-x^{n-m}[y])\rangle,E(K_?))=0$. 

The $k$-vector space $\Hom((x-y)^m\Kc\langle\Psi\rangle,(x-y)^nE(\Kc))=
\Hom(\Kc\langle\Psi\rangle,(x-y)^{n-m}E(\Kc))$ is one-dimensional if $F=k$ and $n\ge m$; 
it vanishes if $F=k$ and $n<m$. 
As $E(\Kc)$ is injective, $\Hom(\langle(x-y)^{2m-n}(y^{n-m}[x]-x^{n-m}[y])\rangle,E(\Kc))=0$, 
i.e. the $\mathrm{Pic}_{\Kc}$-orbit of $P_0$ does not meet the closure of $P_{1,m}$. 
The remaining points are not in the closure of $P_{1,m}$ for trivial reasons.

For any $m\in\mathbb Z$, $\Hom(P_{1,m},P_{0,n})\neq 0$ for all $n\in\mathbb Z$, i.e. $P_{0,n}\in[P_{1,m}]^c$, 
while $\Hom(P_{1,m},P_s)=0$ for all $s>1$, i.e. $P_s\notin[P_{1,m}]^c$, meaning that the closure of 
$\{P_{0,n}~|~n\in\mathbb Z\}\subset\mathrm{Cl}_{K_?}$ contains neither $P_s$ for $s>1$. 

As $\Hom(P_1,P'_2)=0$, the closure of $\{P_{0,n}~|~n\in\mathbb Z\}\subset\mathrm{Cl}_{\Kc}$ does not 
contain $P'_2$ if $p\neq 2$. 

As $\Hom(P_{1,m},P_{1,n})=0$ for all $n\neq m$, the closure of 
$\{P_{0,n}~|~n\in\mathbb Z\}\subset\mathrm{Cl}_{\Kc}$ contains neither $P_{1,n}$ for $n\in\mathbb Z$. 
\end{proof} 

\subsection{A construction of \texorpdfstring{$G$}{G}-extensions with Picard group containing 
a given abelian group} 
Recall that, for a permutation group $G$ and a smooth $G$-field $K$, $\mathrm{Pic}_K(G)$ denotes 
the Picard group of $\Sm_K(G)$, i.e. $\mathrm{Pic}_K(G)=H^1_{\text{{\rm cont}}}(G,K^{\times})$. 

The group $\mathrm{Pic}_K(G)$ acts continuously on the Gabriel spectrum of $\Sm_K(G)$ by 
$\mathcal L\colon I\mapsto\mathcal L\otimes_KI$.

\begin{example}[The fields $\Ka=K_{\Psi,S,\Gamma}^L$] \label{torsion-Pic} 
Fix a quadruple $\mathrm{a}=(\Psi,S,\Gamma,L)$, where $\Psi$ and $S$ are sets, 
$L$ is a field, $\Gamma$ is a subgroup of the group $\Xi:=\mathbb Z\langle S\rangle$. 

Let $L(\Psi\times S)=L(S)_{L,\Psi}$ be the purely transcendental field extension of $L$ with a transcendence 
basis consisting of the variables labeled by the set $\Psi\times S$. Denote by $u_s$ the variable corresponding 
to $(u,s)\in\Psi\times S$. For each $s\in S$, sending $u\in\Psi$ to $u_s$ gives rise to a field embedding 
$L(\Psi)\hookrightarrow L(\Psi\times S)$, $f\mapsto f_s=f^{[s]}$. For all $f\in L(\Psi)^{\times}$ 
and $\gamma=\sum_sm_s[j_s]\in\Xi$, set $f^{\gamma}:=\prod_sf_{j_s}^{m_s}$. Then the embeddings 
$u\mapsto u_s$ give rise to the homomorphism $\Xi\otimes L(\Psi)^{\times}\to L(\Psi\times S)^{\times}$, 
$\gamma\otimes f\mapsto f^{\gamma}$. 

Define a subfield $\Ka=K_{\Psi,S,\Gamma}^L$ of $L(\Psi\times S)$ by 
\[\Ka:=L\left(u^{\gamma},\frac{u_s}{v_s}~|~\gamma\in\Gamma,~s\in S,~u,v\in\Psi\right)=
L\left(x^{\gamma},\frac{u_s}{x_s}~|~\gamma\in\Gamma,~s\in S,~u\in\Psi\right)\] for an arbitrary 
fixed $x\in\Psi$ (so $K_{\Psi,\varnothing,0}^L:=L$, $K_{\Psi,\{\ast\},0}^L=L(u/v~|~u,v\in\Psi)$). 

If $G$ is a group acting on the set $\Psi$ and on the field $L$ then 
\begin{itemize} \item $G$ acts on the fields $L(\Psi)$ and $L(\Psi\times S)$, while the above homomorphisms 
$L(\Psi)\stackrel{(-)_s}{\hookrightarrow}L(\Psi\times S)$ and 
$\Xi\otimes L(\Psi)^{\times}\to L(\Psi\times S)^{\times}$ are $G$-equivariant, \item $G$ preserves 
the one-dimensional $\Ka$-vector subspaces $x^{\lambda}\Ka\subseteq L(\Psi\times S)$ for all $\lambda\in\Xi$, 
\item $L(\Psi\times S)|\Ka$ is a $G$-field extension of degree equal to the index of $\Gamma$. \end{itemize} 
This shows that, for any non-precompact permutation group $G$, any smooth $G$-field $L$, 
any abelian group $\Lambda$, there exists a smooth $G$-field extension $\Ka|L$ and 
a group embedding $\Lambda\hookrightarrow\mathrm{Pic}_{\Ka}(G)$. \end{example}

\begin{lemma} \label{Simple-cogenerators-Ka} We keep notation of Example~\ref{torsion-Pic}. 
\begin{enumerate} \item The $\Ka\langle G\rangle$-module $L(\Psi\times S)$ embeds into 
$\prod_{\lambda\in\Xi}x^{\lambda}\Ka$ if $\Gamma$ is finitely generated. 
\item Assume that the $G$-action on $L(\Psi)$ is smooth. Then sending $\lambda$ to the 
class of $x^{\lambda}\Ka$ gives rise to a homomorphism $\Xi\to\mathrm{Pic}_{\Ka}(G)$ that 
factors through $\iota\colon\Xi/\Gamma\to\mathrm{Pic}_{\Ka}(G)$. If, moreover, all $G$-orbits 
on $\Psi$ are infinite then the homomorphism $\iota$ is injective. 
\item Assume that the $G$-action on $L(\Psi)$ is smooth, all $G$-orbits on $\Psi$ are infinite, 
$\Gamma$ is finitely generated, $\mathrm{Pic}_{L(\Psi\times S)}(G)=0$. Then \begin{itemize} \item the map 
$\iota\colon\Xi/\Gamma\to\mathrm{Pic}_{\Ka}(G)$, $\lambda\mapsto[x^{\lambda}\Ka]$, is an isomorphism, and 
\item the invertible subobjects $x^{\lambda}\Ka\subset L(\Psi\times S)$ form a system of cogenerators 
of $\Sm_{\Ka}(G)$ if $L(\Psi\times S)$ is a cogenerator of 
$\Sm_{L(\Psi\times S)}(G)$. \end{itemize} \end{enumerate} \end{lemma} 
\begin{proof} Suppose that $\Xi$ admits a basis $S'$ such that $\Gamma$ is generated by the elements $m(s)s$ 
for some $S'\xrightarrow{m}\mathbb Z$ and all $s\in S'$. (E.g. if $\Gamma$ is finitely generated then such an 
$S'$ is provided by the Smith normal form.) Let $S_0:=\{s\in S'~|~m(s)=0\}$ and $S_+:=S'\smallsetminus S_0$. 
Fix a total order on $S_0$, and order $\mathbb Z\langle S_0\rangle$ lexicographically. Then 
the Hahn $\mathbb Z\langle S_0\rangle$-power series over the field 
$\widetilde{\Ka}:=L(x_j,u_s/v_s~|~u,v\in\Psi,~s\in S,~j\in S_+)$ form a $G$-field 
$\widetilde{\Ka}((x^{\mathbb Z\langle S_0\rangle})):=\{\sum_{\lambda\in J}a_{\lambda}x^{\lambda}~|~
J\subset\mathbb Z\langle S_0\rangle\mbox{ is well-ordered, }a_{\lambda}\in\widetilde{\Ka}\}$ 
containing naturally $L(\Psi\times S)$. On the other hand, the $\widetilde{\Ka}\langle G\rangle$-module 
$\widetilde{\Ka}((x^{\mathbb Z\langle S_0\rangle}))$ is 
naturally embedded into $\prod_{\lambda\in\mathbb Z\langle S_0\rangle}x^{\lambda}\widetilde{\Ka}$, 
while $\widetilde{\Ka}=\bigoplus_{\xi\in\Gamma^{\perp}}x^{\xi}\Ka$, where 
$\Gamma^{\perp}=\{\sum_{s\in S_+}m_s[s]\in\mathbb Z\langle S_+\rangle~|~0\le m_s<|m(s)|\}$. 
This means that the $\Ka\langle G\rangle$-module $L(\Psi\times S)$ embeds into 
$\prod_{\lambda\in\Xi}x^{\lambda}\Ka$.

If $\mathrm{Pic}_{L(\Psi\times S)}(G)=0$ then any element of $\mathrm{Pic}_{\Ka}(G)$ can be considered as a 
subobject of $L(\Psi\times S)$, i.e. a one-dimensional $\Ka$-vector subspace $\varphi\Ka$ for a rational function 
$\varphi$. Let $S_0$ be a finite subset such that $\varphi\in L(\Psi\times S_0)\subseteq L(\Psi\times S)$. 
Set $\Ka':=L(\Psi\times S_0)\cap\Ka$. Then $\varphi\Ka=\varphi\Ka'\otimes_{\Ka'}\Ka$. We know that 
the $\Ka'\langle G\rangle$-module $L(\Psi\times S_0)$ embeds into 
$\prod_{\lambda\in\mathbb Z\langle S_0\rangle}x^{\lambda}\Ka'$, so $\varphi\Ka'\cong x^{\lambda}\Ka'$ 
for some $\lambda\in\mathbb Z\langle S_0\rangle$, and thus, $\varphi\Ka\cong x^{\lambda}\Ka$. 

As $L(\Psi\times S)$ is a cogenerator of $\Sm_{\Ka}(G)$, to show that the subobjects 
$x^{\lambda}\Ka\subset L(\Psi\times S)$ form a system of cogenerators, it suffices to verify that 
$L(\Psi\times S)$ embeds into $\prod_{\lambda\in\Xi}x^{\lambda}\Ka$. Each element of $L(\Psi\times S)$ 
belongs to $L(\Psi\times S_0)$ for a finite subset $S_0\subseteq S$, while $L(\Psi\times S_0)$ embeds 
into $\prod_{\lambda\in\mathbb Z\langle S_0\rangle}x^{\lambda}\Ka 
\stackrel{(-)\times\text{`0'}}{\hookrightarrow}\prod_{\lambda\in\Xi}x^{\lambda}\Ka$. 

If all $G$-orbits on $\Psi$ are infinite then $L(\Psi\times S)^G\subseteq L$, so $(x^{\gamma}\Ka)^G\neq 0$ 
only if $(x^{\gamma})^{-1}\in\Ka$, which means that $\gamma\in\Gamma$. \end{proof}

\begin{remark} \label{the_socle} Keeping notation of Lemma~\ref{Simple-cogenerators-Ka}, fix 
a subset $\Gamma^{\perp}$ of $\Xi$ projecting bijectively onto the quotient $\Xi/\Gamma$. 
As any simple objects of $\Sm_{\Ka}(\Sy_{\Psi})$ is isomorphic to $x^{\lambda}\Ka$ for 
a unique $\lambda\in\Gamma^{\perp}$, and the $k$-vector space 
$\Hom_{\Ka\langle\Sy_{\Psi}\rangle}(x^{\lambda}\Ka,\prod_{\lambda\in\Xi}x^{\lambda}\Ka)$ is 
isomorphic to $\prod_{\Gamma}k$, the maximal semisimple $\Ka\langle\Sy_{\Psi}\rangle$-submodule 
(the socle) of $\prod_{\lambda\in\Xi}x^{\lambda}\Ka$ is isomorphic to 
$\bigoplus_{\lambda\in\Gamma^{\perp}}x^{\lambda}\Ka\otimes_k\prod_{\Gamma}k$. Then the socle of 
$F_{\Psi}(\Psi\times S)$ is $\bigoplus_{\lambda\in\Gamma^{\perp}}(x^{\lambda}\Ka)^{\oplus\Gamma}$. \end{remark} 

Fix some pairwise distinct $x,y,z\in\Psi$. For each $w\in\Psi$, define the cross-ratio 
$\xi_w:=\frac{(x-w)(z-y)}{(x-z)(w-y)}\in\Kd\cup\{\infty\}$ of the quadruple 
$(w,x,y,z)$ so that $(w,x,y,z)$ is projectively equivalent to $(\xi_w,0,\infty,1)$. 
\begin{lemma} \label{relations_between_K?} Keeping Notation~\ref{the_fields_K?}, the elements $x,y,z\in\Psi$ 
are algebraically independent over $\Kd=F_{\Psi}\left(\xi_u~|~u\in\Psi\smallsetminus\{x,y,z\}\right)$, 
$\Kc=\Kd\left(\frac{x-y}{x-z}\right)=K_{\Psi,\{\ast\},0}^{(F|k)}\cap\Kb$, 
$x/y,x-y$ are algebraically independent over $\Kc$, $\Kb=\Kc(x-y)$, $K_{\Psi,\{\ast\},0}^{(F|k)}=\Kc(x/y)$, 
and $F_{\Psi}(\Psi)=\Kd(x,y,z)=\Kc(x,y)=K_{\Psi,\{\ast\},0}^{(F|k)}(x)=\Kb(x);$ for all 
$?\in\{\mathrm{b},\mathrm{c},\mathrm{d}\}$, the field extensions $\Ka|F_{\Psi}$, $K_?|F_{\Psi}$, 
$F_{\Psi}(\Psi\times S)|K_{\Psi,S,0}^{(F|k)}$, $F_{\Psi}(\Psi)|K_?$ are purely transcendental. \end{lemma} 
\begin{proof} One has $\Ka=F_{\Psi}\left(x^{\gamma},\frac{u_s}{x_s}~|~u\in\Psi\smallsetminus\{x\}\right)$, 
since $u^{\gamma}/x^{\gamma}=(u/x)^{\gamma}$ and $u_s/v_s=\frac{u_s/x_s}{v_s/x_s}$; 
$\Kb=F_{\Psi}\left(u-x~|~u\in\Psi\smallsetminus\{x\}\right)$, since $u-v=(u-x)-(v-x)$; 
$\Kc=F_{\Psi}\left(\frac{x-u}{x-y}~|~u\in\Psi\smallsetminus\{x,y\}\right)$, since 
$\frac{u-v}{u-w}=\frac{\frac{x-u}{x-y}-\frac{x-v}{x-y}}{\frac{x-u}{x-y}-\frac{x-w}{x-y}}$; 
$\Kd=F_{\Psi}\left(\xi_u~|~u\in\Psi\smallsetminus\{x,y,z\}\right)$, since 
$\frac{(\xi_t-\xi_u)(\xi_v-\xi_w)}{(\xi_v-\xi_u)(\xi_t-\xi_w)}=\frac{(t-u)(v-w)}{(v-u)(t-w)}$ 
for all pairwise distinct $t,u,v,w\in\Psi$ with the obvious meaning when one of $\xi$'s 
is $\infty$. Thus, all these fields are purely transcendental over $F_{\Psi}$. 

As $\frac{u}{v}=\frac{1-\frac{x-u}{x-y}(1-\frac{y}{x})}{1-\frac{x-v}{x-y}(1-\frac{y}{x})}$, 
$K_{\Psi,\{\ast\},0}^{(F|k)}=\Kc(x/y)$; as $u-v=(x-y)(\frac{x-v}{x-y}-\frac{x-u}{x-y})$, $\Kb=\Kc(x-y)$; 
as $\frac{x-u}{x-z}\left(\frac{x-y}{x-z}+\frac{(x-u)(y-z)}{(x-z)(y-u)}-1\right)=
\frac{(x-u)(y-z)}{(x-z)(y-u)}\times\frac{x-y}{x-z}$, $\Kc=\Kd\left(\frac{x-y}{x-z}\right)$. Thus, the extensions 
$F_{\Psi}(\Psi)|K_?$ for $?\in\{\mathrm{b},\mathrm{c},\mathrm{d}\}$ and $F_{\Psi}(\Psi\times S)|\Ka$ 
are purely transcendental if $\Xi/\Gamma$ is torsion free. \end{proof} 

\subsection{Structure of \texorpdfstring{$\Sm_{\Ka}(\Sy_{\Psi})$}{Representations over Ka}} 
\begin{proposition} \label{injectivity-Ka} Let $\Psi$ be a set, $F|k$ be a non-trivial regular field extension, 
and $K=F_{\Psi}$. Then $K$ is an injective object of $\Sm_K(\Sy_{\Psi})$. \end{proposition} 
\begin{proof} By Lemma~\ref{injectivity-K-condition}, it suffices to construct, for any finite $J\subset\Psi$, 
a morphism of $K^{(J)}\langle\Sy_{\Psi|J}\rangle$-modules $\xi\colon K\to K^{(J)}$ identical on $K^{(J)}$. 
We consider $K$ as the fraction field of the algebra $K^{(J)}\otimes_kF_J$ and embed both the algebra and 
the field into a field of series with coefficients in $K^{(J)}\otimes_k\overline{k}$ 
for an algebraic closure $\overline{k}$ of $k$. 

For each ordered group $\Gamma$, let $\overline{k}((\Gamma))$ be the field of Hahn power series over 
$\overline{k}$, i.e. the set of formal expressions of the form $\sum_{s\in\Gamma}a_s\cdot s$, where 
$a_s\in\overline{k}$ and the set $\{s\in\Gamma~|~a_s\neq 0\}$ is well-ordered. 
Fix a totally ordered divisible group $\Gamma$, such that transcendence degree of the field 
extension $k((\Gamma))|k$ is at least that of $F|k$. 
By \cite{MacLane}, there is a field embedding $F_J\hookrightarrow\overline{k}((\Gamma))$ over $k$, so 
the $\Sy_{\Psi|J}$-field $K$ becomes a subfield of $(K^{(J)}\otimes_k\overline{k})((\Gamma))$ with 
$\Sy_{\Psi|J}$ acting on the coefficients. Define $\widetilde{\xi}\colon K\to K^{(J)}\otimes_k\overline{k}$ 
as the `constant term'\ of the Hahn power series expression: $\sum_{s\in\Gamma}a_s\cdot s\mapsto a_0$. 
Fix a $k$-linear functional $\nu\colon\overline{k}\to k$ identical on $k$. Finally, we define 
$\xi\colon K\to K^{(J)}$ as $(id_{K^{(J)}}\cdot\nu)\circ\widetilde{\xi}$. \end{proof} 

\begin{lemma} \label{hulls_of_simples_are_cogenerators} If any non-zero object in a Grothendieck 
category $\mathcal A$ admits a simple subquotient $($e.g. $\mathcal A=\Sm_A(G)$ for a permutation 
group $G$ and a unital $G$-ring $A)$ then injective hulls of the simple objects form a `minimal' 
system of cogenerators. $($`Minimal' means$:$ for each injective hull $I$ of a simple object, 
any system of cogenerators contains an element having $I$ as a direct summand.$)$ \end{lemma} 
\begin{proof} Assume that, for any non-zero morphism $X\to Y$, its image admits a subobject 
$Z$ with a simple quotient $Q$. Then the composition $Z\to Q\hookrightarrow E(Q)$ extends to 
a morphism $\varphi\colon Y\to E(Q)$ non-zero on $Z$, and therefore, with a non-zero composition 
$X\to Y\xrightarrow{\varphi}Q\hookrightarrow E(Q)$. Thus, all such $E(Q)$ form a system of cogenerators. 
If $E(Q)$ injects into $\prod_jX_j$ then there is $j$ with non-zero projection $p_j\colon Q\to X_j$. 
But then $p_j$ injects $E(Q)$ into $X_j$, so $E(Q)$ is a direct summand of $X_j$. \end{proof} 

\begin{theorem} \label{smooth-simple} Let $\Psi$ be a set, and $F|k$ be a non-trivial regular field extension. 

Let $K\subseteq F_{\Psi}$ be an $\Sy_{\Psi}$-invariant subfield. Then the object $F_{\Psi}$ is 
an injective cogenerator of the category $\Sm_K(\Sy_{\Psi})$. In particular, {\rm (i)} 
any smooth $K$-semilinear representation of $\Sy_{\Psi}$ can be embedded into a direct product of 
copies of $F_{\Psi};$ {\rm (ii)} any smooth $F_{\Psi}$-semilinear representation of $\Sy_{\Psi}$ 
of finite length is isomorphic to a direct sum of copies of $F_{\Psi}$.\footnote{This is 
\cite[Theorem 3.10]{H90} with corrected conditions on $F|k$, and a restriction on its transcendence 
degree omitted.} \end{theorem} 
\begin{proof} By Proposition~\ref{local-structure}, 
for any smooth simple $F_{\Psi}\langle\Sy_{\Psi}\rangle$-module $M$ there is a finite 
subset $J\subset\Psi$ and an isomorphism of $F_{\Psi}\langle\Sy_{\Psi|J}\rangle$-modules 
$\bigoplus_{s=0}^NF_{\Psi}\langle\binom{\Psi\smallsetminus J}{s}\rangle^{\kappa_s}
\xrightarrow{\sim}M$ for some integer $N,\kappa_0,\dots,\kappa_N\ge 0$. By Lemma~\ref{iso-restr}, 
the $F_{\Psi}\langle\Sy_{\Psi}\rangle$-module $M$ admits a simple 
$F_{\Psi\smallsetminus J}\langle\Sy_{\Psi|J}\rangle$-submodule $M'$. By Lemma~\ref{no-simple-submod}, 
Remark~\ref{more_remarks} (\ref{condition_on_symm-group-field}) and Corollary~\ref{simple-in-K_S}, 
there are no simple $F_{\Psi\smallsetminus J}\langle\Sy_{\Psi|J}\rangle$-submodules in 
$F_{\Psi}\langle\binom{\Psi\smallsetminus J}{s}\rangle$ for $s>0$, so $M'$ is isomorphic to 
$F_{\Psi\smallsetminus J}$, again by Corollary~\ref{simple-in-K_S}, and thus, $M$ is isomorphic to $F_{\Psi}$. 
As $F_{\Psi}$ is injective (Proposition \ref{injectivity-Ka}), Lemma~\ref{hulls_of_simples_are_cogenerators} 
shows that $F_{\Psi}$ is an injective cogenerator of $\Sm_K(\Sy_{\Psi})$. \end{proof}

\begin{corollary}[\cite{H90}, Corollary 3.11] Let $F|k$, $\Psi$ and $K\subseteq F_{\Psi}$ 
be as in Theorem \ref{smooth-simple}. Then any smooth $K$-semilinear irreducible 
representation of $\Sy_{\Psi}$ can be embedded into $F_{\Psi}$. \qed \end{corollary} 

As a consequence, we get a description of level 1 objects of $\Sm_{F_{\Psi}}(\Sy_{\Psi})$: 
\begin{corollary} \label{objects-level-1} Let $\Psi$ be an infinite set, $F|k$ be a non-trivial 
regular field extension, $P_1:=F_{\Psi}\langle\Psi\rangle$, and $M\cong P_1^{\oplus r}$. 
For any set of homomorphisms with a common source $\Upsilon$, denote by ${}^{\perp}\Upsilon$ 
their common kernel. Then there are natural bijections \begin{enumerate} \item 
\begin{gather*}\left\{\begin{array}{c}F_{\Psi}\langle\Sy_{\Psi}\rangle\mbox{{\rm -submodules}}\\ 
V\ \mbox{{\rm of }}M \end{array}\right\}\leftrightarrow
\left\{\begin{array}{c}\mbox{{\rm pairs }}(\Lambda,S),\mbox{{\rm  where }}
\Lambda\ \mbox{{\rm is an }}F\mbox{{\rm -vector subspace in}}\\ 
\Hom_{\Sm_{F_{\Psi}}(\Sy_{\Psi})}(M,P_1)\cong F^{\oplus r},\ S\ \mbox{{\rm is a finite-dimensional }}\\ 
k\mbox{{\rm -vector subspace in}}\ \Hom_{\Sm_{F_{\Psi}}(\Sy_{\Psi})}({}^{\perp}\Lambda,F_{\Psi})
\cong F^{\oplus r}/\Lambda\end{array}\right\},\\ 
V\mapsto(\Lambda(V),\Hom_{\Sm_{F_{\Psi}}(\Sy_{\Psi})}({}^{\perp}\Lambda(V)/V,F_{\Psi})),\ 
\mbox{{\rm where }}\Lambda(V):=\Hom_{\Sm_{F_{\Psi}}(\Sy_{\Psi})}(M/V,P_1)\\ 
(\Lambda,S)\mapsto{}^{\perp}S\ (\mbox{{\rm so }}\dim_kS=\dim_K({}^{\perp}\Lambda/V));\end{gather*} 
\item between the isomorphism classes $[V]$ of $F_{\Psi}\langle\Sy_{\Psi}\rangle$-submodules $M$ 
and the pairs $(t,S)$, where $0\le t\le r$ and $S$ is a $\mathrm{GL}_tF$-orbit of 
finite-dimensional $k$-vector subspaces in $F^{\oplus t}$. \end{enumerate} \end{corollary} 
\begin{proof} By Theorem~\ref{smooth-simple}, this follows from Lemma~\ref{level-1_quotients}. \end{proof}

\begin{corollary} \label{list-inj-Ka} The indecomposable objects $\Ka\langle\binom{\Psi}{s}\rangle$ 
of $\Sm_{\Ka}(\Sy_{\Psi})$ are injective for all integer $s\ge 0$. \end{corollary} 
\begin{proof} This is Lemma~\ref{some-injective}, since by Theorem~\ref{smooth-simple} 
and Lemmas~\ref{Simple-cogenerators-Ka} and \ref{simples_closed-pts_cogenerators} (\ref{simple_cogenerators}), 
$\Ka$ and $\Ka(\Psi)$ are injective objects of $\Sm_{\Ka}(\Sy_{\Psi})$. \end{proof} 

\subsubsection{Partial fraction decomposition} 
\begin{lemma} \label{Part_frac_decomp} Let $X$ be an irreducible curve over a field with the generic point 
$\eta$, $G$ be an automorphism group of the {\bf scheme} $X$, and $\mathcal L$ be a locally free sheaf on $X$. 
Then there exist a natural exact sequence $0\to\Gamma(X,\mathcal L)\to\mathcal L_{\eta}\to
\bigoplus_{x\in X^1}\mathcal L_{\eta}/\mathcal L_x\to H^1(X,\mathcal L)\to 0$, 
and $G$-module isomorphisms $\mathbb Z[G]\otimes_{\mathbb Z[\St_{x_O}]}(\mathcal L_{\eta}/\mathcal L_{x_O})
\xrightarrow{\sim}\bigoplus_{x\in O}\mathcal L_{\eta}/\mathcal L_x$ for all $O\in G\backslash X^1$, 
where $x_O$ is an arbitrary point in $O$. \end{lemma} 
\begin{proof} The first part comes immediately from the flabby 
resolution $\mathcal L_{\eta}\to\coprod_{x\in X^1}\mathcal L_{\eta}/\mathcal L_x$ 
of $\mathcal L$. The second part is evident. \end{proof}

\subsubsection{Lemma~\ref{Simple-cogenerators-Ka} and Corollary~\ref{list-inj-Ka} 
list all indecomposable injectives of $\Sm_{\Ka}(\Sy_{\Psi})$} 
\begin{proof} Let $K$ be of type $\Ka$. 
Any indecomposable injective object is an injective hull of a non-zero cyclic subobject, 
so we only have to show that any smooth finitely generated $K\langle\Sy_{\Psi}\rangle$-module $V$ can be embedded 
into a direct sum of $x^{\lambda}K$ for some $\lambda\in\mathbb Z\langle S\rangle$ 
and of $K\langle\binom{\Psi}{s}\rangle$ for several integer $s\ge 1$. 

By Proposition \ref{local-structure}, there is a subset $\Psi'\subset\Psi$ with 
finite complement $J$ and an isomorphism of $K\langle\Sy_{\Psi|J}\rangle$-modules 
$\bigoplus_{s=0}^NK\langle\binom{\Psi'}{s}\rangle^{\kappa_s}\xrightarrow{\sim}V$ 
for some integer $N,\kappa_0,\dots,\kappa_N\ge 0$. In particular, 
$V':=\indlim_{I\subset\Psi'}V^{\Sy_{\Psi|I}}$, where $I$ runs over the finite subsets of 
$\Psi'$, can be embedded into $\bigoplus_{s=0}^NK\langle\binom{\Psi'}{s}\rangle^{\kappa_s}$.

By Lemma \ref{iso-restr}, it suffices to show that the $K'\langle\Sy_{\Psi|J}\rangle$-module 
$K\langle\binom{\Psi'}{s}\rangle$ is isomorphic to a direct sum of copies of modules 
$x^{\lambda}K',K'\langle\Psi'\rangle,K'\langle\binom{\Psi'}{2}\rangle,\dots$, where $x\in\Psi\smallsetminus J$. 
(Obviously, $\Ka'=F_{\Psi'}(u^{\gamma},u_s/v_s~|~\gamma\in\Gamma,~s\in S,~u,v\in\Psi')$.) 

As $K\langle\binom{\Psi'}{s}\rangle=K\otimes_{K'}K'\langle\binom{\Psi'}{s}\rangle$, this can be reduced 
to the case $s=0$ as follows. Assuming that $K$ is a direct sum of the required type, we only have 
to (i) note that $K'\langle\binom{\Psi'}{s}\rangle\to x^{\lambda}K'\otimes_{K'}K'\langle\binom{\Psi'}{s}\rangle$, 
$[J]\mapsto(\sum_{t\in J}t)^{\lambda}\cdot[J]$, is an isomorphism for any $s\ge 1$, and (ii) to check that 
$K'\langle\binom{\Psi'}{n}\times\binom{\Psi'}{s}\rangle\cong\bigoplus_{j=0}^{n+s}
K'\langle\binom{\Psi'}{j}\rangle^{\oplus N_j}$. It is clear that 
$K'\langle\binom{\Psi'}{n}\times\binom{\Psi'}{s}\rangle
\xrightarrow{\sim}\bigoplus_{j=0}^{\min(n,s)}K'\langle\binom{\Psi'}{j,n-j,s-j}\rangle$, 
$[(I,J)]\mapsto[(I\cap J,I\smallsetminus I\cap J,J\smallsetminus I\cap J)]$, 
where $\binom{\Psi'}{j,n-j,s-j}$ denotes the set of triples of disjoint subsets of $\Psi'$ of orders $j,n-j,s-j$. 
According to Example~\ref{fin-ind-ex}, $K'\langle\binom{\Psi'}{j,n-j,s-j}\rangle$ 
is isomorphic to a direct sum of copies of $K'\langle\binom{\Psi'}{n+s-j}\rangle$. 

Now we treat the case of $s=0$. By Lemma~\ref{some-injective}, the objects 
$x^{\lambda}K',K'\langle\Psi'\rangle,K'\langle\binom{\Psi'}{2}\rangle,\dots$ are injective, while by 
\cite[Theorem 3.18]{H90}, the category $\Sm_{K'}(\Sy_{\Psi|J})$ is locally noetherian, so by Bass--Papp 
theorem, all direct sums of their copies are injective as well. This means that we can split off any 
direct sum of isomorphic copies of the above objects from $K$. It is thus sufficient to show that 
the $K'\langle\Sy_{\Psi|J}\rangle$-submodule in $K\langle\binom{\Psi'}{s}\rangle$ generated by 
any element $\alpha$ is contained in a submodule isomorphic to a finite direct sum 
$\bigoplus_{\lambda}(x^{\lambda}K')^{m_{\lambda}}\oplus\bigoplus_{i\ge 1}K'\langle\binom{\Psi'}{s_i}\rangle^{n_i}$. 

As $F_{\Psi}$ is the fraction field of $F_{\Psi'}\otimes_kF_J$, in the case $K=F_{\Psi}$ it 
suffices to show that, for any regular field extension $L|k$, any element $\alpha$ in the fraction 
field of $F_{\Psi'}\otimes_kL$ is contained in a $K'\langle\Sy_{\Psi|J}\rangle$-submodule 
isomorphic to a finite direct sum $\bigoplus_{i\ge 0}K'\langle\binom{\Psi'}{s_i}\rangle^{n_i}$. 

For a group $G$ and a $G$-field $M$, let $M(t)|M$ be a simple (purely) transcendental 
$G$-field extension such that the line $tM$ is $G$-invariant. In this setting, 
the partial fraction decomposition of Lemma~\ref{Part_frac_decomp} becomes 
\begin{equation}\label{part_fraction_decomp} M(t)=\bigoplus_{n\in\mathbb Z}t^nM\oplus\bigoplus_{m=1}^{\infty}
\bigoplus_O\bigoplus_{j=0}^{\deg O-1}\left(\bigoplus_{P\in O}t^jP(t)^{-m}M\right),\end{equation} where $O$ runs 
over the $G$-orbits of (non-constant) irreducible polynomials $P$ over $M$ such that $P(0)=1$. 
Here the $M\langle G\rangle$-module $\bigoplus_{P\in O}P(t)^{-m}M$ is isomorphic to 
$M\langle G/\St_P\rangle$. 

We proceed by induction on minimal order $d$ of a subset $T\subset L$ such that $\alpha$ is algebraic over 
the fraction field of $F_{\Psi'}\otimes_kk(T)$, the case $d=0$ being trivial. We may assume that $L|k(T)$ 
is finite. Then, in $\Sm_{K'}(\Sy_{\Psi|J})$, the fraction field of $F_{\Psi'}\otimes_kL$ is isomorphic to 
$F_{\Psi'}(T)^{[L:k(T)]}$. Fix some $t\in T$. Taking $G=\Sy_{\Psi'}$ and $M=F_{\Psi'}(T\smallsetminus\{t\})$ 
in decomposition (\ref{part_fraction_decomp}), we see that $F_{\Psi'}(T)=M(t)$ is a direct sum of objects 
isomorphic to $M\langle G/U\rangle$ for open subgroups $U$ of $G$. By induction assumption, 
$F_{\Psi'}(T\smallsetminus\{t\})$ is already of the required type, while 
$\bigoplus_{P\in O}F_{\Psi'}(T\smallsetminus\{t\})\cdot P(t)^{-m}$ is isomorphic to a direct sum of objects 
$K'\langle\binom{\Psi'}{s_i}\rangle\otimes\mathbb Z\langle\Sy_{\Psi'}/\St_P\rangle$ 
for some open subgroups $\St_P\subseteq\Sy_{\Psi'}$. Applying 
Lemmas \ref{open-subgrps-descr} and \ref{finite-index-split} and the above isomorphism 
$K'\langle\binom{\Psi'}{n}\times\binom{\Psi'}{s}\rangle\cong\bigoplus_{j=0}^{n+s}
K'\langle\binom{\Psi'}{j}\rangle^{\oplus N_j}$, we see that the latter object is 
of the required type. This completes the induction step. 

In the case $K=\Ka$, $K$ is contained in $F_{\Psi}(\Psi\times S)$, which is isomorphic in 
$\Sm_{F_{\Psi'}(\Psi'\times S)}(\Sy_{\Psi|J})$ to a direct sum of copies of 
$F_{\Psi'}(\Psi'\times S)\langle\binom{\Psi'}{s}\rangle=
F_{\Psi'}(\Psi'\times S)\otimes_{K'}K'\langle\binom{\Psi'}{s}\rangle$ for $s\ge 0$. To show that 
$F_{\Psi'}(\Psi'\times S)$ is isomorphic in $\Sm_{K'}(\Sy_{\Psi|J})$ to a direct sum of copies of 
injective objects $K'\langle\binom{\Psi'}{s}\rangle$ for $s\ge 1$ and $x^{\lambda}K'$, it suffices to 
check that any element of $F_{\Psi'}(\Psi'\times S)$ belongs to a direct sum of subobjects of such type. 
We may, thus, assume $S$ to be finite, and proceed by induction on $\#S$. As in the proof of 
Lemma~\ref{Simple-cogenerators-Ka}, after a base change of $\mathbb Z\langle S\rangle$, we may assume 
that $\Gamma$ is generated by the elements $m(i)[i]$ for all $i\in S$. 
As $F_{\Psi'}(\Psi'\times S)=K'(x_i~|~i\in S)$ for any $x\in\Psi'$, we again use decomposition 
(\ref{part_fraction_decomp}) with $M=K_{\mathrm{b},t}$ for some $t\in S$ and $G=\Sy_{\Psi'}$ to see that 
$F_{\Psi'}(\Psi'\times S)$ is a direct sum of objects isomorphic to $x_t^dK_{\mathrm{b},t}$ for all 
$d\in\mathbb Z$ or $K_{\mathrm{b},t}\langle G/U\rangle$ for open subgroups $U$ of $G$. Applying arguments 
similar to the above ones, we conclude that $F_{\Psi'}(\Psi'\times S)\langle\binom{\Psi'}{s}\rangle$ 
admits a decomposition of desired type. \end{proof}

\subsection{Structure of \texorpdfstring{$\Sm_{\Kb}(\Sy_{\Psi})$}{Representations over Kb}} 
Let $p$ be the characteristic of $k$. 
\subsubsection{The spectrum of \texorpdfstring{$\Sm_{\Kb}(\Sy_{\Psi})$}{}} 
\begin{theorem} In the notation of Theorem~\ref{spectrum-triple}, the indecomposable injectives 
of $\Sm_{\Kb}(\Sy_{\Psi})$ are injective hulls $P_s^{(\Kb)}$ of $\Kb\langle\binom{\Psi}{s}\rangle$ 
for all integer $s\ge 0$, where $P_0^{(\Kb)}=\Kb[x]$, $P_s^{(\Kb)}=\Kb\langle\binom{\Psi}{s}\rangle$ 
for all $s\ge 1$, except that $P_2^{(\Kb)}=\Kb\langle\Psi^2\smallsetminus\Delta_{\Psi}\rangle$ 
if $F=k$ and $p=2$ $($and $\End_{\Kb\langle\Sy_{\Psi}\rangle}(P_2^{(\Kb)})\cong k[X]/(X^2))$. \end{theorem} 
\begin{proof} By Theorem \ref{smooth-simple}, 
$F_{\Psi}(\Psi)=F(X)_{\Psi}$ is an injective cogenerator of $\Sm_{\Kb}(\Sy_{\Psi})$. For any $x\in\Psi$, 
one has $F(X)_{\Psi}=\Kb(x)=\Kb[x]\oplus\bigoplus_R\bigoplus_{m\ge 1}V^{(m)}_R$, where $R$ runs over 
the $\Sy_{\Psi}$-orbits of non-constant irreducible monic polynomials in $\Kb[x]$ and $V^{(m)}_R$ is 
the $\Kb$-linear envelope of $P/Q^m$ for all $Q\in R$ and $P\in\Kb[x]$ with $\deg P<\deg Q$. As $F(X)_{\Psi}$ 
is injective, its direct summand $\Kb[x]$ is also injective, as well as $V^{(m)}_R$ are for all $R$ and $m$. 

Let $V\subseteq\Kb[x]$ be a non-zero $\Kb\langle\Sy_{\Psi}\rangle$-submodule. Let $Q\in V$ be a monic 
polynomial in $x$. Then $V$ contains $Q-\sigma Q$ for any $\sigma\in\Sy_{\Psi}$ with $\sigma Q\neq Q$. 
Then $Q-\sigma Q$ is a non-zero polynomial of degree $<\deg Q$. In particular, if $Q$ is 
of minimal degree then $Q-\sigma Q=0$ for any $\sigma\in\Sy_{\Psi}$, which means that $Q\in k$, i.e. 
$V$ contains $\Kb$. This implies that $\Kb[x]$ is an injective hull of $\Kb$, in particular, all 
$\Kb\langle\Sy_{\Psi}\rangle$-submodules of $\Kb[x]$ are indecomposable. 
This also follows from Lemma~\ref{End-EKb}. 

Each $V^{(m)}_R$ is filtered by $V^{(j,m)}_R$, $0\le j<\deg R$, where $V^{(j,m)}_R$ is the 
$\Kb$-linear envelope of $P/Q^m$ for all $Q\in R$ and $P\in\Kb[x]$ with $\deg P\le j$. Clearly, 
these decomposition and filtration are independent of $x$. Moreover, for all $0\le j<\deg R$ and $m$ the map 
$V^{(j,1)}_R\to V^{(j,m)}_R$, $P/Q\mapsto P/Q^m$, is an isomorphism of $\Kb\langle\Sy_{\Psi}\rangle$-modules. 

We separate the case where $F=k$ and $p=2$. We denote this case by $\star$. 
Let us show that \begin{itemize} \item for any $R$ and $0\le j<\deg R$, the 
$\Kb\langle\Sy_{\Psi}\rangle$-module $V^{(j,1)}_R$ is non-canonically isomorphic to 
$(\Kb\langle\Sy_{\Psi}/\St_Q\rangle)^{j+1}$, where $Q\in R$ is arbitrary, but in the $\star$ case 
we exclude those $R$ with $\St_Q=\Sy_{\Psi,\{x,y\}}$ for some $Q\in R$ and $x\neq y$ in $\Psi$; 
\item in the $\star$ case, if $\St_Q=\Sy_{\Psi,\{x,y\}}$ for some $R$, $Q\in R$ and $x\neq y$ 
in $\Psi$, and $0\le j<\frac{1}{2}\deg R$, the $\Kb\langle\Sy_{\Psi}\rangle$-module $V^{(2j+1,1)}_R$ is 
non-canonically isomorphic to $(\Kb\langle\Psi^2\smallsetminus\Delta_{\Psi}\rangle)^{j+1}$. \end{itemize} 
(In particular, by Lemma \ref{no-simple-submod}, 
there are no simple $\Kb\langle\Sy_{\Psi}\rangle$-submodules in $V^{(1)}_R$). 

We proceed by induction on $0\le j<\deg R$, the case $j=0$ being obvious, since the morphism 
$\Kb\langle\Sy_{\Psi}/\St_Q\rangle\to V^{(0,1)}_R$, $[g]\mapsto(gQ)^{-1}$, is an isomorphism. 

Let $R$ contain $Q=x_1+\sum_{i=2}^s(x_i-x_1)^2$ for some $s\ge 1$ and pairwise 
distinct $x_1,\dots,x_s\in\Psi$. Then (i) $\St_Q$ is of index $s$ in 
$\Sy_{\Psi,\{x_1,\dots,x_s\}}$; 
(ii) $V_R^{(1)}\cong\Kb\langle\Psi\rangle$ if $s=1$; (iii) 
$V_R^{(1)}\cong\Kb\langle\Psi^2\smallsetminus\Delta_{\Psi}\rangle$ if $s=2$. 
By Lemma~\ref{finite-index-split}, this implies that $V^{(1)}_R$ is isomorphic to 
$\Kb\langle\binom{\Psi}{s}\rangle^s$, unless $F=k$, $p=2$ and $s=2$. 
According to Example~\ref{fin-ind-ex}, this implies that $\Kb\langle\Sy_{\Psi}/U\rangle$ is injective 
for any proper open subgroup $U\subset\Sy_{\Psi}$, with the exception of $U=\Sy_{\Psi,\{x,y\}}$ for 
$x\neq y$ in $\Psi$ in the $\star$ case. 

Now $R$ is arbitrary. By the induction hypothesis, $V^{(j-1,m)}_R$ is isomorphic to 
$(\Kb\langle\Sy_{\Psi}/\St_Q\rangle)^j$, and thus, injective. Then the inclusion 
$V^{(j-1,1)}_R\hookrightarrow V^{(j,1)}_R$ splits: 
$V^{(j,1)}_R\cong V^{(j-1,1)}_R\oplus(V^{(j,1)}_R/V^{(j-1,1)}_R)$. Then the canonical isomorphisms 
$x^j\cdot\colon V^{(0,m)}_R\xrightarrow{\sim}V^{(j,1)}_R/V^{(j-1,1)}_R$ complete the induction step. 

We know that any simple object is isomorphic to $\Kb$, so by Lemma~\ref{hulls_of_simples_are_cogenerators}, 
the injective hull $\Kb[x]$ of $\Kb$ is a cogenerator. 

The remaining indecomposable injectives are described in a way similar to the case of $\Ka$. Namely, we have to 
decompose $F_{\Psi}(\Psi)\langle\binom{\Psi}{s}\rangle=F_{\Psi}(\Psi)\otimes_{\Kb}\Kb\langle\binom{\Psi}{s}\rangle$ 
in $\Sm_{\Kb}(\Sy_{\Psi})$ for all $s\ge 1$. The summands of type $\Kb\langle\Sy_{\Psi}/U\rangle$ in 
$F_{\Psi}(\Psi)$ give the same type of summands in $F_{\Psi}(\Psi)\otimes_{\Kb}\Kb\langle\binom{\Psi}{s}\rangle$. 
Now, if $\Kb\langle\binom{\Psi}{s}\rangle$ is injective then, by the above inductive argument, 
$\Kb[x]^{<n}\otimes_{\Kb}\Kb\langle\binom{\Psi}{s}\rangle\cong\Kb\langle\binom{\Psi}{s}\rangle^{\oplus n}$, so 
$\Kb[x]\otimes_{\Kb}\Kb\langle\binom{\Psi}{s}\rangle\cong\bigoplus_{n\ge 0}\Kb\langle\binom{\Psi}{s}\rangle$. 
Finally, if $p=2$ then 
$\Kb\langle\Psi^2\smallsetminus\Delta_{\Psi}\rangle\to\Kb[x]^{<2}\otimes_{\Kb}\Kb\langle\binom{\Psi}{2}\rangle$, 
$(x,y)\mapsto(x\{x,y\},y\{x,y\})$, is an isomorphism, so reasoning by induction we see that 
$\Kb[x]^{<2n}\otimes_{\Kb}\Kb\langle\binom{\Psi}{2}\rangle$ is injective for any $n\ge 1$ 
and $\Kb[x]^{<2n}\otimes_{\Kb}\Kb\langle\binom{\Psi}{2}\rangle\cong
\Kb\langle\Psi^2\smallsetminus\Delta_{\Psi}\rangle^{\oplus n}$, and thus, 
$\Kb[x]\otimes_{\Kb}\Kb\langle\binom{\Psi}{2}\rangle\cong
\bigoplus_{n\ge 0}\Kb\langle\Psi^2\smallsetminus\Delta_{\Psi}\rangle$. \end{proof} 

\subsubsection{Endomorphisms of \texorpdfstring{$P_0^{(\Kb)}$}{}} For a field $k$ of characteristic 
$p$, let $\mathbb D$ be the completed free divided power $k$-algebra on one PD-generator of the 
maximal ideal, i.e. the elements of $\mathbb D$ are formal series $a_0+\sum_{i\ge 1}a_iD^{(i)}$, 
where $a_i\in k$, with multiplication $D^{(i)}D^{(j)}=\binom{i+j}{i}D^{(i+j)}$ and 
PD-structure on the maximal ideal $(D^{(i)})^{(j)}=\eta_{i,j}D^{(ij)}$ for all integer 
$i,j\ge 1$ and $\eta_{i,j}=\frac{(ij)!}{(i!)^jj!}\in\mathbb Z$. 

\begin{lemma} \label{End-EKb} There is an isomorphism onto the commutative local PD-$k$-algebra $\mathbb D$ 
from the algebra $k[\![X_0,X_1,X_2,\dots]\!]/(X_0^p,X_1^p,X_2^p,\dots)$, given by $X_i\mapsto D^{(p^i)}$, 
if $p>0;$ from the algebra $k[\![X]\!]$, given by $X\mapsto D^{(1)}$, if $p=0$. There is a continuous 
$k$-algebra isomorphism $\mathbb D\xrightarrow{\sim}\End_{\Kb\langle\Sy_{\Psi}\rangle}(\Kb[x])$, 
given by $D^{(i)}\mapsto[x^n\mapsto\binom{n}{i}x^{n-i}]$. 
In particular, $\Hom_{\Kb\langle\Sy_{\Psi}\rangle}(\Kb[x],\Kb)=0$. \end{lemma} 
\begin{proof} By Lucas's theorem, \cite{Lucas}, if $p>0$, $n=\sum_{i\ge 0}n_ip^i$ and $m=\sum_{i\ge 0}m_ip^i$ 
for some $n_i,m_i\in\{0,1,\dots,p-1\}$ then $\binom{n}{m}\equiv\prod_i\binom{n_i}{m_i}\pmod p$. 
Then $D^{(n)}=\prod\limits_{j\ge 0}D^{(n_jp^j)}=
\frac{1}{\prod\limits_{j\ge 0}n_j!}\prod\limits_{j\ge 0}(D^{(p^j)})^{n_j}$. 

An $\Sy_{\Psi}$-endomorphism $\varphi$ of $\Kb[x]$ sends the element $x^n\in\Kb[x]^{\Sy_{\Psi|x}}$ to 
some $P_n^x(x)\in\Kb[x]^{\Sy_{\Psi|x}}=F_x[x]$, where $F_x:=F_{\Psi}^{\Sy_{\Psi|x}}$ is the subfield of 
decomposable tensors with all factors 1, except for the $x$-th. If, moreover, $\varphi$ is $\Kb$-linear 
then $y^n=\sum_{i=0}^n\binom{n}{i}x^i(y-x)^{n-i}$ is sent to $\sum_{i=0}^n\binom{n}{i}(y-x)^{n-i}P_i^x(x)$. 
Therefore, $P_n^y(y)=\sum_{i=0}^n\binom{n}{i}(y-x)^{n-i}P_i^x(x)$. As the left hand side is independent of 
$x$, we get $P_n^y(y)=\sum_{i=0}^n\binom{n}{i}P_i^x(0)y^{n-i}$. Comparing the coefficients of the latter 
polynomials, we see that $P_n^y(y)\in k[y]$, and thus, the map $\varphi$ is given by $\sum_{i\ge 0}a_iD^{(i)}$. 

If $p=0$ then $D^{(i)}=D^i/i!$, so $\End_{\Kb\langle\Sy_{\Psi}\rangle}(\Kb[x])=k[\![D]\!]$. 
If $p>0$ then $D^{(i)}\colon x^n\mapsto(-1)^ix^{n-i}$ for any $n\equiv -1\pmod{p^m}$ with 
$p^m>i$, so $D^{(i)}$ are (topologically) linearly independent. \end{proof} 

For each integer $n\ge 0$, denote by $\delta_p(n)$ the sum of the digits of $n$ in the base $p$, if $p>0$. 

\begin{proposition} \label{structure_of_the_closed_pt} Let $\Phi_s\subset\Kc\left[\frac{x}{x-y}\right]$ be 
the $\Kc$-linear span of $\frac{x^n}{(x-y)^n}$ for all $n\ge 0$ with $\delta_p(n)\le s$, where 
$\delta_0(n):=n$. In particular, $\Phi_{\bullet}$ is exhausting filtration of $\Kc\left[\frac{x}{x-y}\right]$ 
multiplicatively generated by $\Phi_1$, $\Phi_0=\Kc$. 
Then $\Phi_{\bullet}$ is the socle series of $\Kc\left[\frac{x}{x-y}\right]$. The socle series of the object 
$\Kb[x]$ of $\Sm_{\Kb}(\Sy_{\Psi})$ is exhausting as well, and coincides with $\Kb\otimes_{\Kc}\Phi_{\bullet}$. 

If the characteristic of $k$ is $0$ then the module $\Kc[\frac{x}{x-y}]$ is uniserial $($its submodules are 
totally ordered by inclusion$);$ the module $\Kb[x]^{<s}$ admits the injective resolution 
$0\to\Kb[x]^{<s}\to\Kb[x]\xrightarrow{\mathrm{d}^s/\mathrm{d}^sx}\Kb[x]\to 0$, 
so $\mathrm{Ext}^{\ge 2}_{\Kb\langle\Sy_{\Psi}\rangle}(-,V)=0$ for any $V$ of finite length. 

For any $s>0$, there exist indecomposable modules of length $s+1$ with the socle $\cong\Kb^s$. 

Each smooth indecomposable $\Kb\langle\Sy_{\Psi}\rangle$-module $V$ of length $2$ is isomorphic to 
a unique submodule of $\Kb[x]$ that corresponds to a point of $\mathbb P_k((\Kb[x]/\Kb)^{\Sy_{\Psi}})$. 
\end{proposition} 
\begin{proof} Let $p$ be the characteristic of $k$. The $\Kc$-linear map 
$\Phi_s/\Phi_{s-1}\to\bigoplus\limits_{n:~\delta_p(n)=s}(x-y)^{-n}\Kc\subset\Kb$, 
$\left(\frac{x}{x-y}\right)^n\mapsto(x-y)^{-n}$, 
is an isomorphism of $\Kc\langle\Sy_{\Psi}\rangle$-modules, since \begin{multline*} 
\sigma\left(\frac{x}{x-y}\right)^n=\sigma\left(\left(\frac{x}{x-y}\right)^{p^{i_1}+\cdots+p^{i_s}}\right)=
\prod_{j=1}^s\left(\left(\frac{x-y}{x^{\sigma}-y^{\sigma}}\right)^{p^{i_j}}\left(\frac{x}{x-y}\right)^{p^{i_j}}
+\left(\frac{x^{\sigma}-x}{x^{\sigma}-y^{\sigma}}\right)^{p^{i_j}}\right)\\ \equiv
\left(\frac{x-y}{x^{\sigma}-y^{\sigma}}\right)^n\left(\frac{x}{x-y}\right)^n
\bmod{\Phi_{s-1}\cap\Kc\left[\frac{x}{x-y}\right]^{\le n-p^{i_1}}}\end{multline*} 
for any $n=p^{i_1}+\cdots+p^{i_s}$, $i_1\le i_2\le\cdots\le i_s$, and $\sigma\in\Sy_{\Psi}$. 
In particular, $\Phi_s/\Phi_{s-1}$ and $\Kb\otimes_{\Kc}(\Phi_s/\Phi_{s-1})$ are semisimple. It 
remains to check that $\Phi_s/\Phi_{s-1}$ (resp. $\Kb\otimes_{\Kc}(\Phi_s/\Phi_{s-1})$) is the 
socle of $\Kc\left[\frac{x}{x-y}\right]/\Phi_{s-1}$ (resp. of $\Kb[x]/\Kb\otimes_{\Kc}\Phi_{s-1}$). 
We proceed by induction on $s\ge 0$ (with $\Phi_{-1}:=0$), the case $s=0$ being evident. 

As the operators $D^{(i)}\colon\Kb[x]\to\Kb[x]$ commute with the $\Sy_{\Psi}$-action, 
$D^{(i)}$ induce $k$-linear maps $(\Kb[x]/\Kb)^{\Sy_{\Psi}}\to\Kb[x]^{\Sy_{\Psi}}=k$ for all $i\ge 1$. 
If $p>0$ and $n=mp^i$ with $m>1$ prime to $p$ then $D^{(p^i)}x^n=\binom{n}{p^i}x^{n-p^i}=mx^{n-p^i}\neq 0$, 
and therefore, $(\Kb[x]/\Kb)^{\Sy_{\Psi}}\subseteq\{\sum_{i=0}^na_ix^{p^i}~|~a_i\in\Kb\}$. 
If $Q=\sum_{i=0}^na_ix^{p^i}\in(\Kb[x]/\Kb)^{\Sy_{\Psi}}$ then 
$Q^{\sigma}\equiv\sum_{i=0}^na_i^{\sigma}x^{p^i}\pmod\Kb$ for any $\sigma\in\Sy_{\Psi}$, 
so $a_i\in\Kb^{\Sy_{\Psi}}=k$. If $p=0$ then $(\Kb[x]/\Kb)^{\Sy_{\Psi}}=\Kb\cdot x$. 
This identifies $\Phi_1$. 

Let us check that the operators $D^{(i)}$ map $\Phi_s$ to $\Phi_{s-1}$ for all $i\ge 1$. 
This is trivial if $p=0$, so assume that $p>0$. By Lucas's theorem, the set $\mathbb N$ is 
partially ordered by $n\succeq_pm$ if $p\nmid\binom{n}{m}$, 
so if $x^n\in\Phi_s$ then either $D^{(i)}x^n=0$ or $n\succeq_pi$ and 
$x^{n-i}\in\Phi_{s-1}$. Therefore, all $D^{(i)}$ induce $\Kb[x]/\Phi_s\to\Kb[x]/\Phi_{s-1}$ and 
$(\Kb[x]/\Phi_s)^{\Sy_{\Psi}}\to(\Kb[x]/\Phi_{s-1})^{\Sy_{\Psi}}=(\Phi_s/\Phi_{s-1})^{\Sy_{\Psi}}$. 
Moreover, if $x^n\notin\Phi_{s-1}$ and $n\succeq_pp^j$ then $x^{n-p^j}\notin\Phi_{s-2}$. This shows 
that if an element does not belong to $\Phi_{s+1}$ then some $D^{(i)}$ maps it outside of $\Phi_s$, 
so it is not fixed by $\Sy_{\Psi}$ modulo $\Phi_s$, thus completing the induction step. 

Given a smooth indecomposable $\Kb\langle\Sy_{\Psi}\rangle$-module $V$ of length 2, an isomorphism of the 
socle of $V$ onto $\Kb\subset\Kb[x]$ can be extended to a morphism to $\Kb[x]$, which is injective, since the 
socle is essential. To check uniqueness, it suffices to show that the submodules corresponding to distinct 
points of $\mathbb P_k((\Kb[x]/\Kb)^{\Sy_{\Psi}})$ are non-isomorphic. Indeed, as $\Kb[x]$ is injective, 
any isomorphism between submodules of $\Kb[x]$ extends to an endomorphism of $\Kb[x]$, which is scalar on 
$(\Kb[x]/\Kb)^{\Sy_{\Psi}}$. 

For any $s\ge 0$, the dual of $\{a+b_0x+\cdots+b_sx^{p^s}~|~a,b_0\dots,b_s\in\Kb\}$ is 
$\{(a_0+bx,\dots,a_s+bx^{p^s})~|~a_0,\dots,a_s,b\in\Kb\}$ (where the pairing is 
$\langle(a_0+bx,\dots,a_s+bx^{p^s}),a+b_0x+\cdots+b_sx^{p^s}\rangle=ab-a_0b_0-\cdots-a_sb_s$; 
this pairing is equivariant). It is indecomposable, since so is the former. \end{proof} 

\subsection{Structure of \texorpdfstring{$\Sm_{\Kc}(\Sy_{\Psi})$}{Representations over Kc}} 

\subsubsection{The projective line \texorpdfstring{$\mathbb Y$}{Y} over \texorpdfstring{$\Kd;$}{Kd;} invertible 
objects \texorpdfstring{$\omega_{\mathbb Y,\eta}^n$}{} of \texorpdfstring{$\Sm_{\Kc}(\Sy_{\Psi})$}{}} 
\label{projective_line_Y}
For any field extension $L|K$, denote by $\mathrm{Val}_{L|K}$ the set of all discrete 
valuations $L^{\times}\longrightarrow\hspace{-3mm}\to\mathbb Z$ trivial on $K$.

\begin{lemma} \label{stab_val} \begin{itemize} \item $\mathrm{Val}_{\Kc|\Kd}$ can be identified 
naturally with the set $\mathbb Y^1$ of closed points of a projective line $\mathbb Y$ over $\Kd;$ 
the natural $\Sy_{\Psi}$-action on $\mathrm{Val}_{\Kc|\Kd}$ is smooth and degree-preserving. 
\item For each triple of pairwise distinct $\alpha,\beta,\gamma\in\Psi$ there is a unique 
$v_{\gamma}\in\mathrm{Val}_{\Kc|\Kd}$ with $v_{\gamma}\left(\frac{\alpha-\beta}{\beta-\gamma}\right)>0$. 
This $v_{\gamma}$ does not depend on $\alpha,\beta$. The map $\Psi\to\mathrm{Val}_{\Kc|\Kd}$, 
$\gamma\mapsto v_{\gamma}$, is an $\Sy_{\Psi}$-equivariant injection$;$ 
it identifies $\Psi$ with an $\Sy_{\Psi}$-orbit $\Theta:=\{v_{\gamma}~|~\gamma\in\Psi\}$. 
\item For any $\gamma\in\Psi$, the valuation $v_{\gamma}$ is trivial on the subfield 
$\varkappa_{\gamma}:=F_{\Psi}\left(\frac{u-v}{v-w}~|~u,v,w\in\Psi\smallsetminus\{\gamma\}\right)$ 
of $\Kc$. The residue field of $v_{\gamma}$ is identified naturally with $\varkappa_{\gamma}$. 
\item For any distinct $\alpha,\gamma\in\Psi$, the fixed points of $\Sy_{\Psi|\{\alpha,\gamma\}}$ 
acting on $\mathrm{Val}_{\Kc|\Kd}$ are $v_{\alpha}$ and $v_{\gamma}$. \end{itemize} \end{lemma} 
\begin{proof} For any $\alpha'\neq\alpha,\gamma$, the element $\frac{\alpha-\alpha'}{\alpha'-\gamma}$ 
is a $\Kd^{\times}$-multiple of $\frac{\alpha-\beta}{\beta-\gamma}$, so their valuations coincide. 
If $v\in\mathrm{Val}_{\Kc|\Kd}$ is positive on $\frac{\alpha-\beta}{\beta-\gamma}$ then 
$0=v\left(1+\frac{\alpha-\alpha'}{\alpha'-\gamma}\right)=v\left(\frac{\alpha-\gamma}{\alpha'-\gamma}\right)$, 
so $v\left(\frac{\alpha-\alpha'}{\alpha'-\gamma}\right)=v\left(\frac{\alpha-\alpha'}{\alpha-\gamma}\right)=
v\left(\frac{\beta'-\alpha'}{\beta'-\gamma}\right)$ for arbitrary $\beta'\neq\alpha'$ distinct from 
$\gamma$. Thus, the stabilizer of $v$ is $\Sy_{\Psi|\{\gamma\}}$, and sending 
$v$ to $\gamma$ gives rise to an isomorphism of $\Sy_{\Psi}$-sets $\Theta\xrightarrow{\sim}\Psi$. 

Any element $f\in\varkappa_{\gamma}^{\times}$ is fixed by $\Sy_{\Psi|S}$ for a finite subset 
$S\subset\Psi\smallsetminus\{\gamma\}$. Then $v_{\beta}(f)=v_{\gamma}(f)$ for any $\beta$ in 
the $\Sy_{\Psi|S}$-orbit $\Psi\smallsetminus S$ of $\gamma$. As only finitely many 
valuations are non-zero on $f$, $v_{\gamma}(f)=0$. 

By Lemma~\ref{relations_between_K?}, the elements of $\Kc$ can be considered as rational functions over 
$\Kd$ (and over $\varkappa_{\gamma}$) in the variable $T:=\frac{\alpha-\beta}{\beta-\gamma}\in\Kc$, 
so each $v\in\mathrm{Val}_{\Kc|\Kd}\smallsetminus\{v_{\alpha}\}$ corresponds to a closed point of 
$\mathbb A^1_{\Kd}$, i.e. to a maximal ideal in $\Kd[T]$, or to a monic irreducible polynomial 
$P=\sum_{i=0}^na_iT^i\in\Kd[T]$. For any $g\in\Sy_{\Psi|\{\alpha,\gamma\}}$, 
$P^g=\sum_{i=0}^na_i^g\left(\frac{(\alpha-\beta^g)(\beta-\gamma)}{(\alpha-\beta)(\beta^g-\gamma)}\right)^iT^i$, 
so if $v$ is fixed by $\Sy_{\Psi|\{\alpha,\gamma\}}$ then 
\[a_i^g\left(\frac{(\alpha-\beta^g)(\beta-\gamma)}{(\alpha-\beta)(\beta^g-\gamma)}\right)^i
=a_i\left(\frac{(\alpha-\beta^g)(\beta-\gamma)}{(\alpha-\beta)(\beta^g-\gamma)}\right)^n
\mbox{ for all }0\le i\le n,\] which means that $a_i\left(\frac{\alpha-\beta}{\beta-\gamma}\right)^{n-i}
\in\Kc^{\Sy_{\Psi|\{\alpha,\gamma\}}}=F_{\{\alpha,\gamma\}}$, 
i.e. $a_i=0$ for $i\neq n$, so $n=1$, and thus, $v=v_{\gamma}$. \end{proof}

\begin{notation} \begin{itemize} \item 
$\mathcal O_{\mathbb Y}\colon U\mapsto\{f\in\Kc~|~v(f)\ge 0\ \mbox{for all }v\in U\}$ is the structure 
sheaf of $\mathbb Y$; \item for each integer $n$, $\omega_{\mathbb Y}^n$ is the $n$-th tensor power of 
the dualizing sheaf $\Omega^1_{\mathbb Y|\Kd}$ of $\mathbb Y$; 
\item for each point $v$ of $\mathbb Y$, $\omega_{\mathbb Y,v}^n$ is the stalk of $\omega_{\mathbb Y}^n$ at $v$; 
\item $\eta$ is the general point of $\mathbb Y$, so 
\item for each integer $n$, $\omega_{\mathbb Y,\eta}^n$ is the one-dimensional 
$\Kc$-vector space of rational sections of $\omega_{\mathbb Y}^n$. \end{itemize} \end{notation} 

\begin{lemma} \label{fixed_1-form} Fix some pairwise distinct $x,y,z\in\Psi$, 
and set $T:=\frac{y-z}{z-x}\in\Kc$. Then {\rm (i)} the $1$-form 
$\varpi:=\frac{\mathrm{d}T}{(x-y)T}\in\Omega^1_{\Kb|\Kd}$ is fixed by $\Sy_{\Psi}$, 
and therefore, it is independent of $x,y,z;$ {\rm (ii)} for each $n\in\mathbb Z$, 
$\Hom_{\Kc\langle\Sy_{\Psi}\rangle}(\omega_{\mathbb Y,\eta}^n,F_{\Psi}(\Psi))
=k\cdot\iota_{\mathrm{c}}^n$, where 
$\iota_{\mathrm{c}}^n\colon\omega_{\mathbb Y,\eta}^n\xrightarrow[\sim]{\times\varpi^{-n}}
(x-y)^n\Kc\subset\Kb$ sends $\left(\frac{\mathrm{d}T}{T}\right)^n$ to $(x-y)^n$. \end{lemma} 

\begin{proof} For any $\sigma\in\Sy_{\Psi|\{x,y\}}$, one has $T^{\sigma}/T\in\Kd^{\times}$, so 
(i) $\varpi$ is fixed by $\Sy_{\Psi|\{x,y\}}$ and (ii) $\varpi=\frac{(z-x)\mathrm{d}(y-z)-(y-z)
\mathrm{d}(z-x)}{(x-y)(y-z)(z-x)}=\frac{(y-z)\mathrm{d}x+(z-x)\mathrm{d}y+(x-y)\mathrm{d}z}{(x-y)(y-z)(z-x)}
\in\Omega^1_{\Kb|\Kd}\subset\Omega^1_{F_{\Psi}|\Kd}$ is  symmetric in $x,y,z$. Then $\varpi$ is also 
fixed by $\Sy_{\Psi|\{y,z\}}$ and $\Sy_{\Psi|\{z,x\}}$, and thus, it is fixed by $\Sy_{\Psi}$, 
so the multiplication by $\varpi^{-n}$ is a $\Kd\langle\Sy_{\Psi}\rangle$-morphism. 

As (i) $\Hom_{\Kc\langle\Sy_{\Psi}\rangle}(\omega_{\mathbb Y,\eta}^n,F_{\Psi}(\Psi))=
(F_{\Psi}(\Psi)\otimes_{\Kc}\omega_{\mathbb Y,\eta}^{-n})^{\Sy_{\Psi}}$, (ii) $F_{\Psi}(\Psi)^{\Sy_{\Psi}}=k$ 
and (iii) the $F_{\Psi}(\Psi)$-vector space $F_{\Psi}(\Psi)\otimes_{\Kc}\omega_{\mathbb Y,\eta}^{-n}$ is 
one-dimensional, one has 
$\dim_k\Hom_{\Kc\langle\Sy_{\Psi}\rangle}(\omega_{\mathbb Y,\eta}^n,F_{\Psi}(\Psi))\le 1$, i.e. any 
morphism from $\omega_{\mathbb Y,\eta}^n$ to $F_{\Psi}(\Psi)$ is a $k$-multiple of $\iota_n$. \end{proof}

\subsubsection{The spectrum of \texorpdfstring{$\Sm_{\Kc}(\Sy_{\Psi})$}{}} 
\begin{proposition} \label{simple-X-Y} Any simple object of 
$\Sm_{\Kc}(\Sy_{\Psi})$ is isomorphic to $\omega_{\mathbb Y,\eta}^n$ for some $n\in\mathbb Z$. 

The object $M^{(n)}:=\omega_{\mathbb Y,\eta}^n\otimes_{\Kc}\Kc\left[\frac{x}{x-y}\right]\cong
(x-y)^n\Kc\left[\frac{x}{x-y}\right]\subset F_{\Psi}(\Psi)$ is an injective hull of 
$\omega_{\mathbb Y,\eta}^n$. 

The remaining isomorphism classes of indecomposable injective objects 
of $\Sm_{\Kc}(\Sy_{\Psi})$ are presented by $(x-y)^n\Kc\langle\Psi\rangle$ for $n\in\mathbb Z$, 
direct summands of $\Kc\langle\Psi^2\smallsetminus\Delta_{\Psi}\rangle$, and 
$\Kc\langle\binom{\Psi}{s}\rangle$ for $s\ge 3$. \end{proposition} 
\begin{remark} By Proposition~\ref{structure_of_the_closed_pt}, the socle series $\Phi^{(n)}_{\bullet}$ of 
$M^{(n)}$ is exhausting and coincides with $\omega_{\mathbb Y,\eta}^n\otimes_{\Kc}\Phi_{\bullet}$. 
\end{remark} 
\begin{proof} Set $X:=x-y$, so $\Sy_{\Psi}$ acts on $\Kc[X]=\Kc[x-y]$ by $\sigma:\sum_ia_iX^i
\mapsto\sum_ia_i^{\sigma}\left(\frac{x^{\sigma}-y^{\sigma}}{x-y}\right)^iX^i$. 
Then $\Kb=\Kc(x-y)=\bigoplus_{n\in\mathbb Z}(x-y)^n\Kc\oplus\bigoplus_R\bigoplus_{m\ge 1}V_R^{(m)}$, 
where $R$ runs over the set of $\Sy_{\Psi}$-orbits of non-constant irreducible polynomials  
$Q(X)\in\Kc[X]$ with $Q(0)=1$ and $V_R^{(m)}$ is the $\Kc$-vector subspace of $\Kb$ spanned by 
$(x-y)^j/Q(x-y)^m$ for all $0\le j<\deg R$ and all $Q\in R$, so 
\[\Kb[x]=\Kb\otimes_{\Kc}\Kc\left[\frac{x}{x-y}\right]=\bigoplus_{n\in\mathbb Z}M^{(n)}
\oplus\bigoplus_R\bigoplus_{m\ge 1}V_R^{(m)}\otimes_{\Kc}\Kc\left[\frac{x}{x-y}\right].\] In particular, 
$M^{(n)}$ and $V_R^{(1)}$ are injective for all $n\in\mathbb Z$ and $R$. Note that 
for all $m$ the map $V^{(1)}_R\to V^{(m)}_R$, $P(x-y)/Q(x-y)\mapsto P(x)/Q(x-y)^m$, is an isomorphism of 
$\Kc\langle\Sy_{\Psi}\rangle$-modules.

Any simple object $V$ of $\Sm_{\Kc}(\Sy_{\Psi})$ can be embedded into a simple object of 
$\Sm_{\Kb}(\Sy_{\Psi})$, i.e. into $\Kb$. 
As $\Kb$ admits an $\Sy_{\Psi}$-invariant discrete valuation trivial on $\Kc$ and positive on $x-y$, $\Kb$ 
embeds $\Sy_{\Psi}$-equivariantly into $\Kc((x-y))=\indlim_m\prod\limits_{n\ge -m}(x-y)^n\Kc\subset
\prod\limits_{n\in\mathbb Z}(x-y)^n\Kc$, so $V$ is isomorphic to $(x-y)^n\Kc$ for some $n\in\mathbb Z$. 

The natural map $\Hom_{\Kc\langle\Sy_{\Psi}\rangle}(M,M')\to
\Hom_{\Kc\langle\Sy_{\Psi}\rangle}(M\otimes_{\Kc}\mathcal L,M'\otimes_{\Kc}\mathcal L)$ is bijective 
for any invertible $\mathcal L$ in $\Sm_{\Kc}(\Sy_{\Psi})$, so 
$\Hom_{\Kc\langle\Sy_{\Psi}\rangle}((x-y)^m\Kc,M^{(n)})\cong
(M^{(n-m)})^{\Sy_{\Psi}}=\begin{cases}k,&\mbox{if $n=m$,}\\0,&\mbox{if $n\neq m$.}\end{cases}$ 
This means that the socle of $M^{(n)}$ is $(x-y)^n\Kc$, so $M^{(n)}$ is an injective hull of $(x-y)^n\Kc$. 
As $M^{(n)}$ is artinian, the socle series $\Phi^{(n)}_{\bullet}$ of $M^{(n)}$ is exhausting. 

Let $R$ contain $Q(x_1-x_2)=(x_1-x_2)\prod_{i=3}^s\frac{x_i-x_1}{x_i-x_2}+1$ for some $s\ge 2$ and pairwise 
distinct $x_1,\dots,x_s\in\Psi$. Then (i) $V_R^{(1)}\cong\Kc\langle\Sy_{\Psi}/\St_Q\rangle$; 
(ii) $\St_Q$ is of index $s(s-1)$ in $\Sy_{\Psi,\{x_1,\dots,x_s\}}$ for all $s\ge 2$, with the 
exception of $s=2,3$ when $p=2$, where $\St_Q$ is of index $2s-3$ in $\Sy_{\Psi,\{x_1,\dots,x_s\}}$. 
By Lemma~\ref{finite-index-split}, if $U$ is an open subgroup in $\Sy_{\Psi,\{x_1,\dots,x_s\}}$ 
of a finite index $\kappa$ for some $s\ge 3$ then $\Kc\langle\Sy_{\Psi}/U\rangle$ is isomorphic to 
$\Kc\langle\binom{\Psi}{s}\rangle^{\oplus\kappa}$ if $s\ge 3$; 
$\Kc\langle\Psi^2\smallsetminus\Delta_{\Psi}\rangle\to\Kc\langle\binom{\Psi}{2}\rangle\oplus
(x-y)\Kc\langle\binom{\Psi}{2}\rangle$, $(x,y)\mapsto(\{x,y\},(x-y)\{x,y\})$, is an isomorphism for $p\neq 2$.   
In particular, $V^{(1)}_R\cong\Kc\langle\binom{\Psi}{s}\rangle^{\oplus\kappa_s}$ if $s\ge 3$. 

As $\Kb\langle\Psi^2\smallsetminus\Delta_{\Psi}\rangle$ and $\Kb\langle\binom{\Psi}{s}\rangle$ 
are injective for $s\neq 0,2$ (it suffices to use this fact for $s=1$), the objects 
$\Kc\langle\Psi^2\smallsetminus\Delta_{\Psi}\rangle\otimes_{\Kc}\Kc\langle\binom{\Psi}{s}\rangle$ and 
$\Kc\langle\binom{\Psi}{t}\rangle\otimes_{\Kc}\Kc\langle\binom{\Psi}{s}\rangle$ are injective for $t\ge 3$ 
and $s\neq 0,2$. The $\Sy_{\Psi}$-set $\binom{\Psi}{t}\times\Psi$ contains a $\Sy_{\Psi}$-subset 
isomorphic to $\Sy_{\Psi}/U$ for a subgroup $U$ of finite index in $\Sy_{\Psi|\{x_1,\dots,x_t\}}$, 
so $P_t^{(\Kc)}:=\Kc\langle\binom{\Psi}{t}\rangle$ are injective for all $t\ge 3$. Similarly, the 
$\Sy_{\Psi}$-set $(\Psi^2\smallsetminus\Delta_{\Psi})\times\Psi$ contains a $\Sy_{\Psi}$-subset isomorphic 
to $\Psi^2\smallsetminus\Delta_{\Psi}$, so $\Kc\langle\Psi^2\smallsetminus\Delta_{\Psi}\rangle$ 
(which is $P_2^{(\Kc)}$ if $p=2$, $P_2^{(\Kc)}\oplus{P_2'}^{(\Kc)}$ if $p\neq 2$) is 
injective. The remaining summands $(x-y)^n\Kc$ of $\Kb$ give the injective objects 
$P_{1,n}^{(\Kc)}:=(x-y)^n\Kc\langle\Psi\rangle$ in $\Sm_{\Kc}(\Sy_{\Psi})$. Clearly, the $P_s^{(\Kb)}$ 
for $s\ge 1$ are direct sums of such objects in $\Sm_{\Kc}(\Sy_{\Psi})$. 

There remains to decompose $V:=V_R^{(1)}\otimes_{\Kc}\Kc\left[\frac{x}{x-y}\right]$. The object 
$\Kc\left[\frac{x}{x-y}\right]$ is filtered with the successive quotients $(x-y)^n\Kc$ with $n\ge 0$. 
If $V_R^{(1)}\cong\Kc\langle\binom{\Psi}{s}\rangle^{\oplus\kappa}$ with $s\ge 3$ then $V$ is isomorphic 
to a direct sum of copies of $\Kc\langle\binom{\Psi}{s}\rangle$. Similarly, if $V_R^{(1)}$ is isomorphic 
to (a direct summand of) $\Kc\langle\Psi^2\smallsetminus\Delta_{\Psi}\rangle$ then $V$ is isomorphic 
to a direct sum of copies of (direct summands of) $\Kc\langle\Psi^2\smallsetminus\Delta_{\Psi}\rangle$. 
\end{proof} 

\section{Several results on representations over \texorpdfstring{$\Kd$}{}} 
\subsection{Residues of 1-forms on curves} 
\label{Residues} Let $C$ be a smooth curve over a field $K$, and $P$ be a closed point of $C$. 
Denote by $\varkappa(P)$ the residue field of $P$, and assume that $\varkappa(P)$ is separable over $K$. 

\begin{definition} The {\sl residue} at $P$ is the $K$-linear map 
$\mathrm{res}_P:\Omega_{K(C)|K}\xrightarrow{\sum_ia_it_P^idt_P\mapsto a_{-1}}\varkappa(P)$, 
where $t_P$ is a local parameter at $P$, the completion of the local ring $\mathcal O_P$ of $P$ 
is identified with $\varkappa(P)[\![t_P]\!]$ (so the function field $K(C)$ becomes embedded into 
$\varkappa(P)(\!(t_P)\!)$), and $a_i\in\varkappa(P)$. \end{definition} 

It is well-known that $\mathrm{res}_P$ is independent of $t_P$ (\cite[Ch. II, \S11]{Serre}), and 
the Cauchy formula holds: $\sum_{P\in C^1}\mathrm{tr}_{\varkappa(P)|K}(\mathrm{res}_P(\omega))=0$ 
for any $\omega\in\Omega_{K(C)|K}$ if $C$ is projective (\cite[Ch. II, \S12]{Serre}). 

As in \S\ref{projective_line_Y}, for each field extension $L|K$, we denote by $\mathrm{Val}_{L|K}$ 
the set of all discrete valuations of $L$ trivial on $K$. 

For each $G$-orbit $O$ on $\mathrm{Val}_{L|K}$ with the residue fields $\varkappa(v)$ for $v\in O$
separable over $K$, define \begin{equation} \label{Res_O} \mathrm{Res}_O:\Omega^1_{L|K}
\xrightarrow{\omega\mapsto\sum_{v\in O}\mathrm{tr}_{\varkappa(v)|K}(\mathrm{res}_v(\omega))\cdot[v]}
K\langle O\rangle.\end{equation} 

\begin{remark} \label{direct-summands_of_1-forms} Let $G$ be a group, $L|K$ be a purely transcendental 
$G$-field extension of transcendence degree 1, and $O$ be a $G$-orbit of some $v\in\mathrm{Val}_{L|K}$. 
Suppose that $\frac{[\varkappa(v):K]}{N[\varkappa(v_1):K]}\in K^{\times}$ for a $\St_v$-orbit 
$\{v_1,\dots,v_N\}$ in $\mathrm{Val}_{L|K}\smallsetminus O$. Then $\mathrm{Res}_O$ splits. [Namely, 
if $c$ is the greatest common divisor of $[\varkappa(v):K]$ and $N[\varkappa(v_1):K]$ then 
$\frac{N}{c}[\varkappa(v_1):K][v]-\frac{1}{c}[\varkappa(v):K]\sum_{i=1}^N[v_i]$ is the divisor of an element 
$f\in(L^{\times}/K^{\times})^{\St_v}$, so a splitting of $\mathrm{Res}_O$ is determined uniquely 
by the condition $[v]\mapsto\frac{[\varkappa(v):K]}{N[\varkappa(v_1):K]}\cdot\frac{\mathrm{d}f}{f}$.] 
\end{remark}

\subsection{Categories \texorpdfstring{$\Sm_{\Kd}(U)$}{} for open subgroups 
\texorpdfstring{$U\subseteq\Sy_{\Psi}$}{}} 

\begin{proposition} \label{some-Kd-inj} The objects 
$\Kd\langle\{\{1,2,3\}\hookrightarrow\Psi\}\rangle$ 
and $\Kd\langle\binom{\Psi}{s}\rangle$ for $s\ge 4$ are injective. \end{proposition} 
\begin{proof} Set $\binom{\Psi}{2}':=\Psi^2\smallsetminus\Delta_{\Psi}$ if $p=2$ and $F=k$, 
and $\binom{\Psi}{s}':=\binom{\Psi}{s}$ in all other cases. 

By Proposition \ref{simple-X-Y}, for all $s\ge 1$, the objects 
$\Omega^1_{\Kc|\Kd}\otimes_{\Kd}\Kd\langle\binom{\Psi}{s}'\rangle$ of $\Sm_{\Kc}(\Sy_{\Psi})$ 
are injective, and therefore, they are injective as objects of $\Sm_{\Kd}(\Sy_{\Psi})$. 

Suppose that, for an $\Sy_{\Psi}$-orbit $O$ on the set 
$\mathbb Y^1:=\mathrm{Val}_{\Kc|\Kd}$ with the residue fields separable over $\Kd$, the residue map 
$\mathrm{Res}_O:\Omega^1_{\Kc|\Kd}\to\Kd\langle O\rangle$ of (\ref{Res_O}) from \S\ref{Residues} 
is split as a morphism in $\Sm_{\Kd}(\Sy_{\Psi})$.
Then, for each $s\ge 1$, the object $\Kd\langle O\rangle\otimes_{\Kd}\Kd\langle\binom{\Psi}{s}'\rangle
=\Kd\langle O\times\binom{\Psi}{s}'\rangle$ of $\Sm_{\Kd}(\Sy_{\Psi})$ is a direct summand 
of the injective object $\Omega^1_{\Kc|\Kd}\otimes_{\Kd}\Kd\langle\binom{\Psi}{s}'\rangle$, 
so $\Kd\langle O\times\binom{\Psi}{s}'\rangle$ is injective. 

For any $\gamma\in O$, if the stabilizer of $\gamma$ fixes some $J\in\binom{\Psi}{s}'$ then 
the $\Sy_{\Psi}$-orbit of $(\gamma,J)\in O\times\binom{\Psi}{s}'$ is isomorphic to $O$, 
so $\Kd\langle O\rangle$ is injective as well. 

Let us show that for each $s\ge 3$ there exists a pair $(q,f)$, where (i) $q$ is a closed point 
of $\mathbb Y$ with the stabilizer $\Sy_{\Psi|S}$ for some $S=\{x_1,\dots,x_s\}\subset\Psi$ 
of order $s\ge 3$, (ii) $f\in(\Kd(\mathbb Y)^{\times}/\Kd^{\times})^{\Sy_{\Psi|S}}$ is 
such that $\mathrm{Res}_O\left(\frac{\mathrm{d}f}{f}\right)=[q]$, where $O$ is 
the $\Sy_{\Psi}$-orbit of $q$. Obviously, such $\mathrm{Res}_O$ are split. 

Set $T:=\frac{x_2-\beta}{\beta-x_1}$ for some $\beta\in\Psi\smallsetminus S$ and 
$\xi_i:=\frac{(x_i-x_1)(x_2-\beta)}{(x_i-x_2)(x_1-\beta)}\in\Kd$ for $3\le i\le s$. For 
each $\lambda\in\Kd$, let $q_{\lambda}$ be the point of $\mathbb Y^1\smallsetminus\{x_2\}$ 
corresponding (as in the proof of Lemma~\ref{stab_val}) to the polynomial 
$T+\lambda\in\Kd[T]$ (e.g., $q_0=x_1$ and $q_1=\beta$), and $O_{\lambda}$ be the 
$\Sy_{\Psi}$-orbit of $q_{\lambda}$. Then $x_2\notin O_{\lambda}$ if $\lambda\notin\{0,1\}$, 
so $\mathrm{Res}_{O_{\lambda}}\left(\frac{\mathrm{d}T}{T+\lambda}\right)=[q_{\lambda}]$. 

If $s\ge 3$ and $\lambda=\sum_{i=3}^s(\xi_3/\xi_s)^i\xi_i$ then 
$T+\lambda\in(\Kd(\mathbb Y)^{\times}/\Kd^{\times})^{\Sy_{\Psi|S}}$,\footnote{As 
$T+\lambda=\frac{x_2-\beta}{\beta-x_1}\left(1-\sum_{i=3}^s(\xi_3/\xi_s)^i\frac{x_i-x_1}{x_i-x_2}\right)$, 
one has $T^g+\lambda^g=\frac{(x_2-\beta^g)(\beta-x_1)}{(\beta^g-x_1)(x_2-\beta)}(T+\lambda)$ for any 
$g\in\Sy_{\Psi|S}$.} and the stabilizer of $q_{\lambda}$ is $\Sy_{\Psi|S}$, so 
$\Kd\langle\Sy_{\Psi}/\Sy_{\Psi|S}\rangle$ is injective. 
As the symmetric group $\Sy_S$ acts faithfully on $\Kd^{\Sy_{\Psi|S}}$ when $s\ge 4$, by 
Lemma~\ref{finite-index-split}, the object $\Kd\langle\Sy_{\Psi}/\Sy_{\Psi|S}\rangle$ is 
a direct sum of copies of the indecomposable object $\Kd\langle\binom{\Psi}{s}\rangle$ 
for $s\ge 4$, and therefore, $\Kd\langle\binom{\Psi}{s}\rangle$ is also injective. \end{proof}

Let $J\subset\Psi$ be a finite subset of order $m$. By Lemma~\ref{stab_val}, $\Sy_{\Psi|J}$ fixes precisely 
$m$ closed point of $\mathbb Y$ if $m\le 2$, though these points are in the same $\Sy_{\Psi}$-orbit. In 
particular, the (setwise) stabilizer $\Sy_{\Psi,J}$ of $J$ fixes no closed point of $\mathbb Y$ if $m=2$. 
However, if $m\ge 3$, there exist infinitely many $\Sy_{\Psi}$-orbits of closed points of $\mathbb Y$ 
containing points whose stabilizers in $\Sy_{\Psi}$ coincide with $\Sy_{\Psi|J}$.

\begin{lemma} \label{multiplicities} Let $G$ be a group, $\wK$ be a $G$-field, and $K\subseteq\wK$ 
be a $G$-invariant subfield. Set $k:=\wK^G$. Let $V\neq 0$ be a $K\langle G\rangle$-module of 
a finite dimension $d$. Then {\rm (i)} $\dim_k\Hom_{K\langle G\rangle}(V,\wK)\le d;$ {\rm (ii)} 
$V$ is simple and $\End_{K\langle G\rangle}(V)=k$ if $k\subseteq K$ and there is 
a $K\langle G\rangle$-module embedding $V^{\oplus d}\hookrightarrow\wK$. 
In particular, if $\mathcal L$ and $\mathcal L'$ are isomorphic invertible 
$K\langle G\rangle$-submodules of $\wK$ then $\mathcal L=\mathcal L'$ if $k\subseteq K$. \end{lemma} 
\begin{proof} (i) As it is well-known, the multiplication map 
$\wK\otimes_k(\wK\otimes_KV^{\vee})^G\to\wK\otimes_KV^{\vee}$ is injective, so 
$\dim_k\Hom_{K\langle G\rangle}(V,\wK)=\dim_k(\wK\otimes_KV^{\vee})^G
=\dim_{\wK}(\wK\otimes_k(\wK\otimes_KV^{\vee})^G)\le
\dim_{\wK}(\wK\otimes_KV^{\vee})=\dim_KV=:d$. 

(ii) For any submodule $V'\subseteq V$, any embedding $V^{\oplus d}\hookrightarrow\wK$ restricts 
to an embedding ${V'}^{\oplus d}\hookrightarrow\wK$. If $V'$ is simple then the existence 
of such an embedding implies that $\dim_D\Hom_{K\langle G\rangle}(V',\wK)\ge d$, where 
$D:=\End_{K\langle G\rangle}(V')$, so $e:=\dim_k\Hom_{K\langle G\rangle}(V',\wK)\ge d\cdot\dim_kD$. 

By (i), $e\le\dim_KV'$, so $d\ge\dim_KV'\ge e\ge d\cdot\dim_kD$, and therefore, 
$d=\dim_KV'=d\cdot\dim_kD$, i.e. $V'=V$ and $\End_{K\langle G\rangle}(V)=D=k$. \end{proof}

For any permutation group $\mathcal G$, a smooth $\mathcal G$-field $\mathcal K$ and an integer 
$n\ge 1$, let\label{Pi} $\Pi_{\mathcal K}^{(n)}(\mathcal G):=H^1_{\text{{\rm cont}}}(\mathcal G,\GL_n\mathcal K)$ 
denote the set of isomorphism classes of $n$-dimensional objects of $\Sm_{\mathcal K}(\mathcal G)$. 
\begin{proposition} \label{simpleH_simpleG} Let $G$ be a permutation group, $\wK$ be a smooth $G$-field, 
and $H$ be a profinite group of $G$-field automorphisms of $\wK$. Set $K:=\wK^H$ and $k:=\wK^G$. 
\begin{enumerate} \item \label{simple-objects_open-subgrps} Assume that $H$ acts faithfully on $\wK^G$, and any 
simple object of $\Sm_{\wK}(G)$ is isomorphic to $\wK$. Then any simple object of $\Sm_K(G)$ is isomorphic to $K$. 
\item \label{simple_induct} Assume that $k$ coincides with $K^G$. 
Then, for any absolutely irreducible $\rho\in\Sm_k(H)$, the object 
$V_{\rho}:=\Hom_{k[H]}(\rho,\wK)$ of $\Sm_K(G)$ is simple$;\ \dim_KV_{\rho}=\dim_k\rho$. 

If $\#(H/U)$ is invertible in $K$ for any open subgroup $U\subset H$ then {\rm (i)} $\wK$ is 
a direct sum of subobjects of type $V_{\rho}$ for irreducible $\rho\in\Sm_k(H);$ 
{\rm (ii)} $V_{\rho}$ and $V_{\rho'}$ are dual in $\Sm_K(G)$ if and only if $\rho'\cong\rho^{\vee}$. 
\item \label{Pic_of_a_fixed_field} The extension of coefficients\footnote{so that $G$ acts 
trivially on the representatives of 
$\Pi_k^{(n)}(H)$.} $\Pi_k^{(n)}(H)\xrightarrow{\wK\otimes_k(-)}\Pi_{\wK}^{(n)}(G\times H)$ 
is injective and the natural sequence of pointed sets 
$\Pi_k^{(n)}(H)\xrightarrow{\wK\otimes_k(-)}\Pi_{\wK}^{(n)}(G\times H)\xrightarrow{\mathrm{Res}}
\Pi_{\wK}^{(n)}(G)^H$ is exact. 

In particular, there is an exact sequence of groups 
$0\to\mathrm{Pic}_k(H)\to\mathrm{Pic}_K(G)\to\mathrm{Pic}_{\wK}(G)^H$. \end{enumerate} \end{proposition} 
\begin{proof} \begin{enumerate} 
\item Any simple object of $\Sm_K(G)$ can be embedded into a simple object $V$ of $\Sm_{\wK}(G)$, i.e., 
into $\wK$. By the Galois theory, the following conditions on the $H$-action on $\wK^G$ are equivalent: 
\begin{enumerate} \item the $H$-action on $\wK^G$ is faithful; 
\item $[(\wK^G)^U:(\wK^G)^H]=[\wK^U:K]=\#(H/U)$ for any open subgroup $U\subseteq H;$ 
\item \label{sum-copies} $\wK^G\otimes_{(\wK^G)^H}K\xrightarrow{\times}\wK$ is an isomorphism. 
\end{enumerate} 
The condition (\ref{sum-copies}) means that $\wK^U$ is isomorphic to a direct sum 
of $[(\wK^G)^U:(\wK^G)^H]=\#H$ copies of the simple object $K$ of $\Sm_K(G)$ 
for any open subgroup $U\subseteq H$, but $\wK=\bigcup_U\wK^U$. 
\item For any $H$-invariant subfield $k$ of $\wK$ and any $\rho\in\Sm_k(H)$ with 
$\dim_k\rho<\infty$, one has $\rho^{\vee}\otimes_k\wK\cong\wK^{\oplus\dim_k\rho^{\vee}}$ in $\Sm_{\wK}(H)$ 
(by \cite[Satz 1]{Speiser}). As $V_{\rho}=(\rho^{\vee}\otimes_k\wK)^H$, one gets 
$\dim_KV_{\rho}=\dim_K(\wK^{\oplus\dim_k\rho^{\vee}})^H=\dim_k\rho$. 

Set $D_{\rho}:=\End_{k[H]}(\rho)$. Then $V_{\rho}$ is a $D_{\rho}\otimes_kK$-module. 
Clearly, the evaluation map $\rho\otimes_kV_{\rho}\to\wK$ factors through 
$\rho\otimes_{D_{\rho}}V_{\rho}\xrightarrow{\mathrm{ev}_{\rho}}\wK$. 
Let us show that the $K\langle G\times H\rangle$-morphism $\mathrm{ev}_{\rho}$ 
is injective if $K\otimes_k\rho$ is a simple object of $\Sm_K(H)$. 

By the normal basis theorem, there is an isomorphism $K[H/U]\xrightarrow{\sim}\wK^U$ in $\Sm_K(H/U)$, 
where $U:=\ker\rho$. Replacing $\wK$ by $\wK^U$, and further by $K[H/U]$ in the definition of $V_{\rho}$, 
we get an isomorphism $K\otimes_k\rho^{\vee}\xrightarrow{\sim}V_{\rho}\subset\rho^{\vee}\otimes_kK[H/U]$, 
$\lambda\mapsto\widehat{\lambda}:=\sum_{h\in H/U}\lambda^h[h]$, so the evaluation map becomes 
$\xi=\sum_{i=1}^rv_i\otimes\widehat{\lambda}_i\mapsto\sum_{h\in H/U}\sum_{i=1}^r\lambda_i^h(v_i)[h]$, 
which means that $\xi$ is its kernel only if $\sum_{i=1}^r\lambda_i^h(v_i)=\sum_{i=1}^r\lambda_i(hv_i)=0$ 
for all $h\in H$. We may assume that $v_i\in K\otimes_k\rho$ are linearly independent over 
$K\otimes_kD_{\rho}$ and $\lambda_1\neq 0$. By the Jacobson density theorem, the natural ring 
homomorphism $k[H/U]\to\mathrm{End}_{D_{\rho}}(\rho)$ is surjective, so there exist $a_h\in K$ 
such that $\sum_{h\in H/U}a_h[h]$ annihilates $v_2,\dots,v_r$ and sends $v_1$ to $v$ such that 
$\lambda_1(v)\neq 0$. Then 
$\sum_{h\in H/U}a_h\sum_{i=1}^r\lambda_i(hv_i)=\sum_{i=1}^r\lambda_i((\sum_{h\in H/U}a_h[h])v_i)=\lambda_1(v)$, 
giving contradiction. 

If $k\subseteq K^G$ then $G$ acts on $V_{\rho}$, so $V_{\rho}$ becomes a smooth 
$D_{\rho}^{\mathrm{op}}\otimes_kK\langle G\rangle$-module. By Lemma~\ref{multiplicities}, 
this implies that if $\rho$ is absolutely irreducible then $V_{\rho}$ is a simple object of $\Sm_K(G)$. 

Assume that the {\sl characteristic of $K$ does not divide} the indices of open subgroups $U'$ of $H$, 
so any smooth representation of $H$ over $K$ is semisimple. Let $\Pi_{U'}$ be the set 
of isomorphism classes of $k$-linear irreducible representations of $H/U'$. Then the $H$-bimodule 
decomposition $k[H/U']=\bigoplus_{\bar{\rho}\in\Pi_{U'}}\rho\otimes_{D_{\rho}}\rho^{\vee}$ corresponds 
to a decomposition $\wK=\bigoplus_{\bar{\rho}\in\Pi_{U'}}W_{\rho}$ is into a direct sum 
of objects of $\Sm_K(G\times H)$, where $W_{\rho}:=\rho\otimes_{D_{\rho}}V_{\rho}$. 
As $\Hom_{K[H]}(W_{\rho}\otimes_KW_{\rho'},K)\neq 0$ if and only if $\rho'\cong\rho^{\vee}$ 
(i.e. $\rho$ and $\rho'$ are dual), $W_{\rho}$ and $W_{\rho'}$ are dual if and only if 
$\rho'\cong\rho^{\vee}$. In particular, $V_{\rho}$ and $V_{\rho'}$ are dual in $\Sm_K(G)$ 
if and only if $\rho'\cong\rho^{\vee}$. 
\item 
The extension of coefficients is the inflation $H^1_{\text{{\rm cont}}}(H,\GL_n\wK^G)\xrightarrow{\mathrm{Inf}}
H^1_{\text{{\rm cont}}}(G\times H,\GL_n\wK)$ which is injective, while the above natural sequence 
of pointed sets becomes the inflation-restriction sequence $H^1_{\text{{\rm cont}}}(H,\GL_n\wK^G)
\xrightarrow{\mathrm{Inf}}H^1_{\text{{\rm cont}}}(G\times H,\GL_n\wK)\xrightarrow{\mathrm{Res}}
H^1_{\text{{\rm cont}}}(G,\GL_n\wK)^H$ which is exact, see \cite[\S5.8]{CohGal}. \end{enumerate} \end{proof}

For any permutation group $\mathcal G$ and a smooth $\mathcal G$-field $\mathcal K$, 
let $\Sm_{\mathcal K}^{\mathrm{fd}}(\mathcal G)$ denote the category 
of smooth finite-dimensional $\mathcal K\langle\mathcal G\rangle$-modules. 
\begin{corollary} \label{cor-for-U} In setting of Proposition~\ref{simpleH_simpleG}, the functor 
$\Sm_k^{\mathrm{fd}}(H)\xrightarrow{\wK\otimes_k(-)}\Sm_{\wK}^{\mathrm{fd}}(G\times H)$ 
is fully faithful. If 
$\Pi_{\wK}^{(n)}(G)^H=\{\ast\}$ then $\Pi_k^{(n)}(H)\xrightarrow{\wK\otimes_k(-)}\Pi_{\wK}^{(n)}(G\times H)$ 
is bijective. If $\Pi_{\wK}^{(n)}(G)^H=\{\ast\}$ for all $n\ge 1$ then 
$\Sm_k^{\mathrm{fd}}(H)\xrightarrow{\wK\otimes_k(-)}\Sm_{\wK}^{\mathrm{fd}}(G\times H)$ is an equivalence 
of categories. \end{corollary} 
\begin{proof} The full faithfulness is evident: for any $V,V'\in\Sm_k^{\mathrm{fd}}(H)$ 
one has \begin{multline*} \Hom_{\wK\langle G\times H\rangle}(\wK\otimes_kV,\wK\otimes_kV')
=(\wK\otimes_kV^{\vee}\otimes_kV')^{G\times H}=(\wK^G\otimes_kV^{\vee}\otimes_kV')^H
=\Hom_{k\langle H\rangle}(V,V'). \end{multline*} The bijectivity follows from 
Proposition~\ref{simpleH_simpleG} (\ref{Pic_of_a_fixed_field}), which implies equivalence 
of the categories. \end{proof}

\begin{remark} Given a finite non-empty subset $S\subset\Psi$, there are the following 
options for the action on the field $\Kd^{\Sy_{\Psi|S}}=F_S(\mbox{cross-ratios of quadruples in }S)$ 
of a subgroup $H$ of $\Sy_S$: \begin{enumerate} \item it is faithful (then the suboptions are 
(a) $F\neq k$, (b) $\#S\ge 5$, (c) $H=1$, (d) $\#S=4$ and $H$ contains no even non-identical involution), 
\item it is trivial, but $H\neq 1$ (then $F=k$ and the suboptions are 
(a) $\#S\in\{2,3\}$, (b) $\#S=4$ and $H$ consists of even involutions), 
\item it is non-trivial, but not faithful (then $F=k$, $\#S=4$ and 
the natural map $H\to\Sy_3$, given by the $H$-action on the partitions of $S$ into pairs, is not injective).
\end{enumerate} \end{remark}

\begin{lemma} \label{simple-Kd-dim} Let $S\subset\Psi$ be a finite subset and $H$ be a subgroup 
of $\Sy_S$. Set $L:=\Kd^{\Sy_{\Psi|S}}$. \begin{enumerate} \item If $\#S\ge 3$ then 
$\phi:\Sm_L^{\mathrm{fd}}(H)\xrightarrow{\Kd\otimes_L(-)}\Sm_{\Kd}^{\mathrm{fd}}(H\times\Sy_{\Psi|S})$ 
is an equivalence of categories, identifying the simple objects of $\Sm_L(H)$ and those of 
$\Sm_{\Kd}(H\times\Sy_{\Psi|S})$.
\item Let $x,y,z\in\Psi$ be some pairwise distinct elements, $p$ be the characteristic of $k$, 
and $T:=\frac{y-z}{z-x}$. 

Then the isomorphism classes of simple objects of $\Sm_{\Kd}(\Sy_{\Psi,\{x,y\}})$ are presented by $\Kd$, 
$T^n\Kd\oplus T^{-n}\Kd\subset\Kc$ for all integer $n\ge 1$, and by $(x-y)\Kd$ if 
$F=k$ and $p\neq 2$. \end{enumerate} \end{lemma} 
\begin{proof} By Lemma~\ref{relations_between_K?}, $\Kd=F_{\Psi}(\xi_u~|~u\in\Psi\smallsetminus\{x,y,z\})$, 
where $\xi_w:=\frac{(w-x)(y-z)}{(w-y)(x-z)}=\frac{(x-y)^{-1}-(x-z)^{-1}}{(x-y)^{-1}-(x-w)^{-1}}$. 
Then $\Kd=F_{\Psi}\left(\frac{y'-u'}{y'-z'}~|~u\in\Psi\smallsetminus\{x,y,z\}\right)=
F_{\Psi}\left(\frac{u''}{z''}~|~u\in\Psi\smallsetminus\{x,y,z\}\right)$, where $v'=\frac{1}{v-x}$ 
for all $v\in\Psi\smallsetminus\{x\}$, and $v'':=\frac{x-v}{v-y}$ for all $v\in\Psi\smallsetminus\{x,y\}$, 
so the restriction of $\Kd$ \begin{itemize} \item to $\Sy_{\Psi|\{x\}}$ is of type $\Kc$, 
i.e. it is isomorphic to $F_{\Psi}(\frac{y-u}{y-z}~|~u\in\Psi\smallsetminus\{x,y,z\})$; 
\item to $\Sy_{\Psi|\{x,y\}}$ is isomorphic to $F_{\Psi}(u/z~|~u\in\Psi\smallsetminus\{x,y,z\})$ 
(and $(xy)\in\Sy_{\Psi}$ acts by $v''\mapsto(v'')^{-1}$); 
\item to any subgroup of $\Sy_{\Psi|\{x,y,z\}}$ is isomorphic to 
$F_{\Psi}(\Psi\smallsetminus\{x,y,z\})$. \end{itemize} 

In each of these cases, the simple objects of $\Sm_{\Kd}(\Sy_{\Psi|S})$ are invertible, namely isomorphic to 
$e^n\Kd\subset\Kb$ for $n\in\mathbb Z$, where $e^{-1}:=\frac{1}{x-y}-\frac{1}{x-z}=\frac{y-z}{(x-y)(x-z)}$ 
(clearly, $e^n\Kd$ is independent of $y$ and $z$): 
$\mathrm{Pic}_{\Kd}(\Sy_{\Psi|S})\cong\mathbb Z$ if $\#S\in\{1,2\}$; 
$\mathrm{Pic}_{\Kd}(\Sy_{\Psi|S})=0$ (i.e. $e^n\Kd\cong\Kd$) if $\#S\ge 3$. By Proposition~\ref{Satz1}, 
$H^0(H,-)\colon\Sm_{\Kd}(U)\xrightarrow{\sim}\Sm_{\Kd^H}(\Sy_{\Psi|S})$ is an equivalence of categories. 
Any object $W$ of $\Sm_{\Kd^H}(\Sy_{\Psi|S})$ can be embedded into the object $W':=\Kd\otimes_{\Kd^H}W$ 
of $\Sm_{\Kd}(\Sy_{\Psi|S})$, while $\dim_{\Kd^H}W=\dim_{\Kd}W'$. If $W$ is simple it can be embedded 
into a simple quotient $W''$ of $W'$. In particular, any simple object of $\Sm_{\Kd}(U)$ is 
finite-dimensional, if so are the objects of $\Sm_{\Kd}(\Sy_{\Psi|S})$. 

If $\#S\ge 3$ then, by Corollary \ref{cor-for-U}, $\phi$ is an equivalence of categories. In particular, 
any simple object of $\Sm_{\Kd}(H\times\Sy_{\Psi|S})$ comes from a simple object of $\Sm_L(H)$. 

Let ${}^{(xy)}\Kc$ denote the $\Kd\langle\Sy_{\Psi}\rangle$-module $\Kc$ with the usual 
$\Sy_{\Psi|\{x,y\}}$-action, but with the $(xy)$-linear $\Kd$-vector space structure. Then 
$A_n:=\Hom_{\Kd\langle\Sy_{\Psi|\{x,y\}}\rangle}(T^n\Kd,{}^{(xy)}\Kc)\cong F_{\{x,y\}}$, 
$F_{\{x,y\}}\ni\alpha\mapsto[f\cdot T^n\mapsto{}^{(xy)}f\alpha\cdot T^{-n}]$. 
As the restriction of any simple object of $\Sm_{\Kd}(\Sy_{\Psi,\{x,y\}})$ to 
$\Sm_{\Kd}(\Sy_{\Psi|\{x,y\}})$ is a direct sum of objects $T^{n_i}\Kd$, while any non-zero 
element of $A_{n_i}$ identifies $T^{n_i}\Kd$ with $T^{-n_i}\Kd$, we get the second claim, unless 
$n_i=0$. As the involutions in $A_0$ are $\pm 1$, if a one-dimensional object is not isomorphic to $\Kd$ 
then $(xy)$ acts on its generator by $-1$, so it is isomorphic to $(x-y)\Kd$. As $(x-y)(f(x)-f(y))$ 
is fixed by the whole group for any $f\in F$, $(x-y)\Kd\cong\Kd$ if $F\neq k$. \end{proof} 

\begin{lemma} \label{simple-Kd-4} Let $S=\{w,x,y,z\}$ be a set of order $4$, $H$ be a subgroup of 
$\Sy_S$, and $H'\subseteq H$ be the kernel of the $H$-action $H\to\Sy_3$ on the set of partitions of $S$ 
into pairs. Then $H=H'\rtimes H''$ for a subgroup $H''$ of the stabilizer of an element of $S$, unless 
$H$ is a cyclic group of order $4$. 

Set $L:=k(\xi)$ for a field $k$ of characteristic $p\ge 0$, where 
$\xi:=\frac{(w-x)(y-z)}{(w-y)(x-z)}$. 

Then the simple objects of $\Sm_L(H)$ are at most three-dimensional$;$ besides $L$, they are 

\begin{tabular}{c|c|c|c} \hline $H$ & {\rm constraints} & 
{\rm remaining simple objects of} $\Sm_L(H)$ & {\rm dimensions}\\ 
\hline {\rm arbitrary} & $p=2$ & $-$ & $-$\\ 
\hline $H'=1$ & $-$ & $-$ & $-$\\ 
\hline $H=H'$ & $p\neq 2$ & {\rm twist(s) of $L$ by the non-trivial character(s) of $H$}& 
$\underbrace{1,\dots,1}_{\mbox{{\tiny $\#H-1$ {\rm time(s)}}}}$\\ 
\hline & $1\neq -1\in k^{\times 2}$ & 
{\rm twists of $L$ by characters of $H$ of order 4}& $1$, $1$\\ 
\cline{2-4} $\cong\mathbb Z/4\mathbb Z$ & $-1\notin k^{\times 2}$ & {\rm a generator acts on a basis by} 
$\begin{pmatrix} 0&-1\\ 1&0 \end{pmatrix}$ & $2$\\ 
\hline $\langle(wx),(yz)\rangle$ & $p\neq 2$ & 
{\rm the kernel of} $L\langle H/H''\rangle\xrightarrow{\sum_ha_h[h]\mapsto\sum_ha_h}L$ & $1$\\ 
\hline {\rm 2-Sylow} & $p\neq 2$ &  $L(\chi)\ (${\rm any} $H\xrightarrow{\chi}k^{\times},\ \chi|_{H'}\neq 1);\ 
L\otimes_kW$ & $1;$ $2$\\ & & $(W\in\Sm_k(H)$ {\rm is standard 2-dimensional)} & \\ 
\hline $\mathfrak{A}_S,\ \mathfrak{S}_S$&$p\neq 2$&
{\rm the kernel of} $L\langle H/H''\rangle\xrightarrow{\sum_ha_h[h]\mapsto\sum_ha_h}L$& $3$\\ 
\hline \end{tabular} \end{lemma} 
\begin{proof} \begin{enumerate} \item To check that $H=H'\rtimes H''$, consider the options for $H'$. 
If $\#H'=1$ then either $\#H=2$ or $H$ has a normal Sylow 3-subgroup, in any case it is not transitive. 
If $\#H'=2$ then $H$ is a 2-subgroup: either of period 2 or (transitive) cyclic of order 4. 
If $\#H'=4$ then, for any element $s$ of $S$, $H''$ can be chosen to be an appropriate subgroup of 
the stabilizer of $s$. 
\item  Let $V\not\cong L$ be a simple object of $\Sm_L(H)$. 

If $H$ is cyclic of order 4 then it is generated by $g=\sigma g'\sigma^{-1}$ for some 
$\sigma\in\Sy_S$, where $g':=(wxyz)$. Set $\xi'={}^{\sigma}\xi$, so ${}^g\xi'=1-\xi'$. 
Then $V$ is a quotient of the kernel $W=L(1-g^2)\dotplus L(g-g^3)$ of the surjection 
$L\langle H\rangle\xrightarrow{1\mapsto(1,\xi')}L\oplus L$. 
As the lines $L(1-g^2)$ and $L(g-g^3)$ are interchanged by $g$, any invariant line in $W$ contains 
$(1-g^2)+c(g-g^3)$ for some $c=\beta\frac{{\xi'}^n+\dots}{{\xi'}^m+\dots}\in L$ (with $\beta\in k$ and 
$m,n\in\mathbb Z_{\ge 0}$) such that ${}^gcc=-1$. Therefore, $m=n$ and $\beta^2=-1$. 
Thus, $W$ splits if and only if $k$ contains a fourth primitive root of unity $i$: 
$W=L(1-ig-g^2+ig^3)\dotplus L(1+ig-g^2-ig^3)$. In characteristic 2, there is an exact sequence 
$0\to L(\sum_{h\in H}h)\to W\to L(1+g^2)\to 0$, so $W$ is an extension of $L$ by $L$. 

If $H=H'\rtimes H''$ then $H''$ acts faithfully on $L$, so the restriction of 
$V$ to $H''$ is trivial, and therefore, $V$ is a quotient of the induced module 
$L\langle H\rangle\otimes_{L\langle H''\rangle}L=L\langle H/H''\rangle$. 
As the latter admits a natural surjection $L\langle H/H''\rangle\xrightarrow{\varepsilon}L$, 
$\dim_LV\le\#(H/H'')-1\le 3$, while the equality may be attained only if $\#H'=4$. If 
$p\neq 2$, $\#H'=4$ and $H''$ contains a 3-cycle then $\ker\varepsilon$ is simple, since 
it restriction to $H'$ is a sum of three distinct irreducible representations, permuted by 
any 3-cycle. If $H=\langle(wx),(yz)\rangle$ then $\ker\varepsilon$ is one-dimensional, 
and thus, simple; obviously, $\ker\varepsilon\not\cong L$ if $p\neq 2$.

The object $V$ contains an $H'$-invariant one-dimensional subspace $V_0$. As $V$ is simple, 
$V=\sum_{g\in H/H'}g(V_0)$. 

If $V_0$ is a trivial representation of $H'$ (which is always the case if $p=2$) then $V$ 
is so as well, and therefore, $V$ is a semilinear representation of $H/H'$, which is trivial, 
since the $H/H'$-action on $L$ is faithful. 

In the remaining case of a Sylow 2-subgroup $H$, $L\langle H\rangle=L\otimes_kk[H]$, while 
$k[H]=W\oplus\bigoplus_{\chi}k(\chi)$ if $p\neq 2$, where $W$ is 
the standard two-dimensional representation of $H$, and $\chi$ runs over the characters of $H$. 
As $L(\chi)^{H'}=0$ if $\chi|_{H'}\neq 1$, $L(\chi)\not\cong L$ in that case. If $\chi|_{H'}=1$ 
then the $H$-action on $L(\chi)$ factors through a faithful action of $H/H'$, so 
$L(\chi)\cong L$. Finally, the module $L\otimes_kW$ is simple, since is restriction to $H'$ is 
a sum of two non-isomorphic modules, while $H$ interchanges them. \end{enumerate} \end{proof} 

\begin{remark} \label{PFD} For $n\le 0$ and $\mathcal L=\omega_{\mathbb Y}^n$, Lemma~\ref{Part_frac_decomp} 
gives a short exact sequence \[0\to\omega^n_{\mathbb Y,\eta}/\Gamma(\mathbb Y,\omega_{\mathbb Y}^n)\to
\bigoplus_O\Kd\langle\Sy_{\Psi}\rangle\otimes_{\Kd\langle\St_{x_O}\rangle}(\omega^n_{\mathbb Y,\eta}/
\omega^n_{\mathbb Y,x_O})\to\Gamma(\mathbb Y,\omega_{\mathbb Y}^{1-n})^{\vee}\to 0,\] 
where $O$ runs over the $\Sy_{\Psi}$-orbits on $\mathbb Y^1$ and $x_O$ is an element of $O$. 

Let $S$ be a finite non-empty subset of $\Psi$. It follows from the description of the restriction of 
$\Kd$ to $\Sy_{\Psi|S}$, given in the proof of Lemma~\ref{simple-Kd-dim} (and in notation there), that 
the indecomposable injectives of $\Sm_{\Kd}(\Sy_{\Psi|S})$ are $\mathrm{Pic}_{\Kd}(\Sy_{\Psi|S})$-twists 
of $\Kd\langle\binom{\Psi\smallsetminus S}{s}'\rangle$ for $s\ge 1$ and, respectively, of $\Kd$ if 
$\#S>1$, and of $\Kd\left[\frac{z-x}{z-y}\right]$ if $S=\{x\}$. \end{remark}

\begin{lemma} \label{Hom_from_Jm} Fix some pairwise distinct $x,y,z\in\Psi$ and set 
$e:=\frac{(x-y)(x-z)}{y-z}\in\Kb$. Then 
\begin{itemize} \item for each $m\in\mathbb Z$, $e^m\Kd\subset\Kb$ is independent 
of $y$ and $z$, so it is an object of $\Sm_{\Kd}(\Sy_{\Psi|\{x\}});$ 
\item for the infinite-dimensional object 
$\mathcal J_m:=\Kd\langle\Sy_{\Psi}\rangle\otimes_{\Kd\langle\Sy_{\Psi|\{x\}}\rangle}e^m\Kd$ of 
$\Sm_{\Kd}(\Sy_{\Psi})$ one has\footnote{By Lemma \ref{fixed_1-form}, 
$\omega_{\mathbb Y,\eta}^n\cong(x-y)^n\Kc$; by \ref{simple-X-Y}, 
$\bigoplus\limits_{s=0}^{\infty}\frac{x^s}{(x-y)^{s-n}}\Kc$ is an injective hull of $(x-y)^n\Kc$ 
in $\Sm_{\Kc}(\Sy_{\Psi})$.} 
\begin{equation} \label{Hom_from_ind} 
\Hom_{\Kd\langle\Sy_{\Psi}\rangle}\left(\mathcal J_m,\bigoplus_{s=0}^a\frac{x^s}{(x-y)^{s-n}}\Kc\right)=
\begin{cases}x^{n-m}F,&\mbox{{\rm if }}m\le n\le a+m,\\ 0,&\mbox{{\rm otherwise}};\end{cases}\end{equation} 
\item the objects $\Kc\otimes_{\Kd}\mathcal J_n$ and 
$\mathcal E_n:=(x-y)^n\Kc\otimes_{\Kc}\Kc\langle\Psi\rangle$ of $\Sm_{\Kc}(\Sy_{\Psi})$ 
are isomorphic, and, for any non-zero submodule $V$ in $\mathcal J_m$, one has 
\begin{equation}\label{Hom_mn} \Hom_{\Kd\langle\Sy_{\Psi}\rangle}\left(\mathcal J_m,(x-y)^n\Kc\right)
=\Hom_{\Kd\langle\Sy_{\Psi}\rangle}\left(V,\mathcal E_n\right)=\begin{cases} F,&\mbox{{\rm if }}n=m,\\ 
0,&\mbox{{\rm otherwise}}.\end{cases} \end{equation} \end{itemize} \end{lemma} 

\begin{proof} Using the adjunction \begin{equation}\label{J-adj}\Hom_{\Kd\langle\Sy_{\Psi}\rangle}
\left(\mathcal J_m,-\right)=\Hom_{\Kd\langle\Sy_{\Psi|\{x\}}\rangle}\left(e^m\Kd,-\right)
=\left(e^{-m}\Kd\otimes_{\Kd}(-)\right)^{\Sy_{\Psi|\{x\}}},\end{equation} 
we get (\ref{Hom_from_ind}) as a consequence of the following identifications: \begin{multline*} 
\Hom_{\Kd\langle\Sy_{\Psi}\rangle}\left(\mathcal J_m,\bigoplus_{s=0}^a\frac{x^s}{(x-y)^{s-n}}\Kc\right)
=\bigoplus_{s=0}^a\left(e^{-m}\Kd\otimes_{\Kd}\frac{x^s}{(x-y)^{s-n}}\Kc\right)^{\Sy_{\Psi|\{x\}}}\\
=\bigoplus_{s=0}^a\left(\frac{(y-z)^m}{(x-y)^m(x-z)^m}\frac{x^s}{(x-y)^{s-n}}\Kc
\right)^{\Sy_{\Psi|\{x\}}}\\ 
=\bigoplus_{s=0}^ax^s\left((x-y)^{n-m-s}\Kc\right)^{\Sy_{\Psi|\{x\}}}=\begin{cases}
x^{n-m}F_x,&\mbox{{\rm if }}0\le n-m\le a,\\ 0,&\mbox{{\rm otherwise}}.\end{cases}\end{multline*} 

The map $\Kc\otimes_{\Kd}\mathcal J_n\xrightarrow{h\otimes(f[\sigma]\otimes e^n)\mapsto
(x^{\sigma}-y^{\sigma})^nfh\otimes\frac{(x^{\sigma}-z^{\sigma})^n}{(y^{\sigma}-z^{\sigma})^n}[x^{\sigma}]}
\mathcal E_n$ is an isomorphism. 

The adjunction (in particular, (\ref{J-adj})) gives 
\begin{equation}\label{Hom_mnm}\Hom_{\Kc\langle\Sy_{\Psi}\rangle}
\left(\mathcal E_m,\mathcal E_n\right)\cong
\Hom_{\Kd\langle\Sy_{\Psi}\rangle}\left(\mathcal J_m,\mathcal E_n\right)=
\left(e^{-m}\Kd\otimes_{\Kd}\Kc\langle\Sy_{\Psi}\rangle\otimes_{\Kd\langle\Sy_{\Psi|\{x\}}
\rangle}e^n\Kd\right)^{\Sy_{\Psi|\{x\}}}.\end{equation}
Let $\alpha=e^{-m}\otimes\sum_{i=1}^sh_i[\sigma_i]\otimes e^n$, 
where $\sigma_i\in\Sy_{\Psi}$, $h_i\in\Kc^{\times}$, $\#\{\sigma_1x,\dots,\sigma_sx\}=s$, be a non-zero 
element of the rightmost group of (\ref{Hom_mnm}). As $\alpha$ is fixed by $\Sy_{\Psi|\{x\}}$), one has 
$s=1$ and $\sigma_1x=x$, so $\frac{h_1(y-z)^{m-n}}{(x-y)^{m-n}(x-z)^{m-n}}\in\Kb^{\Sy_{\Psi|\{x\}}}=F_x$. 
This means that $n=m$ and $h_1\in F_x$. Thus, (\ref{Hom_mnm}) is 0 for $m\neq n$, 
and it is a one-dimensional $F$-vector space for $m=n$. 

As there are no non-zero finite-dimensional submodules in $\mathcal J_n$, any non-zero morphism 
from a submodules of $\mathcal J_m$ to $\mathcal J_n$ is injective, thus inducing an injection 
to $\mathcal E_n$. By Proposition \ref{simple-X-Y}, $\mathcal E_n$ is injective, so this injection 
extends to a morphism (in fact, an injection) $\mathcal J_m\to\mathcal E_n$. 

As the quotient of $\mathcal J_m$ by any non-zero submodules is finite-dimensional, any 
infinite-dimensional subquotient of $\mathcal J_m$ is just a submodule. 

According to Theorem~\ref{spectrum-triple} (\ref{Kc_P01}), for any $n\in\mathbb Z$, 
the object $\mathcal E_n$ of $\Sm_{\Kc}(\Sy_{\Psi})$ is injective (presenting the point $P_{1,n}$), 
while $\mathcal E_m$ and $\mathcal E_n$ are not isomorphic for $m\neq n$. Therefore, 
any morphism $V\to\mathcal E_n$ extends to a morphism $\mathcal J_m\to\mathcal E_n$. \end{proof} 

\begin{lemma} \label{Jm-toE_c} Let $b,m,n\in\mathbb Z$, and $T:=\frac{y-z}{z-x}$. Then, in notation 
of Lemma~\ref{Hom_from_Jm}, the element $x^n(1-\frac{y}{x})^mT^b$ spans a finite-dimensional 
$\Kd\langle\Sy_{\Psi}\rangle$-submodule of $\bigoplus\limits_{s=0}^{\infty}x^s(x-y)^{n-s}\Kc$ 
if and only if $m\le n\le b\le -m$. There exists a non-zero morphism 
$\mathcal J_m\xrightarrow{\upsilon}\bigoplus\limits_{s=0}^{\infty}\frac{x^s}{(x-y)^{s-n}}\Kc$ with 
the image finite-dimensional over $\Kd$ if and only if $|n|\le -m$. 
\end{lemma} 
\begin{proof} For any $\sigma\in\Sy_{\Psi}$, one has\footnote{As above, $\xi_w:=
\frac{(w-x)(y-z)}{(w-y)(x-z)}=\frac{w-x}{y-w}T\in\Kd\cup\{\infty\}$, so $\xi_x:=0$ and $\xi_y:=\infty$.} 
$x^{\sigma}=\frac{x(y-x^{\sigma})+(x^{\sigma}-x)y}{y-x}=
\frac{x+\frac{x^{\sigma}-x}{y-x^{\sigma}}y}{1+\frac{x^{\sigma}-x}{y-x^{\sigma}}}=
\frac{Tx+\xi_{x^{\sigma}}y}{T+\xi_{x^{\sigma}}}=x-\frac{\xi_{x^{\sigma}}}{T+\xi_{x^{\sigma}}}(x-y)$, 
\begin{multline*}(x-y)^{\sigma}=
\frac{(\xi_{y^{\sigma}}-\xi_{x^{\sigma}})T}{(T+\xi_{x^{\sigma}})(T+\xi_{y^{\sigma}})}(x-y), 
T^{\sigma}=\frac{\xi_{y^{\sigma}}-\xi_{z^{\sigma}}}{\xi_{z^{\sigma}}-\xi_{x^{\sigma}}}
\frac{T+\xi_{x^{\sigma}}}{T+\xi_{y^{\sigma}}},\mbox{ hence }\\ 
(x^{n-m}(x-y)^mT^b)^{\sigma}=\frac{(\xi_{y^{\sigma}}-\xi_{x^{\sigma}})^m(\xi_{y^{\sigma}}-\xi_{z^{\sigma}})^b}
{(\xi_{z^{\sigma}}-\xi_{x^{\sigma}})^b}(x-y)^mT^m
\frac{(Tx+\xi_{x^{\sigma}}y)^{n-m}}{(T+\xi_{x^{\sigma}})^{n-b}(T+\xi_{y^{\sigma}})^{b+m}}.
\end{multline*} Thus, the $\Kd\langle\Sy_{\Psi}\rangle$-span of $x^n(1-\frac{y}{x})^mT^b$ is isomorphic 
to the $\Kd$-span of the rational functions $(X+\xi)^{n-m}(T+\xi)^{b-n}(T+\xi')^{-b-m}\in\Kd(T,X)$ 
for some $\xi,\xi'\in\Kd$, i.e., it is finite-dimensional over $\Kd$ if and only if 
these functions are polynomial, or equivalently, $n\ge m$ and $n\le b\le -m$. 

If such a morphism $\upsilon$ exists then it corresponds to a multiple of the generator $x^{n-m}$ of 
(\ref{Hom_from_ind}) of Lemma~\ref{Hom_from_Jm}, so its image is the $\Kd\langle\Sy_{\Psi}\rangle$-span 
of $x^{n-m}e^m=\frac{x^{n-m}(y-x)^m}{T^m}$, corresponding to the case of $b=-m$, and therefore, 
it is finite-dimensional over $\Kd$ if and only if $|n|\le -m$ (and then its dimension is equal to 
$d_p(m,n):=\#\{s\ |\ \exists i\in\mathbb Z:\ p\nmid\binom{-n-m}{i}\binom{n-m}{s-i}\}=d_p(m,-n)\le 1-2m$). 
\end{proof} 

\begin{theorem}[``Borel--Weil theorem'' for $\Kd$] \label{non-triv-Simple-Kd} Let $p$ 
be the characteristic of $k$. For each integer $n\le 0$, let $W_n$ be the socle 
of the object $\Gamma(\mathbb Y,\omega_{\mathbb Y}^n)$ of $\Sm_{\Kd}(\Sy_{\Psi})$. 
\begin{enumerate} \item \label{finite-Simple-Kd} \begin{itemize} \item The objects $W_n$ 
{\rm (i)} are simple, {\rm (ii)} are pairwise non-isomorphic, {\rm (iii)} are self-dual, 
{\rm (iv)} present all isomorphism classes of simple finite-dimensional objects of $\Sm_{\Kd}(\Sy_{\Psi})$. 
\item If $p=0$ or $p>-2n$ then $W_n=\Gamma(\mathbb Y,\omega_{\mathbb Y}^n)$, so $\dim_{\Kd}W_n=1-2n$. 
\item If $p>0$ then $\dim_{\Kd}W_n=\prod\limits_{t\ge 0}(c_t+1)$, where $-2n=\sum\limits_{t\ge 0}c_tp^t$ and 
$c_t\in\{0,1,\dots,p-1\}$. In particular, $\Gamma(\mathbb Y,\omega_{\mathbb Y}^n)$ is simple if and only if 
$n=1/2-(a-1/2)p^m$ for some integers $0<a\le p/2$ and $m\ge 0.\ ($So it is never simple if $m>0$ and $p=2.)$ 

Let $-2n'=\sum\limits_{t\ge 0}c_t'p^t\in\mathbb Z_{\ge 0}$ for some $0\le c_t'<p$. 
Then there are natural isomorphisms \begin{itemize} 
\item $W_n\otimes_{\Kd}W_{n'}\xrightarrow[\sim]{\eta\otimes\eta'\mapsto\eta\eta'}W_{n+n'}$ 
if $c_tc_t'=0$ for all $t\ge 0;$ 
\item $\Kd\otimes_{\Kd^p}W_n^p\xleftarrow[\sim]{f\otimes\eta^p\mapsfrom f\otimes\eta}
\Kd\otimes_{\Kd,\phi}W_n\xrightarrow[\sim]{f\otimes\eta\mapsto f\eta^p}W_{pn}$, where 
$\phi\colon\Kd\xrightarrow{f\mapsto f^p}\Kd$ is the absolute Frobenius endomorphism. 
\end{itemize} \end{itemize} 
\item \label{infinite-Simple-Kd} For each $m\in\mathbb Z$, let 
$\mathcal J_m\in\Sm_{\Kd}(\Sy_{\Psi})$ be as in Lemma~\ref{Hom_from_Jm}. 
\begin{itemize} \item The objects $\mathcal J_m^{\circ}:=\begin{cases}\mathcal J_m&\mbox{{\rm if} }m>0,\\
\mathrm{socle}(\mathcal J_m)&\mbox{{\rm if} }m\le 0\mbox{ {\rm and }}F=k\end{cases}$ are simple 
and pairwise distinct. 
\item No non-zero subquotient of $\mathcal J_m$ can be embedded into $\mathcal J_n$ if $m\neq n$. 
\item For any $m\le 0$, {\rm (i)} any simple quotient of $\mathcal J_m$ is isomorphic to $W_m;$ 
{\rm (ii)} if $F\neq k$ then $\mathcal J_m$ embeds into a product of copies of $W_m;$ 
{\rm (iii)} if $F=k$ then $\mathcal J_m^{\circ}$ is the common kernel of the morphisms 
$\mathcal J_m\to E(W_n)$ for $m\le n\le 0$, 
$\dim_{\Kd}(\mathcal J_m/\mathcal J_m^{\circ})=\dim_{\Kd}\Gamma(\mathbb Y,\omega_{\mathbb Y}^n)$. \end{itemize} 
\item If there exists another simple object of $\Sm_{\Kd}(\Sy_{\Psi})$ then $F=k$ 
and it can be embedded into $\Kd\langle\Sy_{\Psi}\rangle\otimes_{\Kd\langle U\rangle}V$ for an open subgroup 
$U=H\times\Sy_{\Psi|S}\subseteq\Sy_{\Psi,S}$ of finite index, where $S\subset\Psi$ is a subset of order $3$ 
or $4$, and $V$ is a simple object of $\Sm_{\Kd}(U)\ ($with $\dim_{\Kd}V\le 3)$. \end{enumerate} \end{theorem} 
\begin{proof} By Remark~\ref{various_remarks} (\ref{irred-contained-irred-semilin}), any simple object $W$ of 
$\Sm_{\Kd}(\Sy_{\Psi})$ can be embedded into a simple object of $\Sm_{\Kc}(\Sy_{\Psi})$, i.e. (by 
Proposition~\ref{simple-X-Y}) into $\omega_{\mathbb Y,\eta}^{\nu}$ for some $\nu\in\mathbb Z$. Fix such $\nu$. 

For each integer $N\ge 0$, denote by $\omega_{\mathbb Y}^{\nu}(\ast N)$ the subsheaf of the constant 
sheaf $\omega_{\mathbb Y,\eta}^{\nu}$ on $\mathbb Y$, whose sections do not have poles 
of order $>N$. Fix the minimal possible $N\ge 0$ such that $W$ is actually embedded into 
$\Gamma(\mathbb Y,\omega_{\mathbb Y}^{\nu}(\ast N))$. Clearly (cf. Remark~\ref{PFD}), 
$W$ is infinite-dimensional if $N>0$. 

As before (e.g., in Lemma~\ref{fixed_1-form}), fix some pairwise distinct $x,y,z\in\Psi$, 
and set $T:=\frac{y-z}{z-x}\in\Kc$. 

As $T^{\sigma}=\frac{\xi_{y^{\sigma}}-\xi_{z^{\sigma}}}{\xi_{z^{\sigma}}-\xi_{x^{\sigma}}}
\frac{T+\xi_{x^{\sigma}}}{T+\xi_{y^{\sigma}}}$ for any $\sigma\in\Sy_{\Psi}$, 
for any integer $i$ and $\nu$, one has 
\begin{equation}\label{action}\left(T^i\left(\frac{\mathrm{d}T}{T}\right)^{\nu}\right)^{\sigma}
=(\xi_{y^{\sigma}}-\xi_{x^{\sigma}})^{\nu}
\left(\frac{\xi_{y^{\sigma}}-\xi_{z^{\sigma}}}{\xi_{z^{\sigma}}-\xi_{x^{\sigma}}}\right)^i
(T+\xi_{x^{\sigma}})^{i-\nu}(T+\xi_{y^{\sigma}})^{-i-\nu}(\mathrm{d}T)^{\nu}. \end{equation} 

In particular, $\left(T^i\left(\frac{\mathrm{d}T}{T}\right)^{\nu}\right)^{\sigma}=
\xi_{z^{\sigma}}^{-i}\cdot T^i\left(\frac{\mathrm{d}T}{T}\right)^{\nu}$ for any $\sigma\in\Sy_{\Psi|\{x,y\}}$, 
so $V_i^{(\nu)}:=\Kd\cdot T^i\left(\frac{\mathrm{d}T}{T}\right)^{\nu}$ is a 
$\Kd\langle\Sy_{\Psi|\{x,y\}}\rangle$-submodule of 
$\Gamma(\mathbb Y,\omega_{\mathbb Y}^{\nu}(\ast N_i))\subset\omega_{\mathbb Y,\eta}^{\nu}$, 
where $N_i:=\max(0,|i|+\nu)\le N$; the 
isomorphism class of $V_i=V_i^{(\nu)}$ is independent of $\nu$; the modules $V_i$ are pairwise non-isomorphic. 

\begin{enumerate} \item For all subsets $I\subseteq\mathbb Z$, this implies that $\bigoplus_{i\in I}V_i$ 
are the only $\Kd\langle\Sy_{\Psi|\{x,y\}}\rangle$-submodules of $\bigoplus_{i\in\mathbb Z}V_i$. 

If $N=0$ then $W$ is embedded into $\Gamma(\mathbb Y,\omega_{\mathbb Y}^{\nu})$, so $\nu\le 0$. 
Then $\Gamma(\mathbb Y,\omega_{\mathbb Y}^{\nu})=\bigoplus_{i=\nu}^{-\nu}V_i$, and thus, 
$W=\bigoplus_{i\in I}V_i$ for a subset $I\subseteq\{\nu,\nu+1,\dots,-\nu\}$. 
As one can see from the formula (\ref{action}), $I$ contains $-\nu$. 

The same formula (\ref{action}) with $i=-\nu$ (and $\xi_{x^{\sigma}}\notin\{0,\infty\}$) 
shows that $I$ contains $s$ if and only if $\binom{-2\nu}{s-\nu}\neq 0$ in $k$. In particular, 
$\Gamma(\mathbb Y,\omega_{\mathbb Y}^{\nu})$ is simple if $p=0$. 

Let now $p>0$. Then $W=\bigoplus_{0\le s-\nu\preceq_p-2\nu}V_s$ is the socle of 
$\Gamma(\mathbb Y,\omega_{\mathbb Y}^{\nu})$.\footnote{Here, as in the proof of 
Proposition~\ref{structure_of_the_closed_pt}, $\succeq_p$ is the partial order on $\mathbb N$, 
defined by $n\succeq_pm$ if $p\nmid\binom{n}{m}$.} Then 
$\dim_{\Kd}W=\#\{s\ |\ \binom{-2\nu}{s-\nu}\neq 0\}=\prod\limits_{t\ge 0}(c_t+1)$, where 
$-2\nu=\sum\limits_{t\ge 0}c_tp^t$ and $c_t\in\{0,1,\dots,p-1\}$. 

For each $n\le 0$, the element $T^i\left(\frac{\mathrm{d}T}{T}\right)^n\otimes 
T^{-i}\left(\frac{\mathrm{d}T}{T}\right)^n\in W_n\otimes_{\Kd}W_n$ is fixed by the transposition 
of $z$ and an element of $\Psi\smallsetminus\{x,y,z\}$; the transposition of $x$ and $y$ fixes 
$T^i\left(\frac{\mathrm{d}T}{T}\right)^n\otimes T^{-i}\left(\frac{\mathrm{d}T}{T}\right)^n+
T^{-i}\left(\frac{\mathrm{d}T}{T}\right)^n\otimes T^i\left(\frac{\mathrm{d}T}{T}\right)^n$; 
the transposition of $x$ and $z$ sends $T$ to $-T-1$, so it fixes 
$(T\otimes 1-1\otimes T)^{-2n}=\sum_{s=0}^{-2n}\binom{-2n}{s}(-T)^s\otimes T^{-2n-s}$ 
and $\mathrm{d}T\otimes \mathrm{d}T$, and therefore, it fixes 
$\sum_{s=0}^{-2n}\binom{-2n}{s}(-T)^{s+n}\left(\frac{\mathrm{d}T}{T}\right)^n
\otimes T^{-n-s}\left(\frac{\mathrm{d}T}{T}\right)^n$. This means that the element 
$\sum_{s=n}^{-n}\binom{-2n}{s-n}(-T)^s\left(\frac{\mathrm{d}T}{T}\right)^n\otimes 
T^{-s}\left(\frac{\mathrm{d}T}{T}\right)^n\in W_n\otimes_{\Kd}W_n$ is fixed by the whole group $\Sy_{\Psi}$, 
and therefore, it gives rise to an isomorphism between $W_n$ and its dual. 

Let $\phi\colon\Kd\xrightarrow{f\mapsto f^p}\Kd$ denote the absolute Frobenius endomorphism. 
As the dimensions of $W_n$ and of $W_{pn}$ coincide, the natural injection 
$\Kd\otimes_{\Kd,\phi}W_n\xrightarrow{f\otimes\eta\mapsto f\eta^p}
\Gamma(\mathbb Y,\omega_{\mathbb Y}^{pn})$ factors through an isomorphism onto $W_{pn}$. 

The module $\Gamma(\mathbb Y,\omega_{\mathbb Y}^n)$ is simple (i.e. it coincides with $W$) 
if and only if $-2\nu\succeq_ps$ for all $0<s<-2\nu$, which means that $-2\nu=bp^m-1$ 
for some integers $0<b<p$ and $m\ge 0$. In particular, $\nu=0$ if $p=2$; $b$ is odd if 
$p\neq 2$, i.e. $\nu=1/2-(a-1/2)p^m$ for some integers $0<a\le p/2$ and $m\ge 0$. 

\item[(2--3)] If $N>0$ then $W$ contains an element with a pole of order $N$ at a closed point 
$v_0\in\mathbb Y^1$. 

Let $\St$ be the stabilizer of $v_0$ in $\Sy_{\Psi}$, and $O$ be the $\Sy_{\Psi}$-orbit of $v_0$. 
For each $v\in O$, set $V_v:=\omega_{\mathbb Y}^n(N\cdot v)_v/\omega_{\mathbb Y}^n((N-1)\cdot v)_v$. 
This is a smooth one-dimensional semilinear representation of $\St$ over the residue field of $v$. 
The map $W\xrightarrow{\eta\mapsto(\eta\pmod{\omega_{\mathbb Y}^n((N-1)v)_v})_{v\in O}}V_O
:=\bigoplus_{v\in O}V_v$ is an embedding. The natural $\Kd\langle\St\rangle$-action on $V_O$ 
identifies $V_O$ with $\Kd\langle\Sy_{\Psi}\rangle\otimes_{\Kd\langle\St\rangle}V_{v_0}$. 

As $V_{v_0}$ is finite-dimensional over $\Kd$, it admits a maximal $\Kd\langle\St\rangle$-submodule 
$V'$ such that $\Kd\langle\Sy_{\Psi}\rangle\otimes_{\Kd\langle\St\rangle}V'$ does not contain 
$W$. Then $W$ embeds into $\Kd\langle\Sy_{\Psi}\rangle\otimes_{\Kd\langle\St\rangle}V$ 
for a simple $\Kd\langle\St\rangle$-submodule $V$ of $V_{v_0}/V'$. 

\item[(2)] In notation of Lemma~\ref{Hom_from_Jm}, any simple object of $\Sm_{\Kd}(\Sy_{\Psi|\{x\}})$ 
is isomorphic to $e^m\Kd$ for some $m\in\mathbb Z$. By Lemma~\ref{level-1_quotients}, any non-zero 
submodule of $\mathcal J_m$ is of finite codimension. Thus, if $\mathcal J_m$ is not simple 
then it admits a non-zero morphism to $W_n$ for some $n\le 0$. The object $W_n$ is contained in 
$(x-y)^n\Kc$, so by the equality (\ref{Hom_from_ind}) of Lemma~\ref{Hom_from_Jm} with $a=0$, 
this means that $m=n\le 0$ and any simple quotient of $\mathcal J_m$ is isomorphic to $W_m$. 

Let now $m\le 0$. By Lemma~\ref{fixed_1-form}, the identical inclusion 
$e^m\Kd\hookrightarrow(x-y)^m\Kc$ corresponds to the $\Kd\langle\Sy_{\Psi|\{x\}}\rangle$-morphism 
$e^m\Kd\xrightarrow{(-)\times\varpi^m}\omega_{\mathbb Y,\eta}^m$, sending the generator 
$e^m=\frac{(x-y)^m(x-z)^m}{(y-z)^m}$ to $\left(\mathrm{d}T^{-1}\right)^m\in W_m$, 
thus inducing a surjective $\Kd\langle\Sy_{\Psi}\rangle$-morphism 
$\mathcal J_m\to W_m\subseteq\Gamma(\mathbb Y,\omega_{\mathbb Y}^m)$. 

Still assuming that $m\le 0$, if a cyclic submodule $M$ of $\mathcal J_m$ 
admits a proper non-zero submodule then there is a non-zero morphism 
$M\to W_n\subset\omega_{\mathbb Y,\eta}^n$ for some $n\le 0$. This morphism extends 
to a morphism $\mathcal J_m\xrightarrow{\upsilon}E_{\mathrm{c}}(\omega_{\mathbb Y,\eta}^n)\cong
\bigoplus\limits_{s=0}^{\infty}x^s(x-y)^{n-s}\Kc$, where $E_{\mathrm{c}}$ denotes an injective 
hull in the category $\Sm_{\Kc}(\Sy_{\Psi})$. By Lemma~\ref{Jm-toE_c}, the image of $\upsilon$ is 
finite-dimensional if and only if $m\le n\le -m$. This means that (i) any simple subquotient 
of $\mathcal J_m$ outside the socle is isomorphic to $W_n$ for some $m\le n\le 0$, (ii) the common 
kernel of the morphisms $\mathcal J_m\to E_{\mathrm{c}}(\omega_{\mathbb Y,\eta}^n)$ for 
all $m\le n\le 0$ is of finite codimension if $F=k$, and therefore, it is (non-zero and) simple. 
Its codimension is the dimension of the image of the map 
$\mathcal J_m\to\prod_{n=m}^0E_{\mathrm{c}}(\omega_{\mathbb Y,\eta}^n)$, i.e. of the linear span 
$\mathcal M$ of the elements $((e^m)^{\sigma},(xe^m)^{\sigma},(x^2e^m)^{\sigma},\dots,(x^{-m}e^m)^{\sigma})$ 
for all $\sigma\in\Sy_{\Psi}$, where \[(x^{n-m}e^m)^{\sigma}=\left(\frac{(\xi_{y^{\sigma}}-\xi_{x^{\sigma}})
(\xi_{z^{\sigma}}-\xi_{x^{\sigma}})}{\xi_{y^{\sigma}}-\xi_{z^\xi{\sigma}}}(y-x)T\right)^m
\frac{(Tx+\xi_{x^{\sigma}}y)^{n-m}}{(T+\xi_{x^{\sigma}})^{n+m}}.\] Then $\mathcal M$ is isomorphic to 
the $\Kd$-linear span of $(1-m)$-tuples of polynomials \[((T+\xi)^{-2m},(X+\xi y)(T+\xi)^{-1-2m},
(X+\xi y)^2(T+\xi)^{-2-2m},\dots,(X+\xi y)^{-m}(T+\xi)^{-m})\] 
for some $\xi\in\Kd$. As $(X+\xi y)^t(T+\xi)^{-t-2m}=
\sum_{s=0}^{-2m}\left(\sum_{i=0}^t\binom{t}{i}\binom{-t-2m}{s-i}X^{t-i}y^iT^{i-2m-s-t}\right)\xi^s$, 
the dimension of this linear span is $1-2m$, since for each $0\le s\le -2m$ there exist $0\le t\le -m$ 
and $0\le i\le t$ such that $\binom{t}{i}\binom{-t-2m}{s-i}=1$, namely $t=i=s$ if $s\le -m$, and 
$t=-2m-s$ and $i=0$ if $s\ge -m$. 

As $\Hom_{\Kd\langle\Sy_{\Psi}\rangle}(\mathcal J_m,\omega_{\mathbb Y,\eta}^m)\neq 0$ is an 
$F$-vector space, by Lemmas~\ref{level-1_quotients} and \ref{common_kernel}, if 
$F\neq k$ and $m\le 0$ then the common kernel of the morphisms $\mathcal J_m\to W_m$ is zero, so 
there are no simple submodules in $\mathcal J_m$.\footnote{As Proposition \ref{simplicity_of_Sm} 
shows, a simple subquotient of a product of copies of $W_m$ need not be isomorphic to $W_m$.}

By Lemma~\ref{Hom_from_Jm}, 
$\Kc\otimes_{\Kd}\mathcal J_n\cong\mathcal E_n:=(x-y)^m\Kc\otimes_{\Kc}\Kc\langle\Psi\rangle$, 
and therefore, by the equality (\ref{Hom_mn}) there, no non-zero subquotient of $\mathcal J_m$ 
can be embedded into $\mathcal J_n$ if $m\neq n$. 
In particular, the non-zero socles of distinct $\mathcal J_m$'s are pairwise non-isomorphic.

\item[(3)] By Lemma~\ref{stab_val}, if $O\neq\Theta$ then $\St=H\times\Sy_{\Psi|S}$ 
for a finite subset $S\subset\Psi$ of order $\ge 3$ and a subgroup $H\subseteq\Sy_S$. 

As the case of $O=\Theta$ is already treated above, we further assume that $\#S\ge 3$. 

By Lemma~\ref{simple-Kd-dim}, the simple objects of the category $\Sm_{\Kd}(\St)$ are trivial 
if and only if the simple $F_S\langle H\rangle$-modules are trivial. 
The latter condition holds in the following cases: \begin{enumerate} \item $\Kd^{\Sy_{\Psi|S}}\neq k$ 
and the $H$-action on $\Kd^{\Sy_{\Psi|S}}$ is faithful (which amounts to the following options: 
(i) $\#S\ge 5$, (ii) $\#S\ge 3$ and $F\neq k$, (iii) $\#S=4$ and $\#(H\cap\mathfrak{A}_S)$ is odd); 
\item $\#S=4$ and $p=2$ (by Lemma~\ref{simple-Kd-4}); 
\item $\#S=3$ and at least one of the following options holds: (i) $H=1$, (ii) $\#H=p$. \end{enumerate} 

In these cases we only have to find the simple subobjects of $\Kd\langle\Sy_{\Psi}/\St\rangle$. 

If $\Kd^{\Sy_{\Psi|S}}\neq k$ and the $\Sy_S$-action on $\Kd^{\Sy_{\Psi|S}}$ is faithful 
(which is the case for (a)(i) and (a)(ii)) then, by Lemma~\ref{embedding_into_prod}, 
$\Kd\langle\Sy_{\Psi}/\St\rangle$ embeds into a product of copies of $\Kd$, so there are 
no simple subobjects in $\Kd\langle\Sy_{\Psi}/\St\rangle$. 

Thus, in remaining cases a simple subobject may appear only in 
$\Kc\langle\Sy_{\Psi}/\Sy_{\Psi|S}\rangle\otimes_{L\langle H\rangle}V$, where $F=k$, $V$ is a simple 
object of $\Sm_L(H)$, $\#S\in\{3,4\}$, $L=k$ if $\#S=3$, 
$L$ and $V$ are as in Lemma~\ref{simple-Kd-4} if $\#S=4$. 
\end{enumerate} \end{proof}

\section{Some algebraically non-closed subfields of \texorpdfstring{$F_{\Psi}$}{}} 
\subsection{Torsion of the Picard group, and its trivialization} 
For an integer $n>0$, denote by ${}_n\mathrm{Pic}_K(G)$ the $n$-torsion subgroup in $\mathrm{Pic}_K(G)$. 

\begin{proposition} \label{Pic-torsion} Let $G$ be a permutation group, $K$ be a smooth $G$-field and $n>0$ 
be an integer. Set $k:=K^G$ and $\mu_n:=\{z\in K^{\times}~|~z^n=1\}$. Then there is a natural exact sequence 
\[H^1_{\text{{\rm cont}}}(G,\mu_n)\to{}_n\mathrm{Pic}_K(G)\xrightarrow{\beta}
(K^{\times}/K^{\times n})^G/k^{\times}\xrightarrow{\xi}H^2_{\text{{\rm cont}}}(G,\mu_n).\] 

Assume that $G$ admits no open subgroups of finite index. Then any invertible object $\mathcal L$ 
of $\Sm_K(G)$ of order $n$ is contained in the $G$-field $K(a^{1/n})$ 
for some $a\in(K^{\times}/K^{\times n})^G$. \end{proposition} 
\begin{proof} Let $\mathcal L$ be an object of $\Sm_K(G)$ with $\mathcal L^{\otimes_K^n}\cong K$. 
Choose $\pi\in\mathcal L$ and $\lambda\in(\mathcal L^{\otimes_K^n})^G$, both non-zero, and set 
$\beta([\mathcal L]):=[\pi^{\otimes n}/\lambda]$. In terms of 1-cocycles, $\beta((f_{\sigma}))=[g]$, 
where $g^{\sigma}/g=f_{\sigma}^n$ for all $\sigma\in G$. 

Define $\xi$ by $a\mapsto[(b_{\sigma}b_{\tau}^{\sigma}b_{\sigma\tau}^{-1})]$, where $\sigma a/a=b_{\sigma}^n$ 
for some 1-cochain $(b_{\sigma})$ with values in $K^{\times}$. Obviously, $\xi\beta=0$. 
If $\xi a=0$ then $(b_{\sigma}\zeta_{\sigma})$ is a 1-cocycle for some 1-cochain $(\zeta_{\sigma})$ 
with values in $\mu_n$, so $(b_{\sigma}\zeta_{\sigma})$ defines an element of $\mathrm{Pic}_K(G)$. 
Obviously, it is of order $n$. 

If $\beta((f_{\sigma}))=0$ then there exists $g\in K^{\times}$ such that $(g^n)^{\sigma}/g^n=f_{\sigma}^n$ 
for all $\sigma\in G$, and therefore, $(gf_{\sigma}/g^{\sigma})$ is a 1-cocycle with values in $\mu_n$.

Let $R=R_{\mathcal L,\lambda}$ be the quotient of the symmetric $K$-algebra of $\mathcal L$ by the ideal 
generated by $\lambda-1$. The natural map $\bigoplus_{i=0}^{n-1}\mathcal L^{\otimes_K^i}\to R$ 
is a $K\langle G\rangle$-module isomorphism, so $R$ is an $n$-dimensional $K$-vector space, 
and thus, it has at most $n$ maximal ideals. 

If $G$ admits no open subgroups of finite index then the $G$-action on the set of maximal ideals in $R$ is 
trivial, i.e. any maximal ideal is $G$-invariant. As $\mathcal L^{\otimes_K^i}$ are pairwise non-isomorphic for 
$0\le i<n$, any $G$-invariant $K$-vector subspace in $R$ is a direct sum of several $\mathcal L^{\otimes_K^i}$. 
As any non-empty direct sum of $\mathcal L^{\otimes_K^i}$'s contains invertible elements, we conclude that 0 is 
the only maximal ideal in $R$, i.e. $R$ is a field. Fix a non-zero $\pi\in\mathcal L$ and set 
$a:=\pi^n/\lambda\in K^{\times}$, so $R\cong K[\pi]/(\pi^n-a)$. As $\pi^{\sigma}/\pi\in K^{\times}$ for any 
$\sigma\in G$, one has $(\pi^{\sigma}/\pi)^n=a^{\sigma}/a\in K^{\times n}$, so $a$ is fixed by $G$ modulo 
$K^{\times n}$. \end{proof} 

Proposition~\ref{Pic-torsion} means that torsion in the Picard group is a particular case of finite smooth 
$G$-field extensions $L|K$ with $L^G=K^G$. 
\begin{example} \label{finite_smooth_G-ext} The following $G$-fields $K$ admit no non-trivial finite smooth 
$G$-field extension $L|K$ with $L^G=K^G$, so the Picard group is torsion free: 
(i) any $K$ such that it is a cogenerator of $\Sm_K(G)$, e.g. $\Ka$ with $\Xi=0$, 
(ii) any $K$ of characteristic 0 such that all simple objects of $\Sm_K(G)$ are isomorphic, 
(iii) $\Ka$ with $\Gamma=0$, (iv) $F_{\Psi}(x-y~|~x,y\in\Psi)$.

\begin{proof} (i) If $L$ is a finite smooth $G$-extension of $K$ then $L$ 
is a trivial $K$-semilinear representation of $G$, so $1=\dim_{K^G}L^G=[L:K]$. 

(ii) If $L$ is a non-trivial finite smooth $G$-extension of $K$ and $[L:K]$ is invertible in $K$ then 
$[L:K]^{-1}\mathrm{tr}_{L|K}$ splits the inclusion $K\hookrightarrow L$, so the socle of $L$ contains 
a subobject isomorphic to $K\oplus K$, contradicting $\dim_{K^G}L^G=1$. 

(iii) If $L$ is a finite smooth $\Sy_{\Psi}$-extension of $\Ka$ then $L\cong\bigoplus x^{\lambda}\Ka$ is an integral 
ring with a $\Ka$-semilinear action of $\Sy_{\Psi}$. As $L\otimes_{\Ka}L\xrightarrow{\times}L$ is injective 
on each $\Ka\cdot x^{\lambda}\otimes_{\Ka}\Ka\cdot x^{\lambda}\cong\Ka\cdot x^{2\lambda}$, taking $\lambda$ with 
maximal absolute value of $i$-th coordinate, we see that $L\cong\Ka^{[L:\Ka]}$. By the argument of (i), $L=\Ka$. 

(iv) Let $K=F_{\Psi}(x-y~|~x,y\in\Psi)$. By Theorem~\ref{spectrum-triple} (\ref{unip}), $K[x]$ is injective in 
$\Sm_K(G)$, so if $L$ is a finite smooth $\Sy_{\Psi}$-extension of $K$ then the inclusion $K\hookrightarrow K[x]$ 
extends to a morphism $\varphi\colon L\to K[x]$ in $\Sm_K(G)$. As $L^{\Sy_{\Psi}}=k$, $K$ is the only simple object 
and $\ker\varphi$ is of finite length, the map $\varphi$ is injective. We may thus assume that 
$L\subseteq K[x]^{<d}$ for a minimal $d\ge 1$. As $K[x]$ is injective, any commutative $K$-algebra structure 
$L\otimes_KL\to L\subset K[x]$ extends to a morphism $\xi\colon K[x]\otimes_KK[x]^{<d}\to K[x]$ 
in $\Sm_K(G)$. One has $K[x]\otimes_KK[x]^{<d}=
\bigoplus_{i=0}^{d-1}(K[x]\otimes_KK)(1\otimes x-x\otimes 1)^i\xrightarrow{\sim}K[x]^{\oplus d}$, 
$f\otimes g=\sum_{i=0}^{d-1}(fD^{(i)}g\otimes 1)(1\otimes x-x\otimes 1)^i\mapsto
(fg,fD^{(1)}g,fD^{(2)}g,\dots,fD^{(d-1)}g)$, 
in $\Sm_K(G)$. By Proposition~\ref{structure_of_the_closed_pt}, 
$\xi(f\otimes g)=\sum_{i=0}^{\infty}\sum_{j=0}^{d-1}a_{ij}D^{(i)}f\cdot D^{(j)}g$ 
for some $a_{ij}\in k$. As $\xi(1\otimes 1)\neq 0$, the element $a_{00}$ is non-zero, and therefore, 
$\deg\xi(f\otimes g)=\deg f+\deg g$. If $\deg f$ is maximal on $L$ then $\deg\xi(f\otimes f)=2\deg f\le\deg f$, 
i.e. $\deg f=0$, and thus, $L=K$. \end{proof} \end{example}

\subsubsection{Representations over fixed subfields of $k(\Psi)$} \label{reps-over-fixed-subfields} 
Proposition~\ref{simpleH_simpleG} (\ref{simple_induct}) provides the most straightforward source of 
finite-dimensional irreducible semilinear representations. 

In the following examples, $\wK=k(\Psi)$, $G=\Sy_{\Psi}$ and $H$ is a subgroup of $\mathrm{PGL}_{2,k}$ 
acting `diagonally' on $k(\Psi)$. If $\#H$ is invertible in $k$ then 
such $H$ can only be cyclic, dihedral, or isomorphic to one of $\mathfrak{A}_4$, $\mathfrak{S}_4$ 
and $\mathfrak{A}_5$. Except for the dihedral groups, there is only one conjugacy class for each of these 
groups (with no exceptions if $k^{\times}=k^{\times 2}$), see \cite[Theorem 4.2]{Beauville}. If $k$ 
contains a finite subfield $\mathbb F_q$ then $\mathrm{PGL}_2(\mathbb F_q)$ is one more option for $H$, 
while a complete list can be found in \cite{Faber}.

\begin{example}[Fixed fields of dihedral groups $\mu_{n,k}\rtimes\{\pm 1\}$ and of 
$\mathbb G_{\mathrm{m},k}\rtimes\{\pm 1\}$] \label{irred_2-dim_ex} Let $k$ be a field of characteristic 
$p$, $a\in k^{\times}$, and $\iota\colon k(X)\to k(X)$ be the $k$-field involution $X\mapsto a/X$. 

For each integer $n\ge 2$, let $\mu_{n,k}:=\mathrm{Spec}(k[A,A^{-1}]/(A^n-1))\subset
\mathbb G_{\mathrm{m},k}:=\mathrm{Spec}(k[A,A^{-1}])\subset\mathrm{PGL}_{2,k}$ 
be the standard $k$-groups acting on $k(\Psi)$ diagonally. Fix some distinct $x,y\in\Psi$. Let 
\begin{gather*} L:=k(\Psi)^{\mathbb G_{\mathrm{m},k}}=k(u/v~|~u,v\in\Psi)=k(u/x~|~u\in\Psi\smallsetminus\{x\}),\\ 
L_n:=k(\Psi)^{\mu_{n,k}}=L(x^n),\ 
K:=k(\Psi)^{\mathbb G_{\mathrm{m},k}\rtimes\langle\iota\rangle}=L^{\langle\iota\rangle},\mbox{ and }
K_n:=k(\Psi)^{\mu_{n,k}\rtimes\langle\iota\rangle}=L_n^{\langle\iota\rangle}\end{gather*} 
be the subfields of $k(\Psi)$ fixed by the corresponding $k$-groups 
under the action of \S\ref{actions}. 

As before, $\mathrm{Spec}_?$ is the Gabriel spectrum of $\Sm_?(\Sy_{\Psi})$. Then 
\begin{itemize} \item the set $\Pi_K^{(2)}(\Sy_{\Psi})$ (resp., 
$\Pi_{K_n}^{(2)}(\Sy_{\Psi})$)\footnote{As on p.\pageref{Pi}, $\Pi_{\mathcal K}^{(2)}(\mathcal G)$ denotes 
the set of isomorphism classes of smooth two-dimensional simple $\mathcal K\langle\mathcal G\rangle$-modules.} 
consists of the pairwise distinct classes of the injective objects $x^jL$ for all $j\ge 1$ 
(resp., $xL_n,\dots,x^{[(n-1)/2]}L_n$, so $\Pi_{K_2}^{(2)}(\Sy_{\Psi})=\varnothing$); 
\item the remaining simple objects of $\Sm_K(\Sy_{\Psi})$ (resp., of $\Sm_{K_n}(\Sy_{\Psi})$) 
are invertible; they are all injective if $p\neq 2$; 
\item $\mathrm{Pic}_K(\Sy_{\Psi})=0$ if $p=2$; $E(K)\cong L$, 
$E(K\langle\Psi\rangle)\cong L\langle\Psi\rangle$ and $L/K\xrightarrow[\mathrm{tr}_{L|K}]{\sim}K$ if $p=2$; 
\item $\mathrm{Pic}_K(\Sy_{\Psi})\cong\{\pm 1\}$ is generated by $(x/y-y/x)K$ if $p\neq 2$; the only 
non-trivial $\mathrm{Pic}_K(\Sy_{\Psi})$-orbit on $\mathrm{Spec}_K\smallsetminus\mathrm{Pic}_K(\Sy_{\Psi})$ 
is $\{[K\langle\Psi\rangle],[(x/y-y/x)K\otimes_KK\langle\Psi\rangle]\}$ if $p\neq 2$; 
\item $\mathrm{Pic}_{K_n}(\Sy_{\Psi})\cong\mathrm{Pic}_K(\Sy_{\Psi})\oplus\mathbb Z/(2,n)$; for any 
$\mathcal L\in\mathrm{Pic}_{K_n}(\Sy_{\Psi})$, one has $E(\mathcal L)/\mathcal L\cong\mathcal L$ if $p=2$; 
\item non-trivial $\mathrm{Pic}_{K_n}(\Sy_{\Psi})$-orbits on 
$\mathrm{Spec}_{K_n}\smallsetminus\mathrm{Pic}_{K_n}(\Sy_{\Psi})$ occur only if $n$ is even, 
in which case they are $\{[x^iL_n],[x^{n/2-i}L_n]\}$ for $1\le i<n/4$, while the subgroup 
$\mathrm{Pic}_K(\Sy_{\Psi})$ acts trivially; 
\item the closed points of $\mathrm{Spec}_K$ and of $\mathrm{Spec}_{K_n}$ are presented by 
injective hulls of the simple objects and, in the case of $\mathrm{Spec}_K$, by $E(K\langle\Psi\rangle)$ 
and, if $p\neq 2$, by $(x/y-y/x)K\otimes_KK\langle\Psi\rangle$; 
\item the non-closed points of $\mathrm{Spec}_K$ (resp., of $\mathrm{Spec}_{K_n}$) 
are presented by $K\langle\binom{\Psi}{s}\rangle$ for all integer $s>1$ (resp., 
by $K_n\langle\binom{\Psi}{s}\rangle$ for all integer $s>0$). \end{itemize} 
\begin{proof} As $\iota$ commutes with the $\Sy_{\Psi}$-action, $K_n$ is $\Sy_{\Psi}$-invariant. The 
involution $\iota$ induces isomorphisms $\iota_i\colon x^iL\xrightarrow{\sim}x^{-i}L$. The decomposition 
$k(\Psi)=\bigoplus_{i=0}^{n-1}x^iL_n$ into a direct sum of subobjects shows that $x^iL_n$ are injective, 
while Lemma~\ref{multiplicities} (with $\wK=k(\Psi)$) implies that $x^iL_n$ is irreducible if 
$i\not\equiv -i\pmod n$, and $x^iL_n$'s are pairwise non-isomorphic for $1\le i<n/2$. If $i=n-i$ then the 
$K_n$-eigenspaces $\mathcal L_{\pm}:=(x^{n/2}\pm a^{n/2}x^{-n/2})K_n$ of $\iota_i$ are $\Sy_{\Psi}$-invariant. 
If $p\neq 2$ then $L$ and $L_n=L\otimes_KK_n$ split into the sums of eigenspaces of $\iota$ as 
$L=K\oplus(x/y-y/x)K$ and $L_n=K_n\oplus(x^n-a^nx^{-n})K_n$. If $p=2$ then, by Lemma~\ref{multiplicities} 
(with $\wK=L_n$), the sequence $0\to K\to L\xrightarrow{\mathrm{tr}_{L|K}}K\to 0$ does not split, even 
under the functor $\Sm_K(\Sy_{\Psi})\xrightarrow{(-)\otimes_KK_n}\Sm_{K_n}(\Sy_{\Psi})$; clearly, 
$E(\mathcal L_+)\cong\mathcal L_+\otimes_{K_n}E(K_n)\cong x^{n/2}L_n$. 

The multiplication induces an isomorphism $(x^n-a^nx^{-n})K_n\otimes_{K_n}x^iL_n\xrightarrow{\sim}x^iL_n$ 
and an isomorphism $\mathcal L_+\otimes_{K_n}x^{n/4}L_n\xrightarrow[\sim]{\iota_{3n/4}}x^{n/4}L_n$ if $4|n$. 

The indecomposable injectives in $\Sm_K(\Sy_{\Psi})$ (resp., in $\Sm_{K_n}(\Sy_{\Psi})$) are direct summands 
of the indecomposable injectives in $\Sm_L(\Sy_{\Psi})$ (resp., in $\Sm_{k(\Psi)}(\Sy_{\Psi})$), i.e. direct 
summands of $x^iL$ for all $i\in\mathbb Z$, $L\langle\binom{\Psi}{s}\rangle$ for all $s>0$ (resp., of 
$k(\Psi)\langle\binom{\Psi}{s}\rangle$ for all $s\ge 0$), so their list comes from the isomorphism of 
Lemma~\ref{fixed_basis}: $K\langle\binom{\Psi}{s}\rangle\oplus K\langle\binom{\Psi}{s}\rangle 
\xrightarrow{\sim}L\otimes_KK\langle\binom{\Psi}{s}\rangle$ for $s\ge 2$ (resp., 
$K_n\langle\binom{\Psi}{s}\rangle^{\oplus 2n}\xrightarrow{\sim}k(\Psi)\langle\binom{\Psi}{s}\rangle$ 
for $s\ge 1$). \end{proof} \end{example} 

\begin{example}[Fixed field of $\alpha_{p^n,k}$] \label{alpha_p-invar} Let $F|k$ 
be a regular field extension of characteristic $p>0$, 
$X\in F\smallsetminus k$ be an element such that $F|k(X)$ is {\sl algebraic} separable, and $n\ge 1$ be an 
integer. Then, for each $\lambda\in kF^{p^n}$, there is a unique $k[B]$-algebra endomorphism $\xi_{\lambda}$ 
of $F_{\Psi}[B]/(B^{p^n})$ such that (i) $\xi_{\lambda}$ is identical modulo $(B)$, (ii) 
$u:=X(u)\mapsto u+\lambda(u)B$ for all $u\in\Psi$. It can be considered as an $\Sy_{\Psi}$-equivariant 
and $(kF^{p^n})_{\Psi}$-linear action $F_{\Psi}\to F_{\Psi}[B]/(B^{p^n})$ on the {\sl field} $F_{\Psi}$ of 
the infinitesimal subgroup $\alpha_{p^n,k}:=\mathrm{Spec}(k[B]/(B^{p^n}))$ of $\mathbb G_{\mathrm{a},k}$. 

For each subset $\Lambda\subset kF^{p^n}$, denote by $K_{\Lambda}$ the subfield of $F_{\Psi}$ 
fixed by $\xi_{\lambda}$ for all $\lambda\in\Lambda$. Then 
\begin{itemize} \item $kF^{p^n}\xrightarrow{\lambda\mapsto\xi_{\lambda}}
\mathrm{Aut}_{k[B]\mbox{-alg}}(F_{\Psi}[B]/(B^{p^n}))$, is a group homomorphism, so any choice 
of $\Lambda$ determines an action $F_{\Psi}\to F_{\Psi}\otimes_k(k[B]/(B^{p^n}))^{\Lambda}$ on 
$F_{\Psi}$ of the cartesian power $\alpha_{p^n,k}^{\Lambda}$ of $\alpha_{p^n,k}$; 
\item $[F_{\Psi}:K_{\Lambda}]=p^{nd}$ if $\Lambda$ is a $d$-dimensional $k$-vector subspace 
for an integer $d\ge 1$ and, if moreover $\{\lambda_1,\dots,\lambda_d\}$ is a basis of $\Lambda$, 
\[K_{\Lambda}=kF_{\Psi}^{p^n}\left(\sum_{\sigma\in\Sy_{\{0,\dots,d\}}}\mathrm{sgn}(\sigma)
\lambda_1(u_{\sigma(1)})\cdots\lambda_d(u_{\sigma(d)})u_{\sigma(0)}~|~u_0,\dots,u_d\in\Psi\right);\] 
\item the only closed point of $\mathrm{Spec}_{K_{\lambda}}$ is presented by $E(K_{\lambda})\cong F_{\Psi}$, 
where $K_{\lambda}:=K_{k\cdot\lambda}$ and $\lambda\neq 0$; 
\item the non-closed points of $\mathrm{Spec}_{K_{\lambda}}$ are presented by 
$K_{\lambda}\langle\binom{\Psi}{s}\rangle$ for all integer $s>0$. \end{itemize} 
\begin{proof} The (co)associativity of $\xi_{\lambda}$ is clear: 
$u\mapsto u+\lambda(u)B\mapsto u+\lambda(u)(B+B')=u+\lambda(u)B'+\lambda(u+\lambda(u)B')B$. 
The commutativity:  
$u\mapsto u+\lambda(u)B\mapsto u+\lambda'(u)B+\lambda(u+\lambda'(u)B)B=u+(\lambda(u)+\lambda'(u))B$. 

As $\xi_{\lambda}$ is $kF^{p^n}$-linear, replacing $X$ by $X/\lambda$ we may further assume that 
$\lambda=1$. It suffices to describe the structure of $F_{\Psi}\langle\binom{\Psi}{s}\rangle$ 
as objects of $\Sm_{K_{\lambda}}(\Sy_{\Psi})$ for all integer $s\ge 0$. 
Fix some $x\in\Psi$. If a subobject $V$ of $F_{\Psi}$ contains a polynomial $P$ in $x$ over 
$K_{\lambda}$, say a monic one, then $V$ contains $P-gP$ for any $g\in\Sy_{\Psi}$. 
As $u-v\in K_{\lambda}$, if $P\neq 1$ then 
$\deg(P-gP)<\deg P$ and $gP\neq P$ for some $g\in\Sy_{\Psi}$, and therefore, $V$ contains a non-zero 
polynomial in $x$ of degree $<\deg P$ over $K_{\lambda}$. Thus, $K_{\lambda}$ is an essential 
subobject of $F_{\Psi}$, i.e. $F_{\Psi}$ is indecomposable, and finally, $E(K_{\lambda})\cong F_{\Psi}$. 

For any $s\ge 1$, $S\in\binom{\Psi}{s}$ and $j\ge 0$, the polynomial $\sum_{u\in S}u^{p^n+j}=
\sum_{u\in S}u^{p^n}(x+(u-x))^j=(\sum_{u\in S}u^{p^n})x^j+\sum_{u\in S}u^{p^n}(j(u-x)x^{j-1}+\cdots+(u-x)^j)$ 
is of degree $j$, and therefore, $F_{\Psi}=\bigoplus_{j=0}^{p^n-1}(\sum_{u\in S}u^{p^n+j})K_{\lambda}$, 
so $F_{\Psi}\langle\binom{\Psi}{s}\rangle=\bigoplus_{j=0}^{p^n-1}\left(\bigoplus_{S\in\binom{\Psi}{s}}
(\sum_{u\in S}u^j)K_{\lambda}[S]\right)$ is a direct sum of $\Sy_{\Psi}$-invariant $K_{\lambda}$-vector 
subspaces, while the map $K_{\lambda}\langle\binom{\Psi}{s}\rangle\to\bigoplus_{S\in\binom{\Psi}{s}}
(\sum_{u\in S}u^{p^n+j})K_{\lambda}\cdot[S]$, given by $[S]\mapsto(\sum_{u\in S}u^{p^n+j})[S]$, is 
an isomorphism in $\Sm_{K_{\lambda}}(\Sy_{\Psi})$. \end{proof} \end{example}

\begin{example} \label{irred_A_4_ex} Let $k$ be a field of characteristic $p\neq 2,3$, and $L$ 
be the fixed subfield in $k(\Psi)$ of the Klein four-group generated by the commuting $k$-field 
involutions $u\mapsto -u$ and $\iota\colon u\mapsto\frac{1}{u}$ for all $u\in\Psi$, i.e. $L:=
k\left(x^2+\frac{1}{x^2},\frac{u}{x}+\frac{x}{u},\frac{x^2-x^{-2}}{u/x-x/u}~|~u\in\Psi\smallsetminus\{x\}\right)$ 
for some $x\in\Psi$. Then {\rm (i)} the (third order) automorphism 
$\theta\colon u\mapsto\frac{u+\sqrt{-1}}{u-\sqrt{-1}}$ of the field $k(\sqrt{-1})(\Psi)$ over $k(\sqrt{-1})$ 
preserves $L$, 
{\rm (ii)} the subfield $K=K_{A4}$ of $L$ fixed by $\theta$ is $\Sy_{\Psi}$-invariant, 
{\rm (iii)} \begin{itemize} \item $(x-\frac{1}{x})L$ presents the unique isomorphism class 
of simple 3-dimensional objects of $\Sm_K(\Sy_{\Psi})$, 
\item if $\sqrt{-3}\in k$ then the remaining simple objects are invertible and $\#\mathrm{Pic}_K(\Sy_{\Psi})=3$, 
\item if $\sqrt{-3}\notin k$ then $\frac{x^{12}+15x^8+15x^4+1}{x^2(x^4-1)^2}K\oplus
\frac{x^8+14x^4+1}{(x^4-1)^2}K\subset L$ presents the unique remaining isomorphism class of non-invertible 
simple objects of $\Sm_K(\Sy_{\Psi})$ and $\mathrm{Pic}_K(\Sy_{\Psi})=0$. \end{itemize} 
\end{example} 
\begin{proof} One has $L(x^2)=k\left(x^2,\frac{u}{x}~|~u\in\Psi\right)$ and $[L(x^2):L]=2$ (since $L$ is fixed 
by $\iota$, while $x^2$ is not), $[L(x):L(x^2)]=2$ (since $L(x^2)$ is fixed by $u\mapsto -u$, while $x$ is not), 
$[L(x):L]=4$. This means that $L=k(\Psi)^{\langle -z,\iota\rangle}$, and therefore, $L$ contains 
$u^2+\frac{1}{u^2}$, $\frac{u}{v}+\frac{v}{u}$, $\frac{u^2-u^{-2}}{u/v-v/u}$ and $\frac{u\pm u^{-1}}{x\pm x^{-1}}$ 
for all $u,v\in\Psi$. In particular, (i) the subfield $L\subset k(\Psi)$ is $\Sy_{\Psi}$-invariant, 
(ii) $k(\Psi)=L\oplus\left(x-\frac{1}{x}\right)L\oplus\left(x+\frac{1}{x}\right)L\oplus
\left(x^2-\frac{1}{x^2}\right)L$ is a decomposition into a sum of $\Sy_{\Psi}$-invariant $L$-vector subspaces, 
(iii) $\theta$ preserves $L$ even if $\sqrt{-1}\notin k$ (since the $k(\sqrt{-1})$-linear field automorphisms 
$\sigma\theta\sigma$ and $\theta\iota$ of $k(\sqrt{-1})(\Psi)$ coincide, where $\sigma$ is the generator of 
$\mathrm{Gal}(L(\sqrt{-1})|L)$, while $\theta\iota$ and $\theta$ coincide on $L$). 

Assuming that $p\neq 3$, if $k$ does not contain a third root of unity $\zeta\neq\zeta^{-1}$ 
then we adjoin it, so $K(\zeta)$, 
$\left(\frac{(x^2-1)^2}{4x^2}-\frac{4\zeta x^2}{(x^2+1)^2}-\frac{(x^2+1)^2}{\zeta(x^2-1)^2}\right)K(\zeta)$, 
$\left(\frac{(x^2-1)^2}{4x^2}-\frac{4x^2}{\zeta(x^2+1)^2}-\frac{\zeta(x^2+1)^2}{(x^2-1)^2}\right)K(\zeta)$ 
are the ($\Sy_{\Psi}$-invariant) eigenspaces of $\theta$ in the 3-dimensional $K(\zeta)$-vector space $L(\zeta)$. 
By Lemma~\ref{multiplicities}, the isomorphisms $(x-x^{-1})L\xrightarrow[\widetilde{\phantom{--}}]{\sqrt{-1}\theta}
(x+x^{-1})L\xrightarrow[\widetilde{\phantom{--}}]{\theta}(x^2-x^{-2})L$, 
$x-x^{-1}\mapsto -\frac{4}{x+x^{-1}}\mapsto -2\frac{x+x^{-1}}{x-x^{-1}}$, imply that $(x-x^{-1})L$, 
$(x+x^{-1})L$, $(x^2-x^{-2})L$ are simple. 

[Alternatively. If $(x-x^{-1})L(\zeta)$ (resp. $(x+x^{-1})L(\zeta)$, resp. $(x^2-x^{-2})L(\zeta)$) 
is reducible then either itself or its dual contains an invertible subobject. But this is 
impossible, since by Proposition~\ref{simpleH_simpleG} (\ref{Pic_of_a_fixed_field}), 
$\#\mathrm{Pic}_{K(\zeta)}(\Sy_{\Psi})=3$, 
and the invertible subobjects in $k(\zeta)(\Psi)$ are multiplicity-free.] \end{proof}

\begin{example} \label{irred_A_5_ex} Let $k$ be a field of characteristic $p$, $H$ be a finite subgroup of the 
group (isomorphic to $\mathrm{PGL}_2k$) of $\Sy_{\Psi}$-field automorphisms of $k(\Psi)|k$, and $K:=k(\Psi)^H$. 

\begin{enumerate} \item Assume that $p$ does not divide $\#H$, and the irreducible representations 
of $H$ over $k$ are absolutely irreducible. It follows from Proposition~\ref{simpleH_simpleG} and 
the classification of irreducible representations of $\mathfrak{S}_4$ and $\mathfrak{A}_5$ that 
\begin{itemize} \item if $p\neq 2,3$, $k$ contains $\sqrt{-1}$ and $H=H_{S4}\supset H_{A4}$ is 
isomorphic to $\mathfrak{S}_4$ then there are 5 isomorphism classes of simple objects of $\Sm_K(\Sy_{\Psi})$: 
one of dimension 2 (fixed by $\mathrm{Pic}_K(\Sy_{\Psi})\cong\{\pm 1\}$), and two free 
$\mathrm{Pic}_K(\Sy_{\Psi})$-orbits of classes of dimensions 1 and 3; 
\item if $p\neq 2,3,5$, $k$ contains $\sqrt{-1}$ and $\sqrt{5}$, and 
$H=H_{A5}\supset H_{A4}$ is isomorphic to $\mathfrak{A}_5$ then there are 5 isomorphism classes of 
simple objects of $\Sm_K(\Sy_{\Psi})$: $K$, unique ones of dimensions 4 and 5, and two classes of dimension 3. 
\end{itemize} 
\item \label{PGL_2_p>0} Let $k$ contain a finite subfield $\mathbb F_q$, and $H$ be isomorphic 
$k$-representation $\rho:=k[\mathbb P^1(\mathbb F_q)]^{\circ}$ of $H$ is absolutely irreducible. 
By Proposition~\ref{simpleH_simpleG}, $\dim_KV_{\rho}=q$ and $V_{\rho}$ is irreducible. 
\end{enumerate} \end{example}

\vspace{5mm}

\noindent 
{\sl Acknowledgements.} {\small The study has been funded within the framework 
of the HSE University Basic Research Program. 
I thank Alexander Efimov, Dmitry Kaledin and Alexander Kuznetsov for very useful discussions. 
}


\vspace{5mm}

\begin{thebibliography}{} 
\bibitem{SGA4 I} M.Artin, A.Grothendieck, and J.-L.Verdier (in collaboration with N.Bourbaki,
P.Deligne and B.Saint-Donat), Th\'eorie des topos et cohomologie \'etale des sch\'emas.
S\'eminaire de G\'eom\'etrie Alg\'ebrique du Bois-Marie 1963--1964 (SGA4), Tome 1:
Th\'eorie des topos, Lecture Notes in Mathematics, Vol. 269, Springer-Verlag, Berlin, 1972. 
\bibitem{Beauville} A.Beauville, {\em Finite Subgroups of $\mathrm{PGL}_2(K)$,} 
Contemp. Math. \textbf{522}, 2010, 23--29, {\tt arXiv:0909.3942}. 
\bibitem{BucurDeleanu} I.Bucur, A.Deleanu, Introduction to the theory 
of categories and functors, Wiley, 1968. 
\bibitem{dSG} E. de Shalit, J.Guti\'errez, {\em On the structure of certain $\Gamma$-difference modules,} 
L'Ens. Math. \textbf{68}, 341--377 (2022), {\tt arXiv:2012.12353}. 
\bibitem{Faber} X.Faber, {\em Finite $p$-irregular subgroups of $\mathrm{PGL}_2(k)$,} 
La Matematica \textbf{2}, 479--522 (2023), {\tt arXiv:1112.1999}. 
\bibitem{HenniartVigneras} G.Henniart, M.-F.Vignéras, 
{\em Representations of $\mathrm{GL}_n(D)$ near the identity,} {\tt arXiv:2305.06581}. 
\bibitem{Howe} R.Howe, {\em The Fourier transform and germs of characters,} Math. Ann., \textbf{208} (1974), 305--322. 
\bibitem{Jacobson} N.Jacobson, Lectures in abstract algebra. Vol III: Theory of fields and 
Galois theory. D. Van Nostrand Co., Inc., Princeton, N.J.-Toronto, Ont.-London-New York 1964. 
\bibitem{Johnstone} P.T.Johnstone, Topos theory, Academic Press, London, 1977. London Math. Soc. Monographs, Vol. 10.
\bibitem{glatt} U.Jannsen, M.Rovinsky, {\em Smooth representations and sheaves,} 
Moscow Math. J. \textbf{10}, no. 1, (2010), 189--214. 
\bibitem{Kanda} Ryo Kanda, Classifying Serre subcategories via atom spectrum, Adv. in Math. 
\textbf{231} (2012), 1572--1588. 
\bibitem{D.Lascar} D.Lascar, {\em The group of automorphisms of the field of complex numbers leaving fixed the 
algebraic numbers is simple,} 110--114, in Model Theory of Groups and Automorphism Groups, Ed. D.M.Evans, 1997.
\bibitem{Lucas} E.Lucas, {\em Th\'eorie des fonctions num\'eriques simplement p\'eriodiques,} American 
Journal of Math. \textbf{1}, (1878), part 1: (2) 184--196; part 2: (3) 197--240; part 3: (4) 289--321. 
\bibitem{MacLane} S.MacLane, {\em The universality of formal power series 
fields}  Bull. Amer. Math. Soc. Volume 45, Number 12, Part 1 (1939), 888--890. 
\bibitem{NagpalSnowden} R.Nagpal, A.Snowden, {\em The semi-linear representation 
theory of the infinite symmetric group}, Represent. Theory \textbf{28} (2024), 533--551, {\tt arXiv:1909.08753}. 
\bibitem{Popescu} N. Popescu, Abelian Categories with Applications to Rings and Modules, 
Academic Press: New York and London, 1973, 1975 2nd edn. 
\bibitem{H90} M.Rovinsky, {\em Semilinear representations of symmetric groups and 
of automorphism groups of universal domains,} Selecta Math., New ser. \textbf{24} (2018), 2319--2349. 
\bibitem{Lueroth} M.Rovinsky, {\em L\"uroth's theorem for fields of rational functions in infinitely 
many permuted variables}, Sbornik Math. \textbf{216} (2025), {\tt arXiv:math/2408.04028}. 
\bibitem{Serre} J.-P.Serre, Groupes algébriques et corps de classes, Hermann, 1959. 
\bibitem{CohGal} J.-P.Serre, Cohomologie galoisienne, Springer, 1964.
\bibitem{CL} J.-P.Serre, Corps locaux, $3^{\text{\`{e}me}}$ \'{e}dition, Hermann, 1968. 
\bibitem{Speiser} A.Speiser, {\it Zahlentheoretische S\"{a}tze aus der Gruppentheorie,} 
Math. Zeit., \textbf{5} (1/2) (1919), 1--6. 
\bibitem{Stephenson} R.M.Stephenson Jr., {\it Minimal topological groups,} 
Math. Annalen, \textbf{192} (1971), 193--195. 
\end{thebibliography}
\end{document}